\documentclass[a4paper, 10.5pt]{article}
\usepackage{amsmath}
\usepackage{amssymb,esint}
\usepackage{amscd}
\usepackage{mathrsfs}
\usepackage{xspace}
\usepackage{fancyhdr}
\usepackage{xcolor}
\usepackage{verbatim}
\usepackage{cite}
\usepackage{appendix}
\usepackage{srcltx} 

\newtheorem{theorem}{Theorem}[section]

\newtheorem{corollary}[theorem]{Corollary}

\newtheorem{lemma}[theorem]{Lemma}
\newtheorem{notation}[theorem]{Notation}

\newtheorem{proposition}[theorem]{Proposition}
\newtheorem{remark}[theorem]{Remark}

\newtheorem{assumption}{Assumption}
\newenvironment{proof}[1][Proof]{\textbf{#1.} }{\hfill\rule{0.5em}{0.5em}}
{\catcode`\@=11\global\let\AddToReset=\@addtoreset
\AddToReset{equation}{section}

\AddToReset{theorem}{section}

\begin{document}
\title{Sharp estimates for screened Vlasov-Poisson system around Penrose-stable equilibria in $\mathbb{R}^d $, $ d\geq3$}
	\author{
	{\bf Lingjia Huang\thanks{E-mail address: ljhuang20@fudan.edu.cn, Fudan University, 220 Handan Road, Yangpu, Shanghai, 200433, China.}, Quoc-Hung Nguyen\thanks{E-mail address: qhnguyen@amss.ac.cn, Academy of Mathematics and Systems Science,
			Chinese Academy of Sciences,
			Beijing 100190, PR China.} ~and~Yiran Xu\thanks{E-mail address: yrxu20@fudan.edu.cn, Fudan University, 220 Handan Road, Yangpu, Shanghai, 200433, China. }}}
\date{}  
\maketitle
	\begin{abstract}
	In this paper, we study the asymptotic stability of  Penrose-stable equilibria among solutions of the screened Vlasov-Poisson system in $\mathbb{R}^d$ with $d\geq 3$ that was first established by  Bedrossian, Masmoudi, and Mouhot in \cite{JBedrossian2018} with smooth initial data. More precisely, we prove the sharp decay estimates for the density of the perturbed system, exactly like the free transport with only H\"older (i.e., $C^{a}$ for $0<a<1$) perturbed initial data. This improves the recent works in \cite{HanKwanD2021} by Han-Kwan, Nguyen, and Rousset for lower derivatives of the density and in  \cite{NguyenTT2020}  by T. Nguyen for higher derivatives with a logarithmic correction in time. Furthermore, we establish new estimates and cancellations of the kernel to the linearized problem to obtain this result. Moreover, we also prove this result for the Vlasov-Poisson system in which the electric field obeys a general nonlinear Poisson equation containing massless electrons/ions case. 
\end{abstract}
\section{Introduction}
This paper is devoted to the study of the asymptotic stability of equilibria for  the Vlasov-Poisson system of the form:
\begin{equation}\label{toymodel}
	\begin{cases}
		\partial_t f_i + v\cdot \nabla_x f_i +E\cdot\nabla_v f_i=0\\
		E=-\nabla_xU,    -\Delta U+U =\rho_i-1+A(U),~~ \rho_i(t,x)=\int_{\mathbb{R}^d}f_i(t,x,v)dv\\
	\end{cases}
\end{equation}
on the whole space $x\in\mathbb{R}^d$, $v\in\mathbb{R}^d$, $d\ge3$, where $f_i=f_i(t,x,v)\ge 0$ is the probability distribution of charged particles in plasma, $\rho_i(t,x)$ is the electric charge density, and  $E=E(t,x)$ is electric field and $A:\mathbb{R}\to \mathbb{R}$  is smooth and  satisfies  $A(r)=\circ(r)$ as $r\to 0$. In particular, massless electrons/ ions case if $A(r)= r+1-e^r$;  screened case if $A(r)= 0$; and unscreened case if $A(r)= r$. \vspace{0.15cm}\\
This system describes a hot, unconfined, electrostatic plasma, ions in $d$ dimensions of electrons on a uniform, static, background of ions with $d\geq 3$. This system has been extensively studied (\cite{arsenev,batt,EHorst1982,CBardos1985,FBouchut1991,DHanKwan2011,CBardos2018,KPfaffelmoser1992,JSchaeffer1991,EHorst1993,RTGlassey1996,LionsPL1991,Kwanlacobelli2017,GriffinLacobelli,GuoPausader,Alonescuwa2020,FlyOuPau}), focusing on the global existence, regularity results and longtime behavior of solutions. 
The first paper on the global existence of weak solutions with $A(r)=r$ is due to Arsen\'ev in \cite{arsenev}. Then  Batt in  \cite{batt} established global existence for spherically symmetric data, which was extended by Horst in \cite{EHorst1982} to global classical solvability with cylindrically symmetric data. After that in \cite{CBardos1985},  Bardos and Degond obtained global existence for small data with $d\geq 3$ by using the Lagrangian approach. Next,   in \cite{KPfaffelmoser1992} Pfaffelmoser proved the global existence of a smooth solution with
large (unrestricted size) data. Later, simpler proofs of the same were
published by Schaeffer \cite{JSchaeffer1991}, Horst \cite{EHorst1993}, and Lions and Pertharne \cite{LionsPL1991}. Moreover, the  system with massless electrons ( i.e $A(r)=r+1-e^r$) in $\mathbb{R}^d$ with $d\leq 3$ was also studied by Bouchut in \cite{FBouchut1991}, (see also  \cite{Kwanlacobelli2017,GriffinLacobelli}).
In recent years, in \cite{HJHwang2011} they indicated the sharp faster decay of derivatives. Later on,  in \cite{SHCHoi2011} they extended the results of \eqref{toymodel} to $d\geq2$ with better electric field decay in the case of screened interactions. Furthermore, the vector fields approach in \cite{JSmulevici2016} and the Fourier analysis approach in \cite{Xwang2018,Alonescuwa2020} were also developed in the last few years.\vspace{0.15cm}\\ 
In this paper, we study the stability and the long time behavior of solutions $f_i$ to \eqref{toymodel}  in the form
\begin{equation*}
f_i(t,x,v)=\mu(v)+f(t,x,v),
\end{equation*}
where $\mu(v)$ is a stable equilibrium with $\int_{\mathbb{R}^d}\mu(v)dv=1$ and $f_i(t=0,x,v)$ closes to $\mu(v)$.  So, $f$ solves the following perturbed system
\begin{equation}\label{eq2}
	\begin{cases}
		\partial_t f + v\cdot \nabla_x f +E\cdot\nabla_v \mu=-E\cdot\nabla_v f,\\
	E=-\nabla_xU,    -\Delta U+U =\rho+A(U),~~ \rho(t,x)=\int_{\mathbb{R}^d}f(t,x,v)dv,\\
		f|_{t=0}=f_0.
	\end{cases}
\end{equation}
Here are our assumptions in this paper. 
\begin{assumption}\label{mu bound}
$\mu$ satisfies the \textit{Penrose stability condition}:
	\begin{equation}\label{penrose}
		\inf_{\tau\in\mathbb{R},\xi\in\mathbb{R}^d}\left|1-\int_{0}^{+\infty}e^{-i\tau s}\frac{1}{1+|\xi|^2}i\xi\cdot\widehat{\nabla_v\mu}(s\xi)ds\right|\geq\text{\r{c}},
	\end{equation}
for some constant $\text{\r{c}}>0$, $\widehat{\nabla_v\mu}$ is the Fourier transform of $\nabla_v\mu$ in $\mathbb{R}^d$.
\end{assumption}
\begin{assumption}\label{mu norm} 
$\mu\in L^1$ satisfies 
	$$\|\langle \cdot\rangle^N\nabla_v\mu(\cdot)\|_{W^{2,\infty}}+\|\langle \cdot\rangle^{d+5}\nabla\mu(\cdot)\|_{W^{2d+7,1}}+\frac{1}{N-d}\leq M^*,$$
	for some $N>d$ and $M^*<\infty$.
\end{assumption}
\begin{assumption}\label{A assump} 
$A:\mathbb{R}\to \mathbb{R}$  is $C^2$ and  satisfies  
\begin{align*}
\sup_{|r|\leq 1}	\left(\left|\frac{A(r)}{r^2}\right|+\left|\frac{A'(r)}{r}\right|+\left|{A''(r)}\right|\right)\leq C_A,
\end{align*}
for some constant $C_A>0$.
\end{assumption}
Note that a particular example for the assumption \ref{A assump} is $A(r)=r+1-e^r$ corresponding to the Vlasov–Poisson system with electrons mass, see  \cite{FBouchut1991}.\\

It is well-known that the free transport equation $\partial_t f+v\nabla_xf=0$ exhibits phase mixing which the spatial density $\rho(t,x)=\int_{\mathbb{R}^{d}}f(t,x,v)dv$  decays in time. It was an observation of Landau in \cite{landau} that the linearized Vlasov-Poisson equations near homogeneous Penrose stable equilibria also have the spatial density decaying in time. Under the Penrose condition \eqref{penrose}, nonlinear Landau damping was proved on $\mathbb{T}^d\times\mathbb{R}^d$ in the pioneering work \cite{CMouhot2011} by Mouhot and Villani for data with Gevrey regularity (see also \cite{JBedrossian2016,NguyenTT2020} for refinements and simplifications). In \cite{Bedrossiantunis} Bedrossian showed that the results therein do not hold in finite regularity (see also \cite{NguyenTT2020b})). We note that related mechanisms in the fluid are the vorticity mixing by shear flows   \cite{JbedIhes2015,ALonescu2020,Alonescu2020acta,Naderweiren}. In the whole space  $\mathbb{R}^d\times\mathbb{R}^d$ with $d\geq 3$, it was established for the screened Vlasov-Poisson system in \cite{JBedrossian2018} by Bedrossian, Masmoudi, and Mouhot for data with finite Sobolev regularity. Their proof relies on the dispersive mechanism in Fourier space to control the plasma echo
resonance. However, a decay in time of their result is far from optimal, as the dispersion in the physical space was not taken into account. In the recent work \cite{HanKwanD2021}, Han-Kwan, Nguyen, and Rousset revisited the asymptotic stability of Penrose stable equilibria. They obtained the decay estimates for the density $\rho$ as follows:
\begin{equation}\label{x1}
	\|\rho(t)\|_{L^1}+(1+t)\|\nabla \rho(t)\|_{L^1}+(1+t)^d\|\rho(t)\|_{L^\infty}+(1+t)^{d+1}\|\nabla \rho(t)\|_{L^\infty}\lesssim\varepsilon_0 \log(t+2),~~\forall~t>0.
\end{equation}
This is achieved by a pointwise dispersive estimate directly on the resolvent kernel for the linearized system around Penrose stable equilibria $\mu$. Moreover, \eqref{x1} is optimal up to a logarithmic correction. Very recently, in \cite{AIonescu2022} Ionescu, Pausader, Wang, and Widmayer proved the first asymptotic stability result for the unscreened Vlasov-Poisson system in $\mathbb{R}^3$ around the Poisson equilibrium, see also in \cite{Pausader2021} for the case of a repulsive point charge. The unscreened case is open for the general equilibria. However, in \cite{HNRcmp,JBedrossianNa2020}  they studied the linearized unscreened Vlasov-Poisson equation around suitably stable homogeneous equilibria in  $\mathbb{R}^d_x \times \mathbb{R}^d_v$.
\vspace{0.3cm}\\
Our goal in this paper is to prove \eqref{x1} for  \eqref{eq2} \textit{without} $\log(t+2)$ (in RHS). We recall the Besov norm and Triebel-Lizorkin norm as
$$
\|g\|_{\dot{B}_{p,\infty}^{s}}:=\sup_{\alpha}\frac{\|\delta_\alpha g(x)\|_{L^{p}_x}}{|\alpha|^s},                             	\|g\|_{\dot{F}_{p,\infty}^{s}}:=\big\|\sup_{\alpha}\frac{|\delta_\alpha g(x)|}{|\alpha|^s}\big\|_{L^{p}_x}, $$
and for $\kappa\in (0,1)$ we   define a norm weighted in time:
$$
\|g\|_{ \kappa}=\sum_{p=1,\infty}\sup_{s\in [0,\infty)}\left(\langle s\rangle^{\frac{d(p-1)}{p}}\|  g(s)\|_{L^p}+\langle s\rangle^{\kappa+\frac{d(p-1)}{p}}\|  g(s)\|_{\dot{B}^\kappa_{p,\infty}}\right),~~~
\|g\|_{l, \kappa}=\sum_{j=0}^{l}\|\langle s\rangle^j\nabla^jg\|_{ \kappa}.$$
By Gagliardo-Nirenberg interpolation inequality, we know that $$
	\|g\|_{l, \kappa}\sim\|g\|_{ \kappa}+\|\langle s\rangle^l\nabla^l g\|_{ \kappa} \sim\sum_{p=1,\infty}\sup_{s\in [0,\infty)}\left(\langle s\rangle^{\frac{d(p-1)}{p}}\|  g(s)\|_{L^p}+\langle s\rangle^{l+\kappa+\frac{d(p-1)}{p}}\| \nabla^l g(s)\|_{\dot{B}^\kappa_{p,\infty}}\right).$$
Let us introduce some notations.
\begin{notation}
Let $g:\mathbb{R}^{d}_x\times \mathbb{R}^d_v\to  \mathbb{R}^m$. We define, for $(x,v)\in\mathbb{R}^{d}_x\times \mathbb{R}^d_v$ and $a\in (0,1)$
\begin{align*}
	&	\dot{\mathcal{D}}^a_1g(x,v)=\sup_{z\in\mathbb{R}^d}\frac{|g(x,v)-g(x-z,v)|}{|z|^a}, ~~	\dot{\mathcal{D}}^a_2g(x,v)=\sup_{z\in\mathbb{R}^d}\frac{|g(x,v)-g(x,v-z)|}{|z|^a},\\
	&\dot{\mathcal{D}}^ag=\dot{\mathcal{D}}^a_1g+\dot{\mathcal{D}}^a_2g,	
	\mathcal{D}^ag=|g|+\dot{\mathcal{D}}^ag,\mathcal{D}^a_ig=|g|+\dot{\mathcal{D}}^a_ig,~~i=1,2.
\end{align*}
\end{notation}
Our main result is as follows.
\begin{theorem}\label{Thmmain}
Let $a\in (\frac{\sqrt{5}-1}{2},1)$. There exist $C_0>0$,  $\widetilde{\varepsilon}\in (0,1)$ such that for $0<\varepsilon\leq\widetilde{\varepsilon}$ if 
\begin{equation}\label{f_0 lower condition}
\sum_{p=1,\infty}\sum_{i=0,1}\left\|\mathcal{D}^{a} (\nabla_{x,v}^if_0)\right\|_{L^1_{x}L^p_v\cap L^1_{v}L^p_x}\leq\frac{1}{C_0}\varepsilon,
\end{equation}then the problem \eqref{eq2} has a unique global solution $f$ with 
\begin{equation*}
	\|\rho\|_a+\varepsilon^{\frac{1}{3}}	\|U\|_a\leq \varepsilon ~~\text{and}~~	\|\rho\|_{1,a}+	\|U\|_{1,a}\lesssim_{\text{\r{c}},a,M^*}1.
\end{equation*}
In addition, with $\tilde{\varrho}_0(t,x)=\int_{\mathbb{R}^d}f_0(x-tv,v)dv$,  we have the estimate that
\begin{align}\label{op}
\|\rho-\tilde{\varrho}_0-G*_{(t,x)}\tilde{\varrho}_0\|_a+\varepsilon^{\frac{1}{3}}\|U\|_a\lesssim \varepsilon^{\frac{4}{3}},
\end{align} 
where  the kernel $G$ is defined in \eqref{definition of G about K}. 
Moreover, for $m\in \mathbb{N}^+$ and $b\in (0,1)$ if 	\begin{equation*}
\begin{split}
	&\sum_{j=2}^{m+1}	\sup_{|r|\leq 1}\left|{A^{(j)}(r)}\right|\leq C_{A,m},~~ \quad\sum_{j=0}^{m}	\sum_{p=1,\infty}\|\mathcal{D}^{a} (\nabla_{x,v}^jf_0)\|_{L^1_{x}L^p_v\cap L^1_{v}L^p_x}\leq \mathbf{c}, \\&\quad\|\langle v\rangle^{d+\gamma+5}\nabla\mu(\cdot)\|_{W^{2d+2\gamma+7,1}}+	\sum_{j=1}^{m+3}\sup_{v\in \mathbb{R}^d}\langle v\rangle^N|\nabla^j \mu(v)|\leq \mathbf{c}',
\end{split}
\end{equation*} for some $N>d$ , then $$\|\rho\|_{m,b}+	\|U\|_{m,b}\lesssim_{\text{\r{c}},a,M^*,\mathbf{c},\mathbf{c}'}1.$$
\end{theorem}
\begin{remark}If we have initial data $f_0$ with higher-order derivatives, we can get more regularity of the solution $\rho$. It is different from the condition in \cite{NguyenTTr2020}, which made smallness assumptions on higher derivatives of $f_0$.
\end{remark}
\begin{remark}
The crucial decay estimates of $E(t,x)$ depend on the fact that we work in dimension $d\geq3$. And the decay will be stronger when $d$ becomes higher. The decay estimates are insufficient in the case $d\leq2$ which can be seen in the proof of Proposition \ref{estimates about Y and W}, where we need to estimate $\sup_{\alpha}	\frac{	\|\delta_{\alpha}^v	\nabla_vY_{s,t}\|_{L^\infty_{x,v}}}{|\alpha|^{a-\delta}}$.  That is the reason why we only consider $d\geq3$ in our paper. We deal with the 2d case in \cite{HNX2}.
\end{remark}
\begin{remark}
While finishing our article, we learned that Toan Nguyen was working on this problem and had a similar result to ours. But the two works are independent.
\end{remark}
As \cite[Corollary 1.1]{HanKwanD2021}, thanks to Theorem \ref{Thmmain}, we also obtain the following scattering property for the solution to \eqref{eq2}. The proof is omitted as it is very analogous to \cite[Proof of Corollary 1.1]{HanKwanD2021}.
\begin{corollary}\label{cor1} With the same assumptions and notations as in Theorem \ref{Thmmain}, there is  a  function $f_\infty\in W^{1,\infty}$ given by 
	$$f_\infty(x,v)=f_0(x+Y_\infty(x,v),v+W_\infty(x,v))+\mu(v+W_\infty(x,v))-\mu(v),$$ such that 
	$$\|f(t,x+tv,v)-f_\infty(x,v)\|_{L^\infty_{x,v}}\lesssim_{a,M^*}\frac{\varepsilon}{\langle t\rangle^{d+a-1}},~~t\geq0,$$
and  $\|Y_\infty\|_{L^\infty_{x,v}}+\|W_\infty\|_{L^\infty_{x,v}}\lesssim\varepsilon$.
\end{corollary}

Let us discuss the main idea in this paper. First, we give the equivalence of $(\rho,U)$ and $(f,\rho)$ satisfying the system
\begin{align*}
	&	(\rho,U)=\left(\mathcal{F}_1(\rho,U),\mathcal{F}_2(\rho,U)\right),\\& ~\mathcal{F}_1(\rho,U)=G*_{(t,x)}((\mathcal{I}+\mathcal{R})	(\rho,U))+(\mathcal{I}+\mathcal{R})	(\rho,U),\\&
	\mathcal{F}_2(\rho,U)=(1-\Delta)^{-1}(\rho+A(U)), 
\end{align*}
where the kernel $G$ is defined in \eqref{definition of G about K} and $\mathcal{I},\mathcal{R}$ are defined in \eqref{I R}. This means, $(\rho,U)$ is a fixed point of $\mathcal{F}=(\mathcal{F}_1,\mathcal{F}_2)$. We establish the boundedness of the mapping $\mathcal{G}(\cdot)=G_{(t,x)}*\cdot$ under the norm $\|\cdot\|_a$. To do this, we need to estimate some pointwise decay estimates of the kernel $G(t,x)$, which gives a new proof to be different from the scaling method introduced by Han-Kwan, Nguyen, and Rousset in \cite{HanKwanD2021}.  It should be emphasized that the cancellation of $G(t,x)$, i.e. $\int_{\mathbb{R}^d}G(t,x)dx=0,~~\forall  t\geq 0,$  and the estimate of $\||x|^a|\delta_{\alpha} G(t,x)|\|_{L^p_x}$ play a very important role in proving the boundedness of the mapping $\mathcal{G}$. Using the decay estimates for the characteristics to prove maps $(\rho,U)\mapsto \mathcal{I}(\rho,U),~ (\rho,U)\mapsto \mathcal{R}(\rho,U)$ be  compressed mappings with a small contracting constant. 
Finally, we apply the contracting mapping principle to get a fixed point of  $\mathcal{F}$. \vspace{0.5cm}\\
The paper is organized as follows. In section 2, we establish the equivalence of $(\rho,U)$ and $f$, then we focus on the estimate of $(\rho,U)$ in the following sections. In section 3, we give the pointwise decay estimates of the kernel $G(t,x)$ and its derivatives. Then, we prove the bound:
$\|\mathcal{G}(\mathbf{f})\|_{\gamma,a}\leq M	\|\mathbf{f}\|_{\gamma,a},$
for some constant $M=M(\text{\r{c}},d,M_1)$. In section 4,
we establish the pointwise estimate of the characteristics $\left(X_{s,t}(x,v),V_{s,t}(x,v)\right)$ and their derivatives.  In section 5, we estimate the contribution of the initial data and explain that why we need the  norm $L^1_{x}L^1_v\cap L^1_{v}L^1_x\cap L^1_{x}L^\infty_v\cap L^1_{v}L^\infty_x$ of the initial data $f_0$.
In Section 6, we estimate the reaction term $\mathcal{R}(\rho,U)$ and its higher derivatives. To handle it, we shall introduce a general map
\begin{align*}
	\mathcal{T}[F,\eta](t,x)=-	\int_{0}^{t}\int_{\mathbb{R}^{d}}F(s,X_{s,t}(x,v))\eta(V_{s,t}(x,v))dvds+	\int_{0}^{t}\int_{\mathbb{R}^{d}}F(s,x-(t-s)v)\eta(v)dvds,
\end{align*}
and give the estimate  in Proposition \ref{esT}. In Section 7, we prove the main theorem. 
\vspace{0.4cm}\\
The following are some notations in our paper. 
\begin{notation}
	Throughout this paper,  $A\lesssim B$ means that there exists constant $ C $ only depending dimension $d$ such that $ A\leq CB$.  $A\lesssim_{M}B$ means that there exists constant $ C=C(d,M) $  such that $ A\leq CB$.  The same convention is adopted for $A\gtrsim B$ and  $A\sim B$.
\end{notation}
\begin{notation}
	In this paper,  $\widetilde{\cdot}$ is the "space-time" Fourier transform on $\mathbb{R}^{d+1}$ as follows
	\begin{align*}
		\widetilde{g}(\tau,\xi)=\int_{\mathbb{R}}\int_{\mathbb{R}^{d}}e^{-i\tau t}e^{-ix\cdot\xi}g(t,x)dxdt.
	\end{align*}
	And we define inverse Fourier transform 
	\begin{align*}
		\mathscr{F}^{-1}(h)(t,x)=\frac{1}{(2\pi)^{d+1}}\int_{\mathbb{R}}\int_{\mathbb{R}^{d}}e^{i\tau t}e^{ix\cdot\xi}h(\tau,\xi)d\xi d\tau.
	\end{align*}
\end{notation}

\begin{notation}
	For multi-index $\alpha=(\alpha_1,\alpha_2,...,\alpha_d)$, we denote	 
	$
	~	|\alpha|=\alpha_1+\alpha_2+...+\alpha_d
	.
	$
\end{notation}	
\begin{notation}
	For a vector $x\in \mathbb{R}^d$, and $k\in \mathbb{N}$, we define 
	$x^{\otimes k}=\underbrace{x\otimes x...\otimes x}_{k}$.
\end{notation}
\vspace{0.2cm}
\textbf{Acknowledgements:} 
	Q.H.N.  is supported by the Academy of Mathematics and Systems Science, Chinese Academy of Sciences startup fund, and the National Natural Science Foundation of China (No. 12050410257 and No. 12288201) and  the National Key R$\&$D Program of China under grant 2021YFA1000800. Q.H.N. also wants to thank Benoit Pausader for his stimulating comments and suggestion to consider the Vlasov-Poisson system with massless electrons.\\\\
	\textbf{Data availability}: Data will be made available on reasonable request.
\section{Equivlance of $(\rho,U)$ and $f$}
In this section, we will show that the system  \eqref{eq2} is equivalent to an equation in terms of $(\rho,U)$. Let $\left(X_{s,t}(x,v),V_{s,t}(x,v)\right)$  be the flow associated to the vector field $(v,E(t,x))$, i.e.  $\left(X_{s,t}(x,v),V_{s,t}(x,v)\right)$ satisfies the ODE system:
\begin{equation}\label{X and V}
	\begin{cases}
		\frac{d}{ds}X_{s,t}(x,v)=V_{s,t}(x,v),\qquad&X_{t,t}(x,v)=x,\\
		\frac{d}{ds}V_{s,t}(x,v)=E(s,X_{s,t}(x,v)),\qquad &V_{t,t}(x,v)=v,
	\end{cases}
\end{equation} 
for any $0\leq s\leq t$, $(x,v)\in \mathbb{R}^d\times\mathbb{R}^d$.
Then, the solution $f$ of the equation \eqref{eq2} is given by 
\begin{equation}\label{t1}
	f(t,x,v)=f_0(X_{0,t}(x,v),V_{0,t}(x,v))-\int_{0}^{t}E(s,X_{s,t}(x,v))\cdot\nabla_v\mu(V_{s,t}(x,v))ds,
\end{equation}
(see (1.3) in \cite{Hung2021} with $\mathbf{B}(t,x,v)=(v,E(t,x))$). 
We denote
\begin{align}\label{I R}
	\begin{split}
		&\mathcal{I}(\rho,U)(t,x)=	\int_{\mathbb{R}^{d}}f_0(X_{0,t}(x,v),V_{0,t}(x,v))dv,\\
		&\mathcal{R}(\rho,U)(t,x)=\int_{0}^{t}\int_{\mathbb{R}^{d}}\left(E(s,x-(t-s)v)\cdot\nabla_v\mu(v)-E(s,X_{s,t}(x,v))\cdot\nabla_v\mu(V_{s,t}(x,v))\right)dvds.
	\end{split}
\end{align}
Hence, $\rho(t,x)=\int_{\mathbb{R}^{d}}f(t,x,v)dv$ is the solution of the following equation:
\begin{equation}\label{rho and S in equation}
	\begin{split}
		\rho(t,x)&=\int_{0}^{t}\int_{\mathbb{R}^{d}}[\nabla_x
		(1-\Delta_x)^{-1}(\rho+A\circ U)](s,x-(t-s)v)\cdot\nabla_v\mu(v)dvds+(\mathcal{I}+\mathcal{R})(\rho,U)(t,x)\\
		&=\int_{0}^{t}\int_{\mathbb{R}^{d}}[\nabla_x(1-\Delta_x)^{-1}(\rho+A\circ U)](s,w)\cdot\nabla_v\mu(\frac{x-w}{t-s})
		\frac{1}{(t-s)^d} dwds+(\mathcal{I}+\mathcal{R})(\rho,U)(t,x)\\
		&=\int_{0}^{t}\int_{\mathbb{R}^{d}}
		(\rho(s,w)+A\circ U(s,w))
		\left(1-\frac{1}{(t-s)^2}\Delta_v\right)^{-1}
		\Delta_v\mu(\frac{x-w}{t-s})
		\frac{1}{(t-s)^{d+1}} dwds\\
		&\quad+(\mathcal{I}+\mathcal{R})(\rho,U)(t,x).
	\end{split}
\end{equation}
Define  $K(t,x)=\frac{1}{t^{d+1}}\left(\left(1-\frac{1}{t^2}\Delta_v\right)^{-1}\Delta_v\mu\right)\left(\frac{x}{t}\right)\mathbf{1}_{t>0}$, then $\rho$ can be expressed in the following way
\begin{equation*}
	\rho=\big(K*_{(t,x)}(\rho+A(U))\big)+	\mathcal{I}(\rho,U)+\mathcal{R}(\rho,U).
\end{equation*}
Next, taking Fourier transform in $(t,x)$, we have
\begin{align*}
	\widetilde{\rho}=\widetilde{K}\widetilde{\rho}+\widetilde{K}\widetilde{A(U)}+	\widetilde{\mathcal{I}(\rho,U)}+\widetilde{\mathcal{R}(\rho,U)}.
\end{align*}
Then we have 
\begin{equation}\label{equation rho}
\rho=G*_{(t,x)}(\mathcal{I}(\rho,U)+\mathcal{R}(\rho,U)+A(U))+\mathcal{I}(\rho,U)+\mathcal{R}(\rho,U),
\end{equation}
where  $G$ is the kernel satisfying:
\begin{equation}\label{definition of G about K}
	\widetilde{G}=\frac{\widetilde{K}}{1-\widetilde{K}},
\end{equation}
and 
\begin{equation*}
	\widetilde{K}(\tau,\xi)=\int_{0}^{+\infty}e^{-i\tau t}\frac{1}{1+|\xi|^2} i\xi\cdot\widehat{\nabla_v\mu}(t\xi)dt.
\end{equation*}
Thanks to the Penrose condition, one has 
$
|1-\widetilde{K}|\geq \text{\r c}>0,
$ which implies that 
$\widetilde{G}$ is well defined.\\
It is clear that if we have $f$, we can obtain the formula of $\rho$ and  $U$.  On the other hand, if we have $\rho$ and  $U$, we can get $E$, then we recover the characteristics $\left(X_{s,t}(x,v),V_{s,t}(x,v)\right)$ to finally get $f$ from \eqref{t1}.
 \begin{remark}
 	We will use $ \|\langle \cdot\rangle^N\nabla_v\mu(\cdot)\|_{W^{2,\infty}} $ to estimate the reaction term $\mathcal{R}(\rho,U)$ , and
 	$ \|\langle \cdot\rangle^{d+5}\nabla\mu(\cdot)\|_{W^{2d+7,1}} $ will be used to deal with the boundedness of the operator $\mathcal{G}$ .
 \end{remark}
We define the operator $\mathcal{F}$ as:
\begin{align}\label{mathcal F}
	\mathcal{F}(\rho,U)=\Big(\mathcal{F}_1(\rho,U),\mathcal{F}_2(\rho,U)\Big),
\end{align}
where
\begin{align*}
&	\mathcal{F}_1(\rho,U)=G*_{(t,x)}(\mathcal{I}(\rho,U)+\mathcal{R}(\rho,U)+A(U))+\mathcal{I}(\rho,U)+\mathcal{R}(\rho,U),\\&
 \mathcal{F}_2(\rho,U)=(1-\Delta)^{-1}(\rho+A(U)). 
\end{align*}
We need to prove that $\mathcal{F}$ has a fixed point i.e. $	\mathcal{F}(\rho,U)=(\rho,U)$.  \\
	We  define the norm
\begin{align*}
	\|(g_1,g_2)\|_{S^\varepsilon_\kappa}:=\|g_1\|_\kappa+\varepsilon^{\frac{1}{3}}\|g_2\|_\kappa,~~~
	\|(g_1,g_2)\|_{l, \kappa}:=\|g_1\|_{l, \kappa}+\|g_2\|_{l, \kappa}.
\end{align*}
The main Theorem \ref{Thmmain} is a consequence of the following theorems:
\begin{theorem} \label{Thm1}
Let $a\in (\frac{\sqrt{5}-1}{2},1)$. There exist constants $C_0>0$ and $\widetilde{\varepsilon}>0$ such that	for any $\varepsilon\in(0,\widetilde{\varepsilon}]$ we have 
\begin{align*}
	\|\mathcal{F}(\rho,U)	\|_{S^{\varepsilon}_a}\leq   \varepsilon,~~~~~~
	\|\mathcal{F}(\rho_1,U_1)-\mathcal{F}(\rho_2,U_2)\|_{S^{\varepsilon}_a}\leq\frac{1}{2}\|(\rho_1-\rho_2,U_1-U_2)\|_{S^{\varepsilon}_a},
\end{align*}
for any  $(\rho,U),(\rho_1,U_1),(\rho_2,U_2)\in\mathfrak{B}(\varepsilon):=\{\rho,U\in L^\infty((0,\infty);L^1\cap L^\infty(\mathbb{R}^d)):	\|(\rho,U)	\|_{S^{\varepsilon}_a}\leq \varepsilon\}$, provided that 
\begin{equation}\label{999}
	\sum_{i=0,1}\sum_{p=1,\infty}\left\|\mathcal{D}^{a} (\nabla^i_{x,v}f_0)\right\|_{L^1_{x}L^p_v\cap L^1_{v}L^p_x}\leq\frac{1}{C_0}\varepsilon.
\end{equation}
Under condition \eqref{999},  $\mathcal{F}$ has a unique fixed point $(\rho,U)$ in $\mathfrak{B}(\varepsilon)$ with $ \|(\rho,U)\|_{1,a}\lesssim_{\text{\r{c}},a,M^*}1. $
	Moreover, set  $\tilde{\varrho}_0(t,x)=\int_{\mathbb{R}^d}f_0(x-tv,v)dv$,  we have 
	\begin{align*}
	\|(\rho-\tilde{\varrho}_0-G*_{(t,x)}\tilde{\varrho}_0,U)\|_{S^{\varepsilon}_a}\lesssim_{\text{\r{c}},M^*,a} \varepsilon^{\frac{4}{3}}.
	\end{align*} 
\end{theorem}
\begin{theorem}\label{Thm2}
Let $\rho$ be in Theorem \ref{Thm1}.	Then  for $b\in(0,1)$ and $m\geq2$, we have $$\| (\rho,U)\|_{m,b}\lesssim_{\text{\r{c}},a,M^*,\mathbf{c},\mathbf{c}',C_{A,m}}1,$$    provided the initial data $ f_0$, the stationary state $\mu$ and function $A$ satisfying 
\begin{equation}\label{z17}
\begin{split}
	&\sum_{j=2}^{m+1}	\sup_{|r|\leq 1}\left|{A^{(j)}(r)}\right|\leq C_{A,m},~~~\sum_{j=0}^{m}	\sum_{p=1,\infty}\|\mathcal{D}^{a} (\nabla_{x,v}^jf_0)\|_{L^1_{x}L^p_v\cap L^1_{v}L^p_x}\leq \mathbf{c}',\\&\|\langle v\rangle^{d+\gamma+5}\nabla\mu(\cdot)\|_{W^{2d+2\gamma+7,1}}+
\langle v\rangle^N\sum_{j=0}^{m+3}|\nabla^j \mu(v)|\leq \mathbf{c}'.
\end{split}
\end{equation}
\end{theorem}

\section{Estimate of the kernel G and boundedness of operator $\mathcal{G}$}
In this section, we prove the boundedness of the operator $\mathcal{G}:\mathbf{f}\mapsto{G}*_{(t,x)} \mathbf{f}$ under the norm $\|\cdot\|_{\gamma,a}$, where $\gamma\geq0$. The main result is
$$\|\mathcal{G}(\mathbf{f})\|_{\gamma,a}\leq M	\|\mathbf{f}\|_{\gamma,a},$$
for some constant
$M=M(\text{\r{c}},d,M_{\gamma+1})$. \\
To do this, we need the following Lemmas:
\begin{lemma}
Define 
$$
Q_0(u,\tau,\xi)=\int_{0}^{+\infty}e^{-i\tau t}u(t\xi)dt.
$$
Then, 
\begin{align}\label{Q1}
&|Q_0(u,\tau,\xi)|\lesssim\frac{\sup_x\left(|u(x)|\langle x\rangle^{2}\right)}{|\xi|},\\&\label{Q2}
|Q_0(u,\tau,\xi)|\lesssim
\sum_{j=1}^{m-1}\frac{|\xi|^j}{|\tau|^{j+1}} |\nabla^j u(0)|
+\frac{|\xi|^{m-1}}{|\tau|^m}\sup_{x\in\mathbb{R}}\left(|\nabla^m u(x)|\langle x\rangle^{2}\right).
\end{align}

\end{lemma}
\begin{proof}
Firstly, we can get the result \eqref{Q1}
directly as follows,
\begin{align}\label{z2}
|Q_0(u,\tau,\xi)|	&\leq\int_{0}^{+\infty}|u(t\xi)|dt
\leq\sup_{z\in\mathbb{R}^d}\left(|u(z)|\langle z\rangle^{2}\right)
\int_{0}^{+\infty}\frac{1}{\langle t\xi\rangle^{2}}dt\sim\sup_{z\in\mathbb{R}^d}\left(|u(z)|\langle z\rangle^{2}\right)
\frac{1}{|\xi|}.
\end{align}
Moreover,  we integrate by parts m-times to obtain that  for any nonnegative integer $m$,
\begin{align*}
Q_0(u,\tau,\xi)=-\left(\sum_{j=0}^{m-1}\frac{1}{(i\tau)^{j+1}}\xi^{\otimes j}:\nabla^ju(0)\right)
+\frac{1}{(i\tau)^m}Q(\xi^{\otimes m}:\nabla^mu(t\xi),\tau,\xi).	
\end{align*}
Combining the above equality with \eqref{z2} to get the conclusion \eqref{Q2}.
\end{proof}
\begin{lemma}
For $k\in\mathbb{Z}^{+}$,  define
\begin{align}
Q_k(u,\tau,\xi)=\int_{0}^{+\infty} t^ke^{-i\tau t}
u(t\xi)dt.
\end{align}
Then, 
\begin{align}\label{Qk}
&	|Q_k(u,\tau,\xi)|\lesssim\frac{\sup_x\left(|u(x)|\langle x\rangle^{2+k}\right)}{|\xi|^{k+1}},\\&\label{Qk2}
|Q_k(u,\tau,\xi)|
\lesssim \frac{1}{|\tau|^{k+1}} \left(| u(0)|
+\sum_{l=0}^{k}\sup_{x}\left(|\nabla^{k-l+1} u(x)|\langle x\rangle^{2+k-l}\right)\right),
\\&\label{Qk3}
|Q_k(u,\tau,\xi)|
\lesssim\frac{1}{(|\tau|+|\xi|)^{k+1}} \left( \sup_x\left(|u(x)|\langle x\rangle^{2+k}\right)
+
\sum_{l=0}^{k}\sup_{x}\left(|\nabla^{k-l+1} u(x)|\langle x\rangle^{2+k-l}\right)\right).
\end{align}	
Moreover,	if $u(0)=0$, we have
\begin{align}\label{Q4}
&	|Q_k(u,\tau,\xi)|
\lesssim \frac{|\xi|}{|\tau|^{k+2}}\sum_{l=0}^{k+1}\sup_{x}\left(|\nabla^{k-l+2} u(x)|\langle x\rangle^{3+k-l}\right),
\\&\label{Qk4}
|Q_k(u,\tau,\xi)|
\lesssim\frac{|\xi|}{ (|\tau|+|\xi|)^{k+2}} \left( \sup_x\left(|u(x)|\langle x\rangle^{2+k}\right)
+
\sum_{l=0}^{k+1}\sup_{x}\left(|\nabla^{k-l+2} u(x)|\langle x\rangle^{3+k-l}\right)\right).
\end{align}	
\end{lemma}
\begin{proof}
First of all, $\forall j\in\{1,2,...,d\}$, we obtain
$$
\xi_j^kQ_k(u,\tau,\xi)=\int_{0}^{+\infty}(t\xi_j)^k e^{-i\tau t}
u(t\xi)dt
=Q_0(u_1,\tau,\xi),
$$
where 
$
u_1(\xi)=\xi_j^ku(\xi).
$
By \eqref{Q1}, one has, 
$$
|\xi_j^kQ_k(u,\tau,\xi)|\lesssim \frac{\sup_x\left(|u(x)|\langle x\rangle^{2+k}\right)}{|\xi|},$$
which implies \eqref{Qk}.\\
Note that
\begin{align*}
&\nabla^lu_1(0)=0,\quad\forall l\leq k-1,
~~~~\quad|\nabla^l u_1(0)|
\lesssim|\nabla^{l-k}u(0)|,
\quad\forall l\geq k.
\end{align*}
Then, we know from  \eqref{Q2} with $m=k+1$ that 
\begin{align*}
|\xi_j^kQ_k(u,\tau,\xi)|=Q_0(u_1,\tau,\xi)
\lesssim&\frac{|\xi|^{k}}{|\tau|^{k+1}} |\nabla^k u_1(0)|
+\frac{|\xi|^{k}}{|\tau|^{k+1}}\sup_{x}\left(|\nabla^{k+1} u_1(x)|\langle x\rangle^{2}\right)\\
\lesssim&\frac{|\xi|^k}{|\tau|^{k+1}} | u(0)|
+\frac{|\xi|^{k}}{|\tau|^{k+1}}\sum_{l=0}^{k}\sup_{x}\left(|\nabla^{k-l+1} u(x)|\langle x\rangle^{2+k-l}\right).
\end{align*}
Then, 
$$
|\xi|^k|Q_k(u,\tau,\xi)|
\lesssim \frac{|\xi|^k}{|\tau|^{k+1}} | u(0)|
+\frac{|\xi|^{k}}{|\tau|^{k+1}}\sum_{l=0}^{k}\sup_{x}\left(|\nabla^{k-l+1} u(x)|\langle x\rangle^{2+k-l}\right),
$$
which implies \eqref{Qk2}.\\
Then \eqref{Qk} and \eqref{Qk2} imply \eqref{Qk3} directly.\\
Moreover, if $u(0)=0$, thanks to \eqref{Q2} with $m=k+2$, we have
\begin{align*}
|\xi_j^kQ_k(u,\tau,\xi)|
\lesssim&\frac{|\xi|^{k+1}}{|\tau|^{k+2}} |\nabla^{k+1} u_1(0)|
+\frac{|\xi|^{k+1}}{|\tau|^{k+2}}\sup_{x}\left(|\nabla^{k+2} u_1(x)|\langle x\rangle^{2}\right)\\
\lesssim&\frac{|\xi|^{k+1}}{|\tau|^{k+2}} | \nabla u(0)|
+\frac{|\xi|^{k+1}}{|\tau|^{k+2}}\sum_{l=0}^{k+1}\sup_{x}\left(|\nabla^{k-l+2} u(x)|\langle x\rangle^{3+k-l}\right)\\
\lesssim&\frac{|\xi|^{k+1}}{|\tau|^{k+2}}\sum_{l=0}^{k+1}\sup_{x}\left(|\nabla^{k-l+2} u(x)|\langle x\rangle^{3+k-l}\right),
\end{align*}
which implies \eqref{Q4}.
Finally, \eqref{Qk} and \eqref{Q4} imply \eqref{Qk4}.
\end{proof}\vspace{0.2cm}\\
We define 
$$
\|F\|_{\mathcal{W}^{\beta}_\alpha}:=\sum_{j=0}^{\beta}\int_{\mathbb{R}^d} \langle x\rangle^\alpha |\nabla^{j} F(x)|dx,
$$		and use this norm in the following lemmas.
\begin{lemma}
For integer $\alpha\geq0$ and  $\beta\geq2$, we have following estimates:
\begin{align}\label{alpha and beta K estimate}
&\left|\nabla_\xi \widetilde{K}(\tau,\xi)	\right|\lesssim 
\frac{|\xi|}{\langle \xi\rangle^{2} (|\tau|+|\xi|)^{2}}\| \nabla \mu\|_{\mathcal{W}^{3}_{3}},\\
&	\left|\nabla_\xi^\beta \widetilde{K}(\tau,\xi)	\right|\lesssim 
\left(
\frac{1}{\langle\xi\rangle^{\beta}}
\frac{1}{ (|\tau|+|\xi|)^{2}}
+\frac{1}{\langle\xi\rangle^2}
\frac{1}{ (|\tau|+|\xi|)^{\beta}}	
\right)\| \nabla \mu\|_{\mathcal{W}^{2\beta+1}_{\beta+2}},\\
&	\label{alpha K estimate}
\left|\partial^\alpha_\tau \widetilde{K}(\tau,\xi)	\right|
\lesssim 
\frac{|\xi|^2}{\langle\xi\rangle^2 (|\tau|+|\xi|)^{\alpha+2}}
\|\nabla\mu\|_{\mathcal{W}_{\alpha+3}^{\alpha+2}}.
\end{align}	
\end{lemma}
\begin{proof}
Recall that
\begin{align}\label{K Fourier}
\widetilde{K}(\tau,\xi)=\int_{0}^{+\infty}e^{-i\tau t}\frac{1}{1+|\xi|^2} i\xi\cdot\widehat{\nabla_v\mu}(t\xi)dt.
\end{align}
Then for  $\alpha\geq0$, we have 
$$
\partial^\alpha_\tau \widetilde{K}(\tau,\xi)
=\sum_{j=1}^{d}\frac{(-i)^{\alpha}\xi_j}{\langle \xi\rangle^2}
\int_{0}^{+\infty}t^{\alpha}e^{-i\tau t}  P_j(t\xi)dt,
$$
where 
$$
P(\xi)=(P_1,P_2,...,P_d)(\xi)=\widehat{\nabla_v\mu}(\xi),\quad P(0)=0.
$$
Thanks to \eqref{Qk4},  we have
\begin{align*}
\left|	\int_{0}^{+\infty}t^{\alpha}e^{-i\tau t}  P_j(t\xi)dt\right|\lesssim&
\frac{|\xi|}{ (|\tau|+|\xi|)^{\alpha+2}}\left( \sup_x\left(|P(x)|\langle x\rangle^{2+\alpha}\right)
+
\sum_{l=0}^{\alpha+1}\sup_{x}\left(|\nabla ^{\alpha-l+2} P(x)|\langle x\rangle^{3+\alpha-l}\right)\right)\\
\lesssim&
\frac{|\xi|}{ (|\tau|+|\xi|)^{\alpha+2}}
\|\nabla\mu\|_{\mathcal{W}_{\alpha+3}^{\alpha+2}}.
\end{align*}
Thus we obtain that
\begin{align*}
\left|\partial^\alpha_\tau \widetilde{K}(\tau,\xi)	\right|
&\lesssim  \sum_{j=1}^{d}
\frac{|\xi|}{\langle \xi\rangle^{2}}\left|	\int_{0}^{+\infty}t^{\alpha}e^{-i\tau t}  P_j(t\xi)dt\right|
\lesssim	\frac{|\xi|^2}{\langle\xi\rangle^2 (|\tau|+|\xi|)^{\alpha+2}}
\|\nabla\mu\|_{\mathcal{W}_{\alpha+3}^{\alpha+2}}.
\end{align*}			
Thanks to \eqref{Qk3} with $k=\beta_1$, one has 
\begin{align*}
&\left|	\int_{0}^{+\infty}t^{\beta_1}e^{-i\tau t}  (\nabla^{\beta_1}P_j)(t\xi)dt\right|\\&\lesssim \frac{1}{ (|\tau|+|\xi|)^{\beta_1+1}} \left( \sup_x\left(|\nabla ^{\beta_1}P(x)|\langle x\rangle^{2+\beta_1}\right)
+
\sum_{l=0}^{\beta_1}\sup_{x}\left(|\nabla ^{2\beta_1-l+1} P(x)|\langle x\rangle^{2+\beta_1-l}\right)\right)\\
&\lesssim \frac{1}{ (|\tau|+|\xi|)^{\beta_1+1}}\|\nabla\mu\|_{\mathcal{W}^{2\beta_1+1}_{\beta_1+2}} .
\end{align*}
Therefore, 
$$
\left|\nabla _\xi \widetilde{K}(\tau,\xi)	\right|\lesssim\sum_{j=1}^{d}
\frac{|\xi|}{\langle \xi\rangle^2}\left|	\int_{0}^{+\infty}te^{-i\tau t}  (\nabla P_j)(t\xi)dt\right|
+\sum_{j=1}^{d}
\frac{1}{\langle \xi\rangle^2}\left|	\int_{0}^{+\infty}e^{-i\tau t}  P_j(t\xi)dt\right|
\lesssim\frac{|\xi|}{\langle \xi\rangle^{2} (|\tau|+|\xi|)^{2}}\| \nabla \mu\|_{\mathcal{W}^{3}_{3}}.
$$
For $\beta\geq2$, we have
\begin{align*}
\left|\nabla _\xi^\beta \widetilde{K}(\tau,\xi)	\right|\lesssim& \sum_{j=1}^{d}\sum_{\beta_1=0}^{\beta} \left|\nabla _\xi^{\beta-\beta_1} \left(\frac{\xi_j}{\langle \xi\rangle^2}\right)\right|\left|	\int_{0}^{+\infty}t^{\beta_1}e^{-i\tau t}  (\nabla ^{\beta_1}P_j)(t\xi)dt\right|\\
\lesssim & \sum_{j=1}^{d}\sum_{\beta_1=1}^{\beta-1} \frac{1}{\langle \xi\rangle^{\beta-\beta_1+1}}\left|	\int_{0}^{+\infty}t^{\beta_1}e^{-i\tau t}  (\nabla ^{\beta_1}P_j)(t\xi)dt\right|
\\
&+\sum_{j=1}^{d}
\frac{|\xi|}{\langle \xi\rangle^2}\left|	\int_{0}^{+\infty}t^{\beta}e^{-i\tau t}  (\nabla ^{\beta}P_j)(t\xi)dt\right|
+\sum_{j=1}^{d}
\frac{1}{\langle \xi\rangle^{\beta+1}}\left|	\int_{0}^{+\infty}e^{-i\tau t}  P_j(t\xi)dt\right|
\\
\lesssim&	\| \nabla \mu\|_{\mathcal{W}^{2\beta+1}_{\beta+2}}
\left(
\frac{1}{\langle\xi\rangle^{\beta}}
\frac{1}{ (|\tau|+|\xi|)^{2}}
+\frac{1}{\langle\xi\rangle^2}
\frac{1}{ (|\tau|+|\xi|)^{\beta}}	
\right)	.
\end{align*}
Thus, the proof is complete.
\end{proof}

\begin{lemma}For  any integer $j\geq1$, there holds
\begin{align}\label{K/1-K xi}			
\left|\nabla ^j\left(\frac{\widetilde{K}(\tau,\xi)}{1-\widetilde{K}(\tau,\xi)}\right)\right|
\lesssim&
M_j\left(
\frac{1}{\langle\xi\rangle^j (|\tau|+|\xi|)^2}
+
\frac{1}{\langle\xi\rangle^2 (|\tau|+|\xi|)^j}\right),\\\label{K/1-K tau}			
\left|\partial_{\tau}^j\left(\frac{\widetilde{K}(\tau,\xi)}{1-\widetilde{K}(\tau,\xi)}\right)\right|
\lesssim&
M_j
\frac{|\xi|^2}{\langle\xi\rangle^2 (|\tau|+|\xi|)^{j+2}},
\end{align}			
where	
$$
M_j=\| \nabla \mu\|_{\mathcal{W}_{j+2}^{2j+1}}\left(1+\| \nabla \mu\|_{\mathcal{W}_{j+2}^{2j+1}}\right)^{j-1}.
$$		
\end{lemma}
\begin{proof} Thanks to the Penrose condition, one has $|1-\widetilde{K}(\tau,\xi)|\geq \text{\r c}$.
Firstly, we have
$$
\left|\nabla \frac{\widetilde{K}(\tau,\xi)}{1-\widetilde{K}(\tau,\xi)}\right|=
\left|\nabla \left(\frac{1}{1-\widetilde{K}(\tau,\xi)}\right)\right|\lesssim_{\text{\r{c}}}
|\nabla \widetilde{K}(\tau,\xi)|
\overset{\eqref{alpha and beta K estimate}}{\lesssim_{\text{\r{c}}}}
\frac{|\xi|}{\langle \xi\rangle^{2} (|\tau|+|\xi|)^{2}}\| \nabla \mu\|_{\mathcal{W}^{3}_{3}}.
$$	
Note that we get the following lines for any $j\geq 1$, 
\begin{align*}
&\left|\nabla ^j \left(\frac{\widetilde{K}(\tau,\xi)}{1-\widetilde{K}(\tau,\xi)}\right)\right|=\left|\nabla ^j\left(\frac{1}{1-\widetilde{K}(\tau,\xi)}\right)\right|
\lesssim_{\text{\r{c}}}\sum_{i=1}^{j}|\nabla ^{i}\widetilde{K}(\tau,\xi)|^\frac{j}{i}\\
&\lesssim_{\text{\r{c}}}
\sum_{i=2}^j
\left( \frac{1}{\langle \xi\rangle^j}\frac{1}{ (|\tau|+|\xi|)^{\frac{2j}{i}}}+\frac{1}{\langle \xi\rangle^{\frac{2j}{i}}}\frac{1}{ (|\tau|+|\xi|)^{j}}
\right)
\| \nabla \mu\|_{\mathcal{W}^{2i+1}_{i+2}}^\frac{j}{i}+\	\frac{|\xi|^j}{\langle \xi\rangle^{2j} (|\tau|+|\xi|)^{2j}}\| \nabla \mu\|_{\mathcal{W}^{3}_{3}}^j\\
&\lesssim_{\text{\r{c}}}\left(\frac{1}{\langle\xi\rangle^j}
\frac{1}{ (|\tau|+|\xi|)^2}
+\frac{1}{\langle\xi\rangle^2}
\frac{1}{ (|\tau|+|\xi|)^j}\right)
\left(\| \nabla \mu\|_{\mathcal{W}_{j+2}^{2j+1}}
+\| \nabla \mu\|^j_{\mathcal{W}_{j+2}^{2j+1}}\right)+\	\frac{|\xi|^j}{\langle \xi\rangle^{2j} (|\tau|+|\xi|)^{2j}}\| \nabla \mu\|_{\mathcal{W}^{3}_{3}}^j\\
&\lesssim_{\text{\r{c}}}	 	M_j\left(
\frac{1}{\langle\xi\rangle^j (|\tau|+|\xi|)^2}
+
\frac{1}{\langle\xi\rangle^2 (|\tau|+|\xi|)^j}\right).
\end{align*}	
Moreover, one has 
\begin{align*}
\left|\partial_{\tau}^j\left(\frac{\widetilde{K}(\tau,\xi)}{1-\widetilde{K}(\tau,\xi)}\right)\right|
=&\left|\partial_{\tau}^j\left(\frac{1}{1-\widetilde{K}(\tau,\xi)}\right)\right|
\lesssim_{\text{\r{c}}}
\sum_{j_1=1}^{j}\left|\partial_\tau^{j_1}\widetilde{K}(\tau,\xi)\right|^\frac{j}{j_1}
\overset{\eqref{alpha K estimate}}	{\lesssim_{\text{\r{c}}}}
\sum_{j_1=1}^{j}\left|
\frac{|\xi|^2}{\langle\xi\rangle^2 (|\tau|+|\xi|)^{j_1+2}}
\right|^\frac{j}{j_1}
\|\nabla\mu\|_{\mathcal{W}_{j_1+3}^{j_1+2}}^\frac{j}{j_1}\\
\lesssim_{\text{\r{c}}}&\frac{1}{ (|\tau|+|\xi|)^j}
\left(\frac{|\xi|^2}{\langle\xi\rangle^2 (|\tau|+|\xi|)^2}
+\frac{|\xi|^{2j}}{\langle\xi\rangle^{2j} (|\tau|+|\xi|)^{2j}}
\right)
\left(\| \nabla \mu\|_{\mathcal{W}^{j+2}_{j+3}}^j
+\| \nabla \mu\|_{\mathcal{W}^{j+2}_{j+3}}\right)	\\
\lesssim_{\text{\r{c}}}&M_j
\frac{|\xi|^2}{\langle\xi\rangle^2 (|\tau|+|\xi|)^{j+2}}
.
\end{align*}
This completes the proof.		
\end{proof}

\begin{lemma}\label{estimate of nabla gamma Gn} Define $G_n=\mathscr{F}^{-1}\left(\widetilde{G}\chi_n\right),$
	where $\chi\in[0,1]$ is a  smooth compactly supported function in the annulus $1/4\leq|\xi|\leq4$ and $\chi_n(\xi)=\chi(\frac{\xi}{2^n})$.
Then, we have for any integer $\gamma\geq 0$
$$
\left|\nabla^{\gamma}G_n(t,x)\right|\lesssim_{\text{\r{c}}}
2^{(d+\gamma+1)n-2n^+}	M_{d+\gamma+2} \min\left\{1,\frac{1}{2^{n}(t+|x|)}\right\}^{d+\gamma+2},
$$
where $n^+=\max\{0,n\}$. 
\end{lemma}
\begin{proof} Firstly, we estimate $|\nabla ^\gamma G_{n}(t,x)|$ directly.
\begin{align*}
|\nabla ^\gamma G_{n}(t,x)|=&\left|\mathscr{F}^{-1}\left(
\xi^{\otimes \gamma}
\frac{\widetilde{K}(\tau,\xi)}{1-\widetilde{K}(\tau,\xi)}\chi_{n}(\xi)\right)\right|
\lesssim	\int_{\tau\in\mathbb{R},|\xi|\sim 2^n}
\left|\left(\frac{\widetilde{K}(\tau,\xi)}{1-\widetilde{K}(\tau,\xi)}\right)	\xi^{\otimes\gamma}   \chi_{n}(\xi)\right|
d\xi d\tau\\\lesssim_{\text{\r{c}}}&\int_{\tau\in\mathbb{R},|\xi|\sim 2^{n}}
2^{\gamma n}|\widetilde{K}(\tau,\xi)|d\tau d\xi.
\end{align*}
Using \eqref{alpha K estimate}, one has
\begin{align}\label{gn1}
|\nabla ^\gamma G_{n}(t,x)|
\lesssim_{\text{\r{c}}}&\|\nabla\mu\|_{\mathcal{W}_{ 3}^{ 2}}
\int_{\tau\in\mathbb{R},|\xi|\sim 2^n}
2^{\gamma n}	\frac{|\xi|^2}{\langle\xi\rangle^2 (|\tau|+|\xi|)^2}d\tau d\xi
\lesssim_{\text{\r{c}}}2^{(d+\gamma+1)n-2n^+}\|\nabla\mu\|_{\mathcal{W}_{ 3}^{ 2}}.
\end{align}
Then we integrate by parts in $\tau$ to get that 
\begin{align*}
|\nabla ^\gamma G_{n}(t,x)|
\sim&\left|\int_{\mathbb{R}\times \mathbb{R}^d}e^{i\tau t}e^{i\xi\cdot x}\frac{\widetilde{K}(\tau,\xi)}{1-\widetilde{K}(\tau,\xi)}\xi^{\otimes\gamma}\chi_{n}(\xi)d\xi d\tau\right|
\\
\lesssim &
\frac{1}{|t|^{m_1}}	
\int_{\tau\in\mathbb{R},|\xi|\sim 2^{n}}
\left|	
\partial_{\tau}^{m_1}\left(\frac{\widetilde{K}(\tau,\xi)}{1-\widetilde{K}(\tau,\xi)}\right)
\xi^{\otimes\gamma} \chi_{n}(\xi)\right|
d\xi d\tau.
\end{align*}
Using \eqref{K/1-K tau},  we obtain that
\begin{align}	\nonumber	
|\nabla ^\gamma G_{n}(t,x)|\lesssim_{\text{\r{c}}}&M_{m_1}
\frac{1}{|t|^{m_1}}	\int_{\tau\in\mathbb{R},|\xi|\sim 2^{n}}\frac{|\xi|^{2+\gamma}}{\langle\xi\rangle^2 (|\tau|+|\xi|)^{m_1+2}}d\xi d\tau \\\label{gn2}	
\lesssim_{\text{\r{c}}}&\frac{1}{(|t|2^{n})^{m_1}}2^{(d+\gamma+1)n-2n^+}   
M_{m_1}.
\end{align}
Finally, we integrate by parts in $\xi$ to get that  
\begin{align*}
|\nabla ^\gamma G_{n}(t,x)|
\sim&\left|\int_{\mathbb{R}\times \mathbb{R}^d}e^{i\tau t}e^{i\xi\cdot x}\frac{\widetilde{K}(\tau,\xi)}{1-\widetilde{K}(\tau,\xi)}\xi^{\otimes\gamma}\chi_{n}(\xi)d\xi d\tau\right|
\\
\lesssim&\frac{1}{|x|^{m_2}}
\int_{\tau\in\mathbb{R},|\xi|\sim 2^{n}}
\left|\nabla ^{m_2}\left(\frac{\widetilde{K}(\tau,\xi)}{1-\widetilde{K}(\tau,\xi)}\xi^{\otimes\gamma}\chi_{n}(\xi)\right)\right|d\xi d\tau\\
\lesssim_{\text{\r{c}}}&\frac{1}{|x|^{m_2}}
\int_{\tau\in\mathbb{R},|\xi|\sim 2^{n}}
\left|\widetilde{K}(\tau,\xi)\right|\left|\nabla ^{m_2}\left(\xi^{\otimes\gamma}\chi_{n}(\xi)\right)\right|d\xi d\tau\\
&+\frac{1}{|x|^{m_2}}\sum_{j=1}^{m_2}\int_{\tau\in\mathbb{R},|\xi|\sim 2^{n}}
\left|\nabla ^j\frac{\widetilde{K}(\tau,\xi)}{1-\widetilde{K}(\tau,\xi)}\right|\left|\nabla ^{m_2-j}\left(\xi^{\otimes\gamma}\chi_{n}(\xi)\right)\right|
d\xi d\tau.
\end{align*}
Using \eqref{alpha and beta K estimate} and \eqref{K/1-K xi},  we obtain that
\begin{equation}\label{gn3}
\begin{split}
	|\nabla ^\gamma G_{n}(t,x)|\lesssim_{\text{\r{c}}}&\|\nabla\mu\|_{\mathcal{W}_{ 3}^{ 2}}\frac{1}{|x|^{m_2}}\int_{\tau\in\mathbb{R},|\xi|\sim 2^{n}}
	\frac{|\xi|^2}{\langle\xi\rangle^2	 (|\tau|+|\xi|)^2}2^{(\gamma-m_2)n}d\xi d\tau\\
	&+\frac{1}{|x|^{m_2}}\sum_{j=1}^{m_2}
	\int_{\tau\in\mathbb{R},|\xi|\sim 2^{n}}
	M_{j}
	\left(\frac{1}{\langle \xi\rangle^j}
	\frac{1}{ (|\tau|+|\xi|)^2}
	+\frac{1}{\langle\xi\rangle^2}\frac{1}{(|\tau|+|\xi|)^{j}}\right)2^{(\gamma-m_2+j)n}d\xi d\tau\\
	\lesssim_{\text{\r{c}}}&\frac{M_{m_2}}{|x|^{m_2}}\sum_{j=1}^{m_2}\int_{|\xi|\sim 2^{n}}
	\left[\frac{|\xi|}{\langle\xi\rangle^2	}2^{(\gamma-m_2)n}+\left(\frac{1}{\langle \xi\rangle^j}
	\frac{1}{|\xi|}
	+\frac{1}{\langle\xi\rangle^2}\frac{1}{|\xi|^{j-1}}\right)2^{(\gamma-m_2+j)n}\right]d\xi\\
	\lesssim_{\text{\r{c}}}&\frac{1}{(|x|2^{n})^{m_2}}2^{(d+\gamma+1)n-2n^+}   M_{m_2}.
\end{split}
\end{equation}
Combining \eqref{gn1}, \eqref{gn2} with \eqref{gn3} and taking $m_1=m_2=d+\gamma+2$, we obtain
\begin{align*}
\left|\nabla^{\gamma}G_n(t,x)\right|\lesssim_{\text{\r{c}}}
2^{(d+\gamma+1)n-2n^+}	M_{d+\gamma+2} \min\left\{1,\frac{1}{2^{n}(t+|x|)}\right\}^{d+\gamma+2}.
\end{align*}
This gives the result.
\end{proof}
\begin{theorem}\label{G_estimate}    Let $\gamma \geq 0$, $0\leq a_0<1$ and $0<b_0<1$. Then, 
\begin{equation}\label{Z1}
|\nabla ^\gamma G(t,x)|\lesssim_{\text{\r{c}}}	\frac{\widetilde{M}_\gamma}{(t+|x|)^{d+\gamma-1}(1+t+|x|)^2}.
\end{equation}	
In particular,
\begin{align}\label{Z2}
\|\nabla ^\gamma G(t,x)\|_{L^\infty_x}&\lesssim_{\text{\r{c}}}	\frac{\widetilde{M}_\gamma}{t^{d+\gamma-1}(1+t)^2} ,\\  \label{Z3}
\||x|^{a_0}\nabla ^\gamma G(t,x)\|_{L^1_x}&\lesssim_{\text{\r{c}}}\frac{\widetilde{M}_\gamma}{t^{\gamma-1-a_0}(1+t)^{2}}+\widetilde{M}_\gamma\mathbf{1}_{t\leq 1}+\widetilde{M}_\gamma\mathbf{1}_{\gamma=1+a_0}\log^+\left(\frac{1}{t}\right),\\
\||x|^{b_0}\delta_\alpha  G(t,x)\|_{L^1_x}\label{Z5}
&\lesssim_{\text{\r{c}}}\widetilde{M}_1\frac{|\alpha|}{(1+t+|\alpha|)^{2-b_0}}+\widetilde{M}_1\mathbf{1}_{|\alpha|>t}\frac{|\alpha|^{b_0}}{1+t}+\widetilde{M}_1\mathbf{1}_{|\alpha|\leq t}\frac{|\alpha|^{d+b_0}}{t^{d-1}(1+t)^2},
\end{align}
where
\begin{align*}
\widetilde{M}_\gamma=\overline{M}_\gamma\big(1+\overline{M}_\gamma\big)^{d+\gamma+1},\qquad \text{and}\quad \overline{M}_\gamma=\|\langle\cdot\rangle^{d+\gamma+4}\nabla\mu(\cdot)\|_{W^{2d+2\gamma+5,1}}.
\end{align*} 
\end{theorem}
\begin{proof}  Firstly, we prove \eqref{Z1}.
Note that
$$
\|\nabla\mu\|_{\mathcal{W}_{\alpha}^{\beta}}\lesssim\|\langle\cdot\rangle^{\alpha}\nabla\mu\|_{W^{\beta,1}},
$$
thus
$$
M_{d+\gamma+2}\lesssim \overline{M}_\gamma\big(1+\overline{M}_\gamma\big)^{d+\gamma+1}=\widetilde{M}_\gamma.
$$
Since
\begin{align*}
|\nabla ^\gamma G(t,x)|\leq \sum_{n\leq0}\left|\nabla ^{\gamma}G_n(t,x)\right|+\sum_{n>0}
\left|\nabla ^{\gamma}G_n(t,x)\right|,
\end{align*}
for $n>0$, combining this with Lemma \ref{estimate of nabla gamma Gn}, we can yield that
\begin{align*}
\sum_{n>0}\left|\nabla ^{\gamma}G_n(t,x)\right|\lesssim_{\text{\r{c}}}&
\widetilde{M}_\gamma
\sum_{n>0}
2^{(d+\gamma-1)n}
\min\left\{1,\frac{1}{2^n(t+|x|)}\right\}^{d+\gamma+2}\\
\sim_{\text{\r{c}}}&\widetilde{M}_\gamma\sum_{n>0}
\Bigg(\mathbf{1}_{t+|x|\leq1}\sum_{1\leq2^n\leq (t+|x|)^{-1}}	2^{(d+\gamma-1)n}
+\mathbf{1}_{t+|x|\leq1}\sum_{2^n> (t+|x|)^{-1}}
2^{-3n}\frac{1}{ (t+|x|)^{d+\gamma+2}}\\
&
\qquad\qquad+\mathbf{1}_{t+|x|>1}\frac{1}{2^{3n} (t+|x|)^{d+\gamma+2}}\Bigg)\\
\lesssim_{\text{\r{c}}}&\widetilde{M}_\gamma\left(\mathbf{1}_{t+|x|\leq1}\frac{1}{ (t+|x|)^{d+\gamma-1}}+\mathbf{1}_{t+|x|>1}
\frac{1}{ (t+|x|)^{d+\gamma+2}}\right).
\end{align*}
Similarly, we have for $n\leq0$,
\begin{align*}
\sum_{n\leq0}\left|\nabla ^{\gamma}G_n(t,x)\right|\lesssim_{\text{\r{c}}}	&\widetilde{M}_\gamma\left(\mathbf{1}_{t+|x|\leq1} \sum_{n\leq0}2^{(d+\gamma+1)n}+\mathbf{1}_{t+|x|>1}\sum_{2^n\leq(t+|x|)^{-1}}2^{(d+\gamma+1)n}\right.\\
&\left.+\mathbf{1}_{t+|x|>1}\sum_{(t+|x|)^{-1}<2^n\leq1}\frac{2^{-n}}{ (t+|x|)^{d+\gamma+2}}\right)\\
\lesssim_{\text{\r{c}}}&\widetilde{M}_\gamma\left(\mathbf{1}_{t+|x|\leq1}+\mathbf{1}_{t+|x|>1}\frac{1}{ (t+|x|)^{d+\gamma+1}}\right).
\end{align*}
Finally, we deduce that
\begin{align*}
|\nabla ^\gamma G(t,x)|\lesssim_{\text{\r{c}}}&\widetilde{M}_\gamma\left(
\mathbf{1}_{t+|x|\leq1}\frac{1}{ (t+|x|)^{d+\gamma-1}}+\mathbf{1}_{t+|x|>1}\frac{1}{ (t+|x|)^{d+\gamma+1}}\right)\\
\lesssim_{\text{\r{c}}}&\frac{\widetilde{M}_\gamma}{(t+|x|)^{d+\gamma-1}(1+t+|x|)^2} .
\end{align*}
This implies \eqref{Z1} and \eqref{Z2}. \\
\textbf{(1)}	Now we prove \eqref{Z3}. One has, 
\begin{align*}
\||x|^{a_0}\nabla ^\gamma G(t,x)\|_{L^1_x}&\lesssim_{\text{\r{c}}}\widetilde{M}_\gamma\int\frac{|x|^{a_0}}{(t+|x|)^{d+\gamma-1}(1+t+|x|)^2}dx.
\end{align*}
If $t\geq 1$, we estimate 
\begin{align*}
\||x|^{a_0}\nabla ^\gamma G(t,x)\|_{L^1_x}\lesssim_{\text{\r{c}}}\widetilde{M}_\gamma\int\frac{|x|^{a_0}}{(t+|x|)^{d+\gamma+1}}dx\sim_{\text{\r{c}}}\frac{\widetilde{M}_\gamma}{t^{\gamma+1-a_0}}\int\frac{|x|^{a_0}}{(1+|x|)^{d+\gamma+1}}dx\sim_{\text{\r{c}}}\frac{\widetilde{M}_\gamma}{t^{\gamma+1-a_0}}.
\end{align*}
If $0\leq t\leq 1$,  we estimate 
\begin{align*}
\||x|^{a_0}\nabla ^\gamma G(t,x)\|_{L^1_x}\lesssim_{\text{\r{c}}}&\widetilde{M}_\gamma\left(\int_{|x|\leq t}+\int_{t<|x|<1}+\int_{|x|\geq1}\right)\frac{|x|^{a_0}}{(t+|x|)^{d+\gamma-1}(1+t+|x|)^2}dx\\
\lesssim_{\text{\r{c}}}&\widetilde{M}_\gamma\left(\int_{|x|\leq t}\frac{|x|^{a_0}}{t^{d+\gamma-1}}dx+\int_{t<|x|<1}\frac{|x|^{a_0}}{|x|^{d+\gamma-1}}dx+\int_{|x|\geq1}\frac{|x|^{a_0}}{|x|^{d+\gamma+1}}dx\right)\\
\lesssim_{\text{\r{c}}}&\widetilde{M}_\gamma\left(\frac{1}{t^{\gamma-1-{a_0}}}+1+\mathbf{1}_{\gamma<1+a_0}+\mathbf{1}_{\gamma=1+a_0}\log\left(\frac{1}{t}\right)+\frac{\mathbf{1}_{\gamma>1+a_0}}{t^{\gamma-a_0-1}}\right)\\
\sim_{\text{\r{c}}}&\widetilde{M}_\gamma\left(\frac{1}{t^{\gamma-1-a_0}}+1+\mathbf{1}_{\gamma=1+a_0}\log\left(\frac{1}{t}\right)\right).
\end{align*}
Combining these two cases to deduce that
\begin{align*}
\||x|^{a_0}\nabla ^\gamma G(t,x)\|_{L^1_x}&\lesssim_{\text{\r{c}}}\mathbf{1}_{t\geq 1}\frac{\widetilde{M}_\gamma}{t^{\gamma+1-a_0}}+\mathbf{1}_{t\leq 1}\widetilde{M}_\gamma\left(\frac{1}{t^{\gamma-1-a_0}}+1+\mathbf{1}_{\gamma=1+a_0}\log\left(\frac{1}{t}\right)\right)\\
&\sim_{\text{\r{c}}} \frac{\widetilde{M}_\gamma}{t^{\gamma-1-a_0}(1+t)^2}+\widetilde{M}_\gamma\mathbf{1}_{t\leq 1}+\mathbf{1}_{\gamma=1+a_0}\log^+\left(\frac{1}{t}\right).
\end{align*}
\textbf{(2)}	Now we prove \eqref{Z5}. One has 
\begin{align*}	
\||x|^{b_0}|\delta_\alpha  G(t,x)|\|_{L^1_x}=\left(\int_{|\alpha|\leq|x|/2}+\int_{|\alpha|\geq |x|/2}\right)	|x|^{b_0}|\delta_\alpha  G(t,x)|dx
=:&I_{G,b_0}^1+I_{G,b_0}^2.
\end{align*}
For the term $I^{1}_{G,b_0}$,
\begin{align*}
I^1_{G,b_0}\leq&\int_{|\alpha|\leq|x|/2}	\int_{0}^{1}|x|^{b_0}
\left|\nabla G(t,x-\tau\alpha)\right||\alpha|d\tau dx\\
\lesssim_{\text{\r{c}}}&|\alpha|\int_{0}^{1}
\int_{|\alpha|\leq|x|/2}|x|^{b_0}\frac{\widetilde{M}_1}{(t+|x-\tau\alpha|)^{d}(1+t+|x-\tau\alpha|)^2}dxd\tau\\
\lesssim_{\text{\r{c}}}&|\alpha|
\int_{|\alpha|\leq|x|/2}|x|^{b_0}\frac{\widetilde{M}_1}{(t+|x|/2)^{d}(1+t+|x|/2)^2}dx\\
\lesssim_{\text{\r{c}}}&\mathbf{1}_{|\alpha|\geq t/2}
|\alpha|
\int_{|\alpha|\leq|x|/2}\frac{\widetilde{M}_1}{|x|^{d-b_0}(1+|x|)^2}dx+\mathbf{1}_{|\alpha|< t/2}
|\alpha|
\int_{|\alpha|\leq|x|/2}|x|^{b_0}\frac{\widetilde{M}_1}{(t+|x|/2)^{d}(1+t+|x|/2)^2}dx\\
=&:I^{1,1}_{G,b_0}+I^{1,2}_{G,b_0}.
\end{align*}
Then, we estimate that 
\begin{align}
I^{1,1}_{G,b_0}\lesssim&\nonumber\mathbf{1}_{|\alpha|\geq t/2}
\mathbf{1}_{|\alpha|\geq1}|\alpha|
\int_{|\alpha|\leq|x|/2}\frac{\widetilde{M}_1}{|x|^{d+2-b_0}}dx\\
&+\mathbf{1}_{|\alpha|\geq t/2}
\mathbf{1}_{|\alpha|<1}|\alpha|
\left(	\int_{|\alpha|\leq|x|\leq2}\widetilde{M}_1\frac{1}{|x|^{d-b_0}}dx+\int_{|x|\geq2}\widetilde{M}_1\frac{1}{|x|^{d+2-b_0}}dx\right)
\nonumber\\
\lesssim&\widetilde{M}_1|\alpha|\left(\mathbf{1}_{|\alpha|\geq t/2}
\mathbf{1}_{|\alpha|\geq1}\frac{1}{|\alpha|^{2-b_0}}
+\mathbf{1}_{|\alpha|\geq t/2}
\mathbf{1}_{|\alpha|<1}\right)
\sim\widetilde{M}_1\mathbf{1}_{|\alpha|\geq t/2}\frac{|\alpha|}{(1+|\alpha|)^{2-b_0}},\label{IGa11}
\end{align}
and 
\begin{equation}\label{IGa12}
\begin{split}
	I^{1,2}_{G,b_0}
	\lesssim&\mathbf{1}_{|\alpha|< t/2}
	|\alpha|\int_{|x|\geq t}\frac{\widetilde{M}_1}{|x|^{d-b_0}(1+|x|)^2}
	dx+\mathbf{1}_{|\alpha|< t/2}
	|\alpha|\int_{2|\alpha|\leq|x|\leq t}\frac{\widetilde{M}_1|x|^{b_0}}{t^{d}(1+t)^2}
	dx\\
	\lesssim&\mathbf{1}_{|\alpha|< t/2}\mathbf{1}_{t\geq 1}
	|\alpha|
	\int_{|x|\geq t}
	\frac{\widetilde{M}_1}{|x|^{d-b_0+2}}
	dx
	+\mathbf{1}_{|\alpha|< t/2}	\mathbf{1}_{ t\leq1}
	|\alpha|
	\left(	\int_{t\leq|x|\leq1}
	\frac{\widetilde{M}_1}{|x|^{d-b_0}}
	dx+\int_{|x|\geq 1}
	\frac{\widetilde{M}_1}{|x|^{d-b_0+2}}
	dx\right)\\
	&+\mathbf{1}_{|\alpha|< t/2}
	|\alpha|\int_{|\alpha|\leq|x|\leq2t}
	\frac{\widetilde{M}_1|x|^{b_0}}{t^{d}(1+t)^2}dx\\
	\sim&\widetilde{M}_1|\alpha|\mathbf{1}_{|\alpha|< t/2}
	\left(\mathbf{1}_{t\geq 1}\frac{1}{t^{2-b_0}}+\mathbf{1}_{t\leq1 }
	+ \frac{t^{b_0}}{(1+t)^2}\right)
	\sim\widetilde{M}_1\mathbf{1}_{|\alpha|< t}\frac{|\alpha|}{(1+t)^{2-b_0}}.
\end{split}
\end{equation}
Combining \eqref{IGa11} with \eqref{IGa12} to yield
\begin{equation}\label{IGa1}
I_{G,b_0}^1\lesssim_{\text{\r{c}}}\widetilde{M}_1\mathbf{1}_{|\alpha|\geq t}\frac{|\alpha|}{(1+|\alpha|)^{2-b_0}}+\widetilde{M}_1\mathbf{1}_{|\alpha|< t}\frac{|\alpha|}{(1+t)^{2-b_0}}\sim _{\text{\r{c}}} \widetilde{M}_1\frac{|\alpha|}{(1+t+|\alpha|)^{2-b_0}}.
\end{equation}
For the term $I_{G,b_0}^2$, one has
\begin{align*}
I_{G,b_0}^2\lesssim_{\text{\r{c}}}&\widetilde{M}_0|\alpha|^{b_0}\int_{|x|\lesssim|\alpha|}\frac{1}{(t+|x|)^{d-1}(1+t+|x|)^2}dx
\sim_{\text{\r{c}}}\widetilde{M}_0|\alpha|^{b_0}\int_{0}^{|\alpha|}\frac{r^{d-1}}{(t+r)^{d-1}(1+t+r)^2}dr\\
\sim_{\text{\r{c}}}&\mathbf{1}_{|\alpha|>t}\widetilde{M}_0|\alpha|^{b_0}\left(	\int_{0}^t\frac{r^{d-1}}{t^{d-1}(1+t)^2}dr+\int_{t}^{|\alpha|}\frac{r^{d-1}}{r^{d-1}(1+r)^2}dr\right)+\mathbf{1}_{|\alpha|\leq t}\widetilde{M}_0|\alpha|^{b_0}	\int_{0}^{|\alpha|}\frac{r^{d-1}}{t^{d-1}(1+t)^2}dr\\
\lesssim_{\text{\r{c}}}&\mathbf{1}_{|\alpha|>t}\widetilde{M}_0|\alpha|^{b_0}\left(\frac{t}{(1+t)^2}+\frac{1}{1+t}\right)+\mathbf{1}_{|\alpha|\leq t}\frac{\widetilde{M}_0|\alpha|^{d+b_0}}{t^{d-1}(1+t)^2}\\
\lesssim_{\text{\r{c}}}&\widetilde{M}_0\mathbf{1}_{|\alpha|>t}\frac{|\alpha|^{b_0}}{1+t}+\widetilde{M}_0\mathbf{1}_{|\alpha|\leq t}\frac{|\alpha|^{d+b_0}}{t^{d-1}(1+t)^2}.
\end{align*}
Combining this with \eqref{IGa1} and note that $\widetilde{M}_0\leq\widetilde{M}_1$, one gets 
\begin{align*}
\||x|^{b_0}\delta_\alpha  G(t,x)\|_{L^1_x}
&	\lesssim_{\text{\r{c}}}	\widetilde{M}_1\frac{|\alpha|}{(1+t+|\alpha|)^{2-b_0}}+\widetilde{M}_1\mathbf{1}_{|\alpha|>t}\frac{|\alpha|^{b_0}}{1+t}+\widetilde{M}_1\mathbf{1}_{|\alpha|\leq t}\frac{|\alpha|^{d+b_0}}{t^{d-1}(1+t)^2},
\end{align*}
which implies \eqref{Z5}. 
\end{proof}	\\
\begin{lemma}\label{estimation G}
For the kernel $G$ defined in \eqref{definition of G about K},
there holds	
\begin{equation}\label{cancellation}
\int_{\mathbb{R}^d}G(t,x)dx=0,~~\forall  t\geq 0.
\end{equation}
\end{lemma}
\begin{proof} By \eqref{alpha K estimate} with $\alpha=0$,  one has $\widetilde{K}(\tau,0)=0$.
Then from \eqref{definition of G about K} , we obtain that $\widetilde{G}(\tau,0)=0$, for all $\tau\in \mathbb{R}$, which implies the conclusion \eqref{cancellation}.
\end{proof}	\vspace{0.3cm}\\
Then we can prove the main result in this section.
\begin{theorem}\label{t modulus estimation} There holds
	$$\|G*_{(t,x)}\mathbf{f}\|_{\gamma,a}\leq M	\|\mathbf{f}\|_{\gamma,a},$$
	with
	$M=M(\text{\r{c}},d,M_{\gamma+1})$.
\end{theorem}
It is easy to check that for any $t\in [0,1]$
\begin{equation*}
	\sum_{p=1,\infty}\left(	\| G*_{(t,x)}\mathbf{f}(t)\|_{L^p}+\|\nabla^\gamma (G*_{(t,x)}\mathbf{f}(t))\|_{\dot{B}_{p,\infty}^{a}}\right)\leq \|\mathbf{f}\|_{\gamma,a}\int_{0}^{1}\|G(s)\|_{L^1_x}ds  \overset{\eqref{Z3}}\lesssim_{\text{\r{c}},d,M_{\gamma+1}}\|\mathbf{f}\|_{\gamma,a}.
\end{equation*}
To prove the  theorem \ref{t modulus estimation}, we need to have the following Lemmas.
\begin{lemma}\label{t gamma modulus 1} Let $t\geq 1$, 
and $p=1,\infty$, we have 
\begin{align*}
\langle t\rangle^{\frac{d(p-1)}{p}}
\|G*_{(t,x)}\mathbf{f}(t)\|_{L^p}\lesssim_{\text{\r{c}},M_0} \sup_{s\in [0,\infty)}\left(
\|\mathbf{f}(s)\|_{L^1}
+\langle s\rangle^{a+\frac{d(p-1)}{p}}
\| \mathbf{f}(s)\|_{\dot{B}_{p,\infty}^{a}}\right).
\end{align*}
\end{lemma}
\begin{proof} We use the idea in \cite{ChenNguyen}.
Since we have the cancellation shown in the  \eqref{cancellation}, we rewrite lines as follows when changing of variable $z=x-y$ in the second term,

\begin{align*}
\|G*_{(t,x)}\mathbf{f}\|_{L^p}&\leq
\left\|\int_{0}^{t/2}\int_{\mathbb{R}^{d}} G(t-s,x-y)\mathbf{f}(s,y)dyds\right\|_{L^p_x}\\ 
&\quad +\left\|\int_{t/2}^{t}\int_{\mathbb{R}^{d}}G(t-s,z)(\mathbf{f}(s,x)-\mathbf{f}(s,x+z))dzds\right\|_{L^p_x}:=I_{1}+I_{2}.
\end{align*}
First, we consider the term $I_1,$ $$
I_{1}
\leq
\int_{0}^{t/2}\| G(t-s)\|_{L^p}\|\mathbf{f}(s)\|_{L^1}ds.
$$
Using \eqref{Z2} and \eqref{Z3} in  Theorem \ref{G_estimate}, we  obtain
\begin{align}\label{i1}
I_{1}\lesssim_{\text{\r{c}},M_0}
\int_{0}^{t/2}\frac{1}{\langle t-s\rangle^{1+\frac{d(p-1)}{p}}}
\|\mathbf{f}(s)\|_{L^1}ds
\lesssim_{\text{\r{c}},M_0}
\widetilde{M}{\langle t\rangle^{-\frac{d(p-1)}{p}}}
\sup_{s\in [0,\infty)}	\|\mathbf{f}(s)\|_{L^1}.
\end{align}
Similarly, we have
\begin{align}\label{i2}
\begin{split}
	I_{2}
&=\left\|\int_{t/2}^{t}\int_{\mathbb{R}^{d}}G(t-s,z)|z|^a\frac{\mathbf{f}(s,x)-\mathbf{f}(s,x+z)}{|z|^a}dzds\right\|_{L^p_x}\\
&\leq\int_{t/2}^{t}\int_{\mathbb{R}^{d}}|G(t-s,z)||z|^a\frac{\left\|\mathbf{f}(s,x)-\mathbf{f}(s,x+z)\right\|_{L^p_x}}{|z|^a}dzds\\
&\leq\int_{t/2}^{t}\sup_{z}\frac{\left\|\delta_z\mathbf{f}(s,x)\right\|_{L^p_x}}{|z|^a}\|G(t-s,z)|z|^a\|_{L^1_x}ds.
\end{split}
\end{align}
Using \eqref{Z3} in  Theorem \ref{G_estimate}, we have,		
\begin{align*}
I_{2}&\lesssim_{\text{\r{c}},M_0}\widetilde{M}\int_{t/2}^{t}\|\mathbf{f}(s)\|_{\dot{B}_{p,\infty}^a}\frac{1}{\langle t-s\rangle^{1-a}}ds\\&\lesssim_{\text{\r{c}},M_0}\widetilde{M}\left[\sup_{s\in [0,\infty)}\langle s\rangle^{\frac{d(p-1)}{p}+a}\|\mathbf{f}(s)\|_{\dot{B}_{p,\infty}^a}\right]\langle t\rangle^{-\frac{d(p-1)}{p}-a}\int_{t/2}^{t}\frac{1}{\langle t-s\rangle^{1-a}}ds\\
&\lesssim_{\text{\r{c}},M_0}\widetilde{M}\langle t\rangle^{-\frac{d(p-1)}{p}}\sup_{s\in [0,\infty)}\left(\langle s\rangle^{\frac{d(p-1)}{p}+a}\|\mathbf{f}(s)\|_{\dot{B}_{p,\infty}^a}\right).
\end{align*}
Combining \eqref{i1} and \eqref{i2} we obtain that
\begin{align*}
\langle t\rangle^{\frac{d(p-1)}{p}}
\|G*_{(t,x)}\mathbf{f}(t)\|_{L^p}
\lesssim_{\text{\r{c}},M_0}\widetilde{M}\sup_{s\in [0,\infty)}\left(
\|\mathbf{f}(s)\|_{L^1}
+\langle s\rangle^{a+\frac{d(p-1)}{p}}
\| \mathbf{f}\|_{\dot{B}_{p,\infty}^{a}}\right).
\end{align*}
This implies the result.
\end{proof}
\begin{lemma}\label{t modulus 3} Let $t\geq 1$,   
	and $p=1,\infty$, we have estimates  as follows,
	\begin{align*}
		\langle t\rangle^{a+\frac{d(p-1)+\gamma}{p}}\|\nabla^\gamma(G*_{(t,x)}\mathbf{f}(t))\|_{\dot{B}_{p,\infty}^a}
		\lesssim_{\text{\r{c}},M_1}\widetilde{M}\sup_{s\in[0,\infty)}\left(\|\mathbf{f}(s)\|_{L^1}
		+\langle s\rangle^{a+\frac{d(p-1)}{p}+\gamma}\|\nabla^\gamma\mathbf{f}(s)\|_{\dot{B}_{p,\infty}^a}\right).
	\end{align*}	
\end{lemma}
\begin{proof}
	When referring to $\| G*_{(t,x)}\mathbf{f}(t)\|_{\dot{B}_{p,\infty}^a}$, we consider $\|\delta_\alpha(G*_{(t,x)}\mathbf{f}(t))\|_{L^p}$  at first.
	Set 
	\begin{align*}
		&u_1(t,x)=	\int_{0}^{t/2}\int_{\mathbb{R}^d} G(t-s,y)\mathbf{f}(s,x-y)dy ds,\\&
		u_2(t,x)=	\int_{t/2}^{t}\int_{\mathbb{R}^d} G(t-s,y)\mathbf{f}(s,x-y)dy ds.	
	\end{align*}
	Thus,
	$$	\|\delta_\alpha\nabla^\gamma (G*_{(t,x)}\mathbf{f}(t))\|_{L^p}
	\leq
	\|\delta_{\alpha}\nabla^\gamma u_1(t,x)\|_{L^p_x}
	+\|\delta_{\alpha}\nabla^\gamma u_2(t,x)\|_{L^p_x}.
	$$
	Firstly, we focus on the second term. Thanks to \eqref{Z3} and \eqref{cancellation}, we have
	for $|\alpha|>t/2\geq\frac{1}{2}$
	\begin{align*}
		\frac{\|\delta_{\alpha}\nabla^\gamma u_2(t,x)\|_{L^p_x}}{|\alpha|^a}\leq&
		\frac{2}{|\alpha|^a}\|\nabla^\gamma  u_2(t,x)\|_{L^p_x}
		=\frac{2}{|\alpha|^a}\left\|\int_{t/2}^{t}\int_{\mathbb{R}^d} G(t-s,y) \nabla^\gamma
		\mathbf{f}(s,x-y)dy ds\right\|_{L^p_x}\\
		\lesssim&\frac{2}{|\alpha|^a}\int_{t/2}^{t}
		\||y|^aG(t-s,y)\|_{L^1_y}
		\sup_y\left(\frac{\|\delta_y\nabla^\gamma\mathbf{f}(s)\|_{L^p}}{|y|^a}\right)ds\\
		\lesssim_{\text{\r{c}}}&\frac{1}{|\alpha|^a}
		\left[\sup_{s\in[0,\infty)}\langle s\rangle^{\frac{d(p-1)}{p}+a+\gamma}
		\|\nabla^\gamma\mathbf{f}\|_{\dot{B}^a_{p,\infty}}\right]\langle t\rangle^{-\frac{d(p-1)}{p}-a-\gamma}
		\int_{t/2}^{t}
		\frac{\widetilde{M}}{\langle t-s\rangle^{1-a}}ds\\
		\lesssim_{\text{\r{c}}}&\widetilde{M}\frac{\langle t \rangle^a}{|\alpha|^a}
		\left[\sup_{s\in[0,\infty)}\langle s\rangle^{\frac{d(p-1)}{p}+a+\gamma}
		\|\nabla^\gamma\mathbf{f}\|_{\dot{B}^a_{p,\infty}}\right]
		\langle t\rangle^{-\frac{d(p-1)}{p}-a-\gamma}\\
		\lesssim_{\text{\r{c}}}&
		\widetilde{M}\langle t\rangle^{-\frac{d(p-1)}{p}-a-\gamma}\sup_{s\in[0,\infty)}\left(\langle s\rangle^{\frac{d(p-1)}{p}+a+\gamma}
		\|\nabla^\gamma\mathbf{f}\|_{\dot{B}^a_{p,\infty}}\right)
		.
	\end{align*}
	For the case $|\alpha|\leq t/2$, we know that
	\begin{align*}
		\frac{\|\delta_{\alpha}\nabla^\gamma u_2(t,x)\|_{L^p_x}}{|\alpha|^a}\lesssim&	\frac{1}{|\alpha|^a}\left\|\int_{t/2}^{t}\int_{\mathbb{R}^d}\mathbf{1}_{|t-s|\leq|\alpha|} \delta_{\alpha}G(t-s,y) \nabla^\gamma\delta_y\mathbf{f}(s,x)dy ds\right\|_{L^p_x}\\
		&+	\frac{1}{|\alpha|^a}\left\|\int_{t/2}^{t}\int_{\mathbb{R}^d}\mathbf{1}_{|t-s|>|\alpha|} \delta_{\alpha}G(t-s,y) \nabla^\gamma\delta_y\mathbf{f}(s,x)dy ds\right\|_{L^p_x}
		=:\mathcal{J}_1+\mathcal{J}_2.
	\end{align*}
	For the term $\mathcal{J}_1$, one has
	\begin{align*}
		\mathcal{J}_1&\leq
		\frac{1}{|\alpha|^a}\left\|\int_{t/2}^{t}\int_{\mathbb{R}^d} \mathbf{1}_{|t-s|\leq|\alpha|}G(t-s,y)
		\delta_y\nabla^\gamma\mathbf{f}(s,x)dy ds\right\|_{L^p_x}\\
		&\leq\frac{1}{|\alpha|^a}\int_{t/2}^{t}
		\mathbf{1}_{|t-s|\leq|\alpha|}
		\||y|^a G(t-s,y)\|_{L^1}
		\sup_y\frac{\|\delta_y \nabla^\gamma\mathbf{f}(s)\|_{L^p}}{|y|^a}ds\\
		&\overset{\eqref{Z3}}{\lesssim_{\text{\r{c}}}}\frac{1}{|\alpha|^a}
		\left[\sup_{s\in[0,\infty)}\langle s\rangle^{\frac{d(p-1)}{p}+a+\gamma}
		\|\nabla^\gamma\mathbf{f}(s)\|_{\dot{B}^a_{p,\infty}}\right]
		\langle t\rangle^{-\frac{d(p-1)}{p}-a-\gamma}
		\int_{0}^{t/2}\||y|^a G(s,y)\|_{L^1_y}\mathbf{1}_{|s|\leq|\alpha|}ds.
	\end{align*}
	Since we have
	\begin{align*}
		\int_{0}^{t/2}\||y|^a G(s,y)\|_{L^1_y}\mathbf{1}_{|s|\leq|\alpha|}ds
		\overset{\eqref{Z3}}	{\lesssim_{\text{\r{c}}}}\widetilde{M}\int_{0}^{|\alpha|}\frac{1}{ s^{1-a}}ds\leq\widetilde{M}|\alpha|^a,
	\end{align*}
	then
	\begin{align}\label{z10}
		\mathcal{J}_1\lesssim_{\text{\r{c}},M_0}\widetilde{M}	\langle t\rangle^{-\frac{d(p-1)}{p}-a-\gamma}\sup_{s\in[0,\infty)}\left(\langle s\rangle^{\frac{d(p-1)}{p}+a+\gamma}
		\|\nabla^\gamma\mathbf{f}(s)\|_{\dot{B}^a_{p,\infty}}\right).
	\end{align}
	Now we deal with the term $\mathcal{J}_2$,
	\begin{align*}
		\mathcal{J}_2&\leq\frac{1}{|\alpha|^a}\int_{t/2}^{t}\mathbf{1}_{|t-s|>|\alpha|}\left(\int_{\mathbb{R}^d} |y|^a|\delta_{\alpha}G(t-s,y)|dy\right)
		\sup_y\frac{\left\|\delta_y \nabla^\gamma\mathbf{f}(s,x)\right\|_{L^p_x}}{|y|^a} ds\\
		&\leq\frac{1}{|\alpha|^a}\langle t\rangle^{-a-\frac{d(p-1)}{p}-\gamma}\left[\sup_{s\in[0,\infty)}\langle s\rangle^{a+\frac{d(p-1)}{p}+\gamma}\|\nabla^\gamma\mathbf{f}(s)\|_{\dot{B}_{p,\infty}^a}\right]\int_{|\alpha|}^{t/2}\||y|^a\delta_{\alpha}G(s,y)\|_{L^1_y} ds\\
		&\overset{\eqref{Z5}}	{\lesssim_{\text{\r{c}},M_0}}\frac{1}{|\alpha|^a}\langle t\rangle^{-a-\frac{d(p-1)}{p}-\gamma}\left[\sup_{s\in[0,\infty)}\langle s\rangle^{a+\frac{d(p-1)}{p}+\gamma}\|\nabla^\gamma\mathbf{f}(s)\|_{\dot{B}_{p,\infty}^a}\right]\int_{|\alpha|}^{t/2} \frac{|\alpha|}{(1+s)^{2-a}}ds\\
		&\lesssim_{\text{\r{c}},M_0}\langle t\rangle^{-a-\frac{d(p-1)-\gamma}{p}}\sup_{s\in[0,\infty)}\left(\langle s\rangle^{a+\frac{d(p-1)}{p}+\gamma}\|\nabla^\gamma\mathbf{f}(s)\|_{\dot{B}_{p,\infty}^a}\right).
	\end{align*}
	Combining this with \eqref{z10} to conclude that 
	\begin{align*}
		\frac{\|\delta_{\alpha} u_2(t,x)\|_{L^p_x}}{|\alpha|^a}&\lesssim_{\text{\r{c}},M_0}\langle t\rangle^{-a-\frac{d(p-1)}{p}-\gamma}\sup_{s\in[0,\infty)}\left(\langle s\rangle^{a+\frac{d(p-1)}{p}+\gamma}\|\nabla^\gamma\mathbf{f}(s)\|_{\dot{B}_{p,\infty}^a}\right).
	\end{align*}
	Now we estimate $u_1(t,x)$.  One
	\begin{align*}
		\|\delta_{\alpha}\nabla^\gamma u_1(t,x)\|_{L^p_x}
		=&	\left\|\int_{0}^{t/2}\int_{\mathbb{R}^d} \delta_{\alpha}\nabla^\gamma G(t-s,x-y) \mathbf{f}(s,y)dy ds\right\|_{L^p_x}\\   	
		\leq&
		\int_{t/2}^{t} \|\delta_{\alpha}\nabla^\gamma G(s)\|_{L^p_x}
		ds	\sup_{s\in [0,\infty)}\|\mathbf{f}(s)\|_{L^1}.
	\end{align*}
	Note that for any $s\gtrsim 1$, we have
	\begin{align*}
		\|\delta_{\alpha}\nabla^\gamma G(s)\|_{L^p_x}&\leq \min\{	\|\nabla^\gamma G(s)\|_{L^p_x},|\alpha|	\| \nabla^{\gamma+1} G(s)\|_{L^p_x}\}\\&\overset{\eqref{Z2},\eqref{Z3}}{\lesssim _{\text{\r{c}},M_{\gamma+1}}}  \min\left\{\frac{1}{\langle s\rangle^{1+\frac{d(p-1)}{p}+\gamma}}, \frac{|\alpha|}{\langle s\rangle^{2+\frac{d(p-1)}{p}+\gamma}}\right\}\lesssim _{\text{\r{c}},M_{\gamma+1}}\frac{|\alpha|^a}{\langle s\rangle^{1+a+\frac{d(p-1)}{p}+\gamma}}.
	\end{align*}
	Then
	\begin{align*}
		\int_{t/2}^{t} \|\delta_{\alpha}\nabla^\gamma G(s)\|_{L^p_x}
		ds&\lesssim _{\text{\r{c}},M_{\gamma+1}} \frac{|\alpha|^a}{\langle t\rangle^{1+a+\frac{d(p-1)}{p}+\gamma}} \langle t\rangle\lesssim _{\text{\r{c}},M_{\gamma+1}}|\alpha|^a{\langle t\rangle^{-a-\frac{d(p-1)}{p}-\gamma}}.
	\end{align*}
	Thus we obtain
	\begin{align*}
		\frac{\|\delta_{\alpha}\nabla^\gamma u_1(t,x)\|_{L^p_x}}{|\alpha|^a}\lesssim _{\text{\r{c}},M_{\gamma+1}}
		\langle t\rangle^{-a-\frac{d(p-1)}{p}-\gamma}\sup_{s\in [0,\infty)}\|\mathbf{f}(s)\|_{L^1}.
	\end{align*}
Therefore, we deduce 
	\begin{align*}
		\langle t\rangle^{a+\frac{d(p-1)}{p}+\gamma}\|G*_{(t,x)}\mathbf{f}(t)\|_{\dot{B}_{p,\infty}^a}\lesssim _{\text{\r{c}},M_{\gamma+1}}\sup_{s\in[0,\infty)}\left(\|\mathbf{f}(s)\|_{L^1}
		+\langle s\rangle^{a+\frac{d(p-1)}{p}+\gamma}\|\nabla^\gamma\mathbf{f}(s)\|_{\dot{B}_{p,\infty}^a}\right).
	\end{align*}
This gives the result.
\end{proof}
\section{Estimates for Characteristics}
In this section, we  study the characteristics $(X_{s,t}(x,v),V_{s,t}(x,v))$ defined in \eqref{X and V}. First, we have for any $0\leq s\leq t< \infty $, 
\begin{equation}\label{definition of X and V}
\begin{split}
&X_{s,t}(x,v)=x-v(t-s)+\int_{s}^{t}(\tau-s)E(\tau,X_{\tau,t}(x,v))d\tau,\\
&V_{s,t}(x,v)=v-\int_{s}^{t}E(\tau,X_{\tau,t}(x,v))d\tau.
\end{split}
\end{equation}
Set 
\begin{align}\label{definition Y W}
\begin{split}
	&Y_{s,t}(x,v)=\int_{s}^{t}(\tau-s)E(\tau,x+\tau v+Y_{\tau,t}(x,v))d\tau,\\&
W_{s,t}(x,v)=-\int_{s}^{t}E(\tau,x+\tau v+Y_{\tau,t}(x,v))d\tau.
\end{split}
\end{align}
Hence, 
\begin{equation}\label{z4}
X_{s,t}(x,v)=x-(t-s)v+Y_{s,t}(x-vt,v),~~V_{s,t}(x,v)=v+W_{s,t}(x-vt,v).
\end{equation}
We obtain the bounds on characteristics as follows.

\begin{proposition}\label{estimates about Y and W} Let $a\in (\frac{\sqrt{5}-1}{2},1)$, then there exists $\widetilde{\varepsilon_0}\in(0,1)$  such that  for any $\|(\rho,U)\|_{S^{{\varepsilon_0}}_a}\leq\varepsilon_0\leq \widetilde{\varepsilon_0}$ and $0<\delta<a$, we have the following estimates:
\begin{align}\nonumber
&\sup_{0\leq s\leq t}\langle s\rangle^{d-2+a}\left(	\|Y_{s,t}\|_{L^\infty_{x,v}}+\|\nabla_x Y_{s,t}\|_{L^\infty_{x,v}}+\sup_{\alpha}	\frac{	\|\delta_{\alpha}^x	\nabla_xY_{s,t}\|_{L^\infty_{x,v}}}{|\alpha|^a}\right)\\&\quad\quad\quad\quad\quad\quad+	\sup_{0\leq s\leq t}\langle s\rangle^{d-3+a}\|\nabla_v Y_{s,t}\|_{L^\infty_{x,v}}+	\delta\sup_{0\leq s\leq t}\langle s\rangle^{d-3+\delta}\sup_{\alpha}	\frac{	\|\delta_{\alpha}^v	\nabla_vY_{s,t}\|_{L^\infty_{x,v}}}{|\alpha|^{a-\delta}}\lesssim_a \|(\rho,U)\|_{S^{{\varepsilon_0}}_a},\label{results of Y small}
\end{align}
and
\begin{align}\nonumber
&\sup_{0\leq s\leq t}\langle s\rangle^{d-1+a}\left(	\|W_{s,t}\|_{L^\infty_{x,v}}+\|\nabla_x W_{s,t}\|_{L^\infty_{x,v}}+\sup_{\alpha}	\frac{	\|\delta_{\alpha}^x	\nabla_xW_{s,t}\|_{L^\infty_{x,v}}}{|\alpha|^a}\right)\\&\quad\quad\quad\quad\quad\quad+	\sup_{0\leq s\leq t}\langle s\rangle^{d-2+a}\|\nabla_v W_{s,t}(x,v)\|_{L^\infty_{x,v}}+	\sup_{0\leq s\leq t}\langle s\rangle^{d-2}\sup_{\alpha}	\frac{	\|\delta_{\alpha}^v	\nabla_vW_{s,t}\|_{L^\infty_{x,v}}}{|\alpha|^a}\lesssim_a \|(\rho,U)\|_{S^{{\varepsilon_0}}_a}.\label{results of W small}
\end{align}
Here we use the notations 
\begin{align*}
\delta_{\alpha}^xY_{s,t}(x,v)=Y_{s,t}(x,v)-Y_{s,t}(x-\alpha,v),~~~\delta_{\alpha}^vY_{s,t}(x,v)=Y_{s,t}(x,v)-Y_{s,t}(x,v-\alpha).
\end{align*}
\end{proposition}
Before we prove this Proposition, we state the following Lemma:
\begin{lemma}\label{varrho}
For $p=1,\infty$ and $\kappa,\kappa_1\in(0,1)$, one has
\begin{align*}
&	\sum_{j=0}^{1}\|\nabla^j(1-\Delta)^{-1}\psi\|_{L^p}+\|\nabla(1-\Delta)^{-1}\psi\|_{\dot{B}^\kappa_{p,\infty}}\leq c_0\|\psi\|_{L^p},\\
&\sum_{j=1}^{2}\|\nabla^j(1-\Delta)^{-1}\psi\|_{L^p}+[2+\kappa-(\kappa_1+j)]\sum_{\substack{j=0,1,2\\\kappa_1+j<2+\kappa}}\|\nabla^j(1-\Delta)^{-1}\psi\|_{\dot{B}^{\kappa_1}_{p,\infty}\cap \dot{F}^{\kappa_1}_{p,\infty}}
\leq c_0\|\psi\|_{\dot{B}^{\kappa}_{p,\infty}},\\
&\|\nabla^2(1-\Delta)^{-1}\psi\|_{\dot{B}^\kappa_{\infty,\infty}}\leq c_0\|\psi\|_{\dot{B}^{\kappa}_{\infty,\infty}},
\end{align*}
for some  $c_0=c_0(\kappa)$.
\end{lemma}		
\begin{proof}[Proof of Proposition \ref{estimates about Y and W}]
Note  that $\|(\rho,U)\|_{S^{{\varepsilon_0}}_a}\leq \varepsilon_0$ implies $\|U\|_a\leq \varepsilon_0^{\frac{2}{3}}$, then by Assumption \ref{A assump}, we have that 
\begin{align}\label{AU leq U}
	\begin{split}
		\|A(U)\|_{\dot{B}^a_{p,\infty}}&\leq \max_{|r|\leq\varepsilon_0^{2/3}}|A'(r)|\|U\|_{\dot{B}^a_{p,\infty}}\leq C_A \varepsilon_0^{\frac{2}{3}}\|U\|_{\dot{B}^a_{p,\infty}},\\
		\|A(U)\|_{L^p}&\leq C_A \|U^2\|_{L^p}\leq C_A \varepsilon_0^{\frac{2}{3}}\|U\|_{L^p}.
	\end{split}
\end{align}
Thus taking $\tilde{\varepsilon}_0\leq C_A^{-3}$, we get
\begin{align}\label{AU_a leq U_a}
	\|\rho+A(U)\|_a\leq	\|\rho\|_a+ C_A \varepsilon_0^{\frac{2}{3}}\|U\|_{a}\leq\|\rho\|_a+ \varepsilon_0^{\frac{1}{3}} \|U\|_a=\|(\rho,U)\|_{S^{{\varepsilon_0}}_a}.
\end{align}Thanks to Lemma \ref{varrho},
\begin{equation}\label{E estimate by rho}
\|E(\tau)\|_{L^\infty}+\|\nabla_x E(\tau)\|_{L^\infty}	+\|\nabla_x E(\tau)\|_{\dot{B}^a_{\infty,\infty}}\leq c_0(a)\|\rho(\tau)+A(U)(\tau)\|_{\dot{B}^a_{\infty,\infty}}\leq c_0(a)\frac{\|(\rho,U)\|_{S^{{\varepsilon_0}}_a}}{\langle \tau\rangle^{d+a}},
\end{equation} 
and there exists a constant $c_1>0$ such that

\begin{align}\label{c1}
\int_{s}^{t}\frac{1}{\langle\tau\rangle^{d-i+a}}d\tau\leq \frac{c_1}{\langle s\rangle^{d-i+a-1}},~~i=0,1,2,
\end{align}
where $c_1$  only depends on $d$.\\
\textbf{(1)} By \eqref{E estimate by rho} one has 
\begin{align*}
\|Y_{s,t}\|_{L^\infty_{x,v}}\leq&\int_{s}^{t}(\tau-s)\|E(\tau,X_{\tau,t}(x,v))\|_{L^\infty_{x,v}}d\tau\leq\int_{s}^{t}\frac{c_0\|(\rho,U)\|_{S^{{\varepsilon_0}}_a}}{\langle \tau\rangle^{d-1+a}}d\tau\leq c_0c_1\frac{\|(\rho,U)\|_{S^{{\varepsilon_0}}_a}}{\langle s\rangle^{d-2+a}},\\
\|\nabla_x Y_{s,t}\|_{L^\infty_{x,v}}\leq&
(1+\sup_{0\leq\tau\leq t}\langle \tau\rangle^{d-2+a}\|\nabla_x Y_{\tau,t}\|_{L^\infty_{x,v}})\int_{s}^{t}(\tau-s)\|\nabla_x E(\tau)\|_{L^\infty_{x,v}}d\tau\\
\leq&c_0\|(\rho,U)\|_{S^{{\varepsilon_0}}_a}(1+\sup_{0\leq\tau\leq t}\langle \tau\rangle^{d-2+a}\|\nabla_x Y_{\tau,t}\|_{L^\infty_{x,v}})
\int_{s}^{t}\frac{1}{\langle\tau\rangle^{d-1+a}}d\tau\\\leq&\frac{c_0c_1\|(\rho,U)\|_{S^{{\varepsilon_0}}_a}}{\langle s\rangle^{d-2+a}}(1+\sup_{0\leq\tau\leq t}\langle \tau\rangle^{d-2+a}\|\nabla_x Y_{\tau,t}\|_{L^\infty_{x,v}}).
\end{align*}
Then, the right hand side can be absorbed by the left hand side provided that $\|(\rho,U)\|_{S^{{\varepsilon_0}}_a}\leq\frac{1}{400c_0c_1d^2}$.\\
By the same method, we have
\begin{align*}
\|\nabla_v Y_{s,t}\|_{L^\infty_{x,v}}\leq & \int_{s}^{t}|\tau-s|\|\nabla_x E(\tau)\|_{L^\infty}\left(|\tau|+\|\nabla_vY_{\tau,t}\|_{L^\infty_{x,v}}\right)d\tau
\\	\leq & c_0c_1\|(\rho,U)\|_{S^{{\varepsilon_0}}_a} \left(\frac{1}{\langle s\rangle^{d-3+a}}+\sup_{0\leq\tau\leq t}\langle \tau\rangle^{d-3+a}\|\nabla_v Y_{\tau,t}\|_{L^\infty_{x,v}}\right).
\end{align*}
Hence, 
\begin{equation}\label{nabla x Y estimate}
\sup_{0\leq s\leq t}\langle s\rangle^{d-2+a}\left(	\|Y_{s,t}\|_{L^\infty_{x,v}}+\|\nabla_x Y_{s,t}\|_{L^\infty_{x,v}}\right)+	\sup_{0\leq s\leq t}\langle s\rangle^{d-3+a}\|\nabla_v Y_{s,t}\|_{L^\infty_{x,v}}\leq 2c_0c_1 \|(\rho,U)\|_{S^{{\varepsilon_0}}_a},
\end{equation}
provided $\|(\rho,U)\|_{S^{{\varepsilon_0}}_a}\leq\frac{1}{400c_0c_1d^2}. $\\
\textbf{(2)} By \eqref{E estimate by rho}, \eqref{nabla x Y estimate},  one obtains
\begin{align*}
\sup_{\alpha}	\frac{	\|\delta_{\alpha}^x	\nabla_xY_{s,t}\|_{L^\infty_{x,v}}}{|\alpha|^a}\leq& 	\sup_{\alpha}\frac{1}{|\alpha|^a} \left[\int_{s}^{t}(\tau-s)\|\nabla_x E(\tau)\|_{L^{\infty}}\|\delta_{\alpha}^x\nabla_x Y_{\tau,t}\|_{L^\infty_{x,v}}d\tau\right.\\&
\left.+\int_{s}^{t}(\tau-s)\|\nabla_x E(\tau)\|_{\dot{B}^a_{\infty,\infty}} \left(|\alpha|+\|\delta_\alpha^x Y_{\tau,t}\|_{L^\infty_{x,v}}\right)^a\left(1+\|\nabla_x Y_{\tau,t}\|_{L^\infty_{x,v}}\right)d\tau\right]\\\leq&  \int_{s}^{t}(\tau-s)\|\nabla_x E(\tau)\|_{L^{\infty}}d\tau\sup_{0\leq\tau\leq t}\langle \tau\rangle^{d-2+a}	\sup_{\alpha} \frac{\|\delta_{\alpha}^x\nabla_x Y_{\tau,t}\|_{L^\infty_{x,v}}}{|\alpha|^a}\\&
+\int_{s}^{t}(\tau-s)\|\nabla_x E(\tau)\|_{\dot{B}^a_{\infty,\infty}} \left(1+\|\nabla_x Y_{\tau,t}\|_{L^\infty_{x,v}}\right)^{1+a}d\tau \\\leq&  c_0c_1 \frac{\|(\rho,U)\|_{S^{{\varepsilon_0}}_a}}{\langle s\rangle^{d-2+a}}\sup_{0\leq\tau\leq t}\langle \tau\rangle^{d-2+a}	\sup_{\alpha} \frac{\|\delta_{\alpha}^x\nabla_x Y_{\tau,t}\|_{L^\infty_{x,v}}}{|\alpha|^a}
+c_0c_1\frac{\|(\rho,U)\|_{S^{{\varepsilon_0}}_a}}{\langle s\rangle^{d-2+a}}.
\end{align*}
Similarly,
\begin{equation}\label{delta}
\begin{split}
	\sup_{\alpha}	\frac{	\|\delta_{\alpha}^v	\nabla_vY_{s,t}\|_{L^\infty_{x,v}}}{|\alpha|^{a-\delta}}&\leq 	\sup_{\alpha}\frac{1}{|\alpha|^{a-\delta}} \left(\int_{s}^{t}(\tau-s)\|\nabla_x E(\tau)\|_{L^{\infty}}\|\delta_{\alpha}^v\nabla_v Y_{\tau,t}\|_{L^\infty_{x,v}}d\tau\right.\\&
	\left.+\int_{s}^{t}(\tau-s)\|\nabla_x E(\tau)\|_{\dot{B}^{a-\delta}_{\infty,\infty}} \left(|\alpha\|\tau|+\|\delta_\alpha^v Y_{\tau,t}\|_{L^\infty_{x,v}}\right)^{a-\delta}\left(|\tau|+\|\nabla_v Y_{\tau,t}\|_{L^\infty_{x,v}}\right)d\tau\right) \\&\leq  c_0c_1\frac{\|(\rho,U)\|_{S^{{\varepsilon_0}}_a}}{\langle s\rangle^{d+a-2}}\sup_{0\leq\tau\leq t}\langle \tau\rangle^{d-3+\delta}	\sup_{\alpha} \frac{\|\delta_{\alpha}^v\nabla_v Y_{\tau,t}\|_{L^\infty_{x,v}}}{|\alpha|^{a-\delta}}
	+c_0c_1\frac{\|(\rho,U)\|_{S^{{\varepsilon_0}}_a} }{\delta\langle s\rangle^{d-3+\delta}}.
\end{split}
\end{equation}
These imply
\begin{align}
&\sup_{0\leq s\leq t}\left(	\langle s\rangle^{d-2+a}\sup_{\alpha}	\frac{	\|\delta_{\alpha}^x	\nabla_xY_{s,t}\|_{L^\infty_{x,v}}}{|\alpha|^a}+\delta\langle s\rangle^{d-3+\delta}\sup_{\alpha}	\frac{	\|\delta_{\alpha}^v	\nabla_vY_{s,t}\|_{L^\infty_{x,v}}}{|\alpha|^{a-\delta}}\right)\lesssim_a\|(\rho,U)\|_{S^{{\varepsilon_0}}_a}, \label{deltanabla x Y estimate}
\end{align}
provided $\|(\rho,U)\|_{S^{{\varepsilon_0}}_a}\leq\frac{1}{400c_0c_1d^2}$.\\
Combining this with \eqref{nabla x Y estimate}, one gets \eqref{results of Y small}. By the same argument, we also obtain \eqref{results of W small}. Thus we choose $\widetilde{\varepsilon_0}=\min\{C_A^{-3},\frac{1}{400c_0c_1d^2}\}$ to finish the proof.
\end{proof}
\begin{remark}
For $d\geq4$, using the same method in  \eqref{delta}, we obtain
\begin{align}\label{d4}
\sup_{0\leq s\leq t}\langle s\rangle^{d-3}\sup_{\alpha}	\frac{	\|\delta_{\alpha}^v	\nabla_vY_{s,t}\|_{L^\infty_{x,v}}}{|\alpha|^{a}}\lesssim_a\|(\rho,U)\|_{S^{{\varepsilon_0}}_a} .
\end{align}
\end{remark}
\begin{proposition}\label{proposition of higher derivative estimates Y W}  Let $a\in (\frac{\sqrt{5}-1}{2},1)$,  $\widetilde{\varepsilon_0}$ be as in Proposition \ref{estimates about Y and W}. Assume that $\|(\rho,U)\|_{S^{{\varepsilon_0}}_a}\leq\varepsilon_0\leq \widetilde{\varepsilon_0}$  and $\|(\rho, U)\|_{\gamma,b}\leq \mathbf{c}$ for some $\gamma>0$ and $b\in (0,1)$. Then,  for $j=1,...,\gamma$ and $\delta\in (0,b)$, we have 
\begin{align*}
&	\sup_{0\leq s\leq t}\langle s\rangle^{d-2+b}	\|	\nabla_vY_{s,t}\|_{L^\infty_{x,v}}+\langle s\rangle^{d-1+b}	\|	\nabla_vW_{s,t}\|_{L^\infty_{x,v}}\lesssim_{\mathbf{c},b}1,\\
&\sum_{j=1}^{\gamma}\sup_{0\leq s\leq t}\langle s\rangle^{d-2}\sup_{\alpha}	\frac{	\|\delta_{\alpha}^v	\nabla_v^jY_{s,t}\|_{L^\infty_{x,v}}}{|\alpha|^{b}}+\langle s\rangle^{d-1}\sup_{\alpha}	\frac{	\|\delta_{\alpha}^v	\nabla_v^jW_{s,t}\|_{L^\infty_{x,v}}}{|\alpha|^{b}}\lesssim_{\mathbf{c},b}1,\\&\sup_{0\leq s\leq t}\delta\langle s\rangle^{d-3+\delta}\sup_{\alpha}	\frac{	\|\delta_{\alpha}^v	\nabla_v^{\gamma+1}Y_{s,t}\|_{L^\infty_{x,v}}}{|\alpha|^{b-\delta}}+\langle s\rangle^{d-2}\sup_{\alpha}	\frac{	a\delta_{\alpha}^v	\nabla_v^{\gamma+1}W_{s,t}\|_{L^\infty_{x,v}}}{|\alpha|^{b}}\lesssim_{\mathbf{c},b}1,\\&\sum_{k=1}^{\gamma+1}\sup_{0\leq s\leq t}\langle s\rangle^{d-1+b}\sup_{\alpha}	\frac{	\|\delta_{\alpha}^x	\nabla_x^kY_{s,t}\|_{L^\infty_{x,v}}}{|\alpha|^{b}}+\langle s\rangle^{d+b}\sup_{\alpha}	\frac{	\|\delta_{\alpha}^x	\nabla_x^kW_{s,t}\|_{L^\infty_{x,v}}}{|\alpha|^{b}}\lesssim_{\mathbf{c},b}1,
\end{align*}
provided that  \begin{align}\label{A assump high}
	\sum_{j=2}^{\gamma+1}\sup_{|r|\leq 1}	\left|{A^{(j)}(r)}\right|\leq C_{A,\gamma}.
\end{align}
\end{proposition}
\begin{proof}  
	Note that for  $0\leq j\leq \gamma$ and $\|(\rho,U)\|_{S^{{\varepsilon_0}}_a}\leq \varepsilon_0$, we have 
	\begin{align*}
		\|\nabla^j(A(U))\|_{\dot{B}^a_{p,\infty}}
		\leq&\sum_{k=1}^{j}|A^{(k)}(U)|_{L^\infty}\sum_{\substack{m_1,...,m_k\geq 1\\m_1+..+m_k=j}}\left\|\prod_{i=1}^k \nabla^{m_i} U\right\|_{\dot{B}^a_{p,\infty}}+\sum_{k=1}^{j}|A^{(k)}(U)|_{\dot{B}^a_{p,\infty}}\sum_{\substack{m_1,...,m_k\geq 1\\m_1+..+m_k=j}}\left\|\prod_{i=1}^k \nabla^{m_i} U\right\|_{L^\infty}\\
		\leq& C_{A,\gamma}\left(\sum_{\substack{m_1,...,m_k\geq 1\\m_1+..+m_k=j}}\left\|\prod_{i=1}^k \nabla^{m_i} U\right\|_{\dot{B}^a_{p,\infty}}+\sum_{\substack{m_1,...,m_k\geq 1\\m_1+..+m_k=j}}\left\|\prod_{i=1}^k \nabla^{m_i} U\right\|_{L^\infty}\right)\lesssim_{\mathbf{c},C_{A,\gamma}}1.
	\end{align*}
	 Combining this with Lemma \ref{varrho},  there exists a constant $c_3=c_3(d,b,b',C_{A,\gamma},\mathbf{c})$ such that
\begin{equation}	\label{E estimate by rho'}		
\sup_{\tau\geq 0}\sum_{j=0}^{\gamma+1}	\langle \tau \rangle^{d+\min\{j+1,\gamma+b\}}\|\nabla_x^j E(\tau)\|_{L^\infty}	+\langle \tau\rangle^{d+\min\{j+b'+1,\gamma+b\}}\|\nabla_x^j E(\tau)\|_{\dot{B}^{b'}_{\infty,\infty}}
\lesssim_{b,b'}  \|\rho+A(U)\|_{\gamma,b} \leq c_3,
\end{equation} 
for any $b'\in (0,b]$. \\	
\textbf{(1)}	Let $c_1$ be in \eqref{c1}. As the proof of \eqref{nabla x Y estimate}, we have 
\begin{align}\label{a1}
\begin{split}
	\|\nabla_v Y_{s,t}\|_{L^\infty_{x,v}}\leq & \int_{s}^{t}|\tau-s|\|\nabla_x E(\tau)\|_{L^\infty}|\tau|d\tau+ \int_{s}^{t}|\tau-s|\|\nabla_x E(\tau)\|_{L^\infty}\||\nabla_vY_{\tau,t}\|_{L^\infty_{x,v}}d\tau
	\\	
	\leq&c_3\int_{s}^{t}\frac{|\tau|^2}{\langle \tau\rangle^{d+b+1}}d\tau+c_0	\left(\sup_{s\leq \tau\leq t}\langle \tau\rangle^{d-2+b}	\|	\nabla_vY_{\tau,t}\|_{L^\infty_{x,v}}\right)\int_{s}^{t}\frac{|\tau|\|(\rho,U)\|_{S^{{\varepsilon_0}}_a}}{\langle \tau\rangle^{d-2+b}\langle \tau\rangle ^{d+a}}d\tau\\
	\leq&c_3c_1\frac{1}{\langle s\rangle^{d+b-2}}+c_0c_1\|(\rho,U)\|_{S^{{\varepsilon_0}}_a}\left(\sup_{s\leq \tau\leq t}\langle \tau\rangle^{d-2+b}	\|	\nabla_vY_{\tau,t}\|_{L^\infty_{x,v}}\right)\frac{1}{\langle s\rangle^{d+b-2}}.
\end{split}
\end{align}
Since $\|(\rho,U)\|_{S^{{\varepsilon_0}}_a}\leq\frac{1}{400_0c_1d^2}$, then we have
\begin{align*}
\sup_{0\leq s\leq t}\langle s\rangle^{d-2+b}	\|	\nabla_vY_{s,t}\|_{L^\infty_{x,v}}\leq 2c_3c_1\mathbf{c}\lesssim_{\mathbf{c},b,C_{A,\gamma}}1.
\end{align*}
Similarily, we obtain
\begin{align*}
\sup_{0\leq s\leq t}\langle s\rangle^{d-1+b}	\|	\nabla_vW_{s,t}\|_{L^\infty_{x,v}}\lesssim_{\mathbf{c},b,C_{A,\gamma}}1.
\end{align*}
\textbf{(2)}  As  \eqref{a1}, it is easy to check that 
\begin{align}\label{z1}
\sum_{j=0}^{\gamma+1}\|\nabla_v^j \left(Y_{s,t},W_{s,t}\right)\|_{L^\infty_{x,v}}+\sum_{j=0}^{\gamma}\sup_\alpha\frac{\|\delta_{\alpha}^v\nabla_v^j \left(Y_{s,t},W_{s,t}\right)\|_{L^\infty_{x,v}}}{|\alpha|^b}\lesssim_{\mathbf{c},b,C_{A,\gamma},t}1,
\end{align} for any $0\leq s\leq t<\infty$.\\
By \eqref{z1}, one has  
\begin{align*}
\sup_\alpha\frac{\|	\delta_{\alpha}^v\nabla_v^{k+1} Y_{s,t}\|_{L^\infty_{x,v}}}{|\alpha|^{b-\delta}}\leq&  \int_{s}^{t}(\tau-s)\|\nabla E(\tau)\|_{L^\infty}
\sup_\alpha\frac{\|	\delta_{\alpha}^v\nabla_v^{k+1} Y_{\tau,t}\|_{L^\infty_{x,v}}}{|\alpha|^{b-\delta}} d\tau\\&+ \int_{s}^{t}(\tau-s)\|\nabla E(\tau)\|_{B^{b-\delta}_{\infty,\infty}}\left(|\tau|+ \|\nabla_v Y_{\tau,t}\|_{L^\infty_{x,v}}\right)^{b-\delta}
\|\nabla_v^{k+1} Y_{\tau,t}\|_{L^\infty_{x,v}}d\tau
\\&+C(\mathbf{c})\sum_{j=1}^{k}\int_{s}^{t} (\tau-s)\|\nabla^{j+1} E(\tau)\|_{B^{b-\delta}_{\infty,\infty}}
(|\tau|+1)^{j+1+b-\delta} d\tau
\\&+C(\mathbf{c})\sum_{j=1}^{k}\int_{s}^{t} (\tau-s)\|\nabla^{j+1} E(\tau)\|_{L^\infty}
(|\tau|+1)^{j+1} d\tau\\\overset{\eqref{E estimate by rho}, \eqref{E estimate by rho'}}\leq&c_0 \int_{s}^{t}\frac{\|(\rho,U)\|_{S^{{\varepsilon_0}}_a}}{\langle \tau\rangle^{d-1+a}}
\sup_\alpha\frac{\|	\delta_{\alpha}^v\nabla_v^{k+1} Y_{\tau,t}\|_{L^\infty_{x,v}}}{|\alpha|^{b-\delta}}d\tau
+\frac{C(\mathbf{c},b,C_{A,\gamma})}{\langle s\rangle^{d-2+\delta}}+\frac{C(\mathbf{c},b,C_{A,\gamma})\textbf{1}_{k=\gamma}}{\delta\langle s\rangle^{d-3+\delta}}.
\end{align*}
One derives that 
\begin{align*}\omega(s)\sup_\alpha\frac{\|	\delta_{\alpha}^v\nabla_v^{k+1} Y_{s,t}\|_{L^\infty_{x,v}}}{|\alpha|^{b-\delta}}\leq \frac{1}{4}	B(t)+C(\mathbf{c},b,C_{A,\gamma}),
\end{align*}
where 
\begin{equation*}
B(t)=	\sup_{0\leq\tau\leq t}	\omega(\tau)\sup_\alpha\frac{\|	\delta_{\alpha}^v\nabla_v^{k+1} Y_{\tau,t}\|_{L^\infty_{x,v}}}{|\alpha|^{b-\delta}},~~~	\omega(s)=\begin{cases} 	\langle s\rangle^{d-2}~~~\text{if}~~k\leq \gamma-1,\\ 	\delta\langle  s\rangle^{d-3+\delta}~~\text{if}~~k= \gamma.\end{cases}
\end{equation*}
This implies 
\begin{equation*}
B(t)\lesssim_{\mathbf{c},b,C_{A,\gamma}} 1. 
\end{equation*}
In particular, 
\begin{equation}
\sum_{k=0}^{\gamma-1}		\sup_{0\leq\tau\leq t}	\langle\tau\rangle^{d-2}\sup_\alpha\frac{\|	\delta_{\alpha}^v\nabla_v^{k+1} Y_{\tau,t}\|_{L^\infty_{x,v}}}{|\alpha|^{b}}+\delta \sup_{0\leq\tau\leq t}	\langle\tau\rangle^{d-3+\delta}\sup_\alpha\frac{\|	\delta_{\alpha}^v\nabla_v^{\gamma+1} Y_{\tau,t}\|_{L^\infty_{x,v}}}{|\alpha|^{b-\delta}}\lesssim_{\mathbf{c},b,C_{A,\gamma}} 1.
\end{equation}
Similarly, one also has
\begin{align*}
\sup_\alpha\frac{\|	\delta_{\alpha}^x\nabla_x^{k+1} Y_{s,t}\|_{L^\infty_{x,v}}}{|\alpha|^{b}}\leq & \int_{s}^{t}(\tau-s)\|\nabla E(\tau)\|_{L^\infty}
\sup_\alpha\frac{\|	\delta_{\alpha}^x\nabla_x^{k+1} Y_{\tau,t}\|_{L^\infty_{x,v}}}{|\alpha|^{b}} d\tau\\&+ \int_{s}^{t}(\tau-s)\|\nabla E(\tau)\|_{B^{a}_{\infty,\infty}}
\|\nabla_x^{k+1} Y_{\tau,t}\|_{L^\infty_{x,v}}d\tau
\\&+C(\mathbf{c})\sum_{j=1}^{k}\int_{s}^{t} (\tau-s)\left(\|\nabla^{j+1} E(\tau)\|_{B^{b}_{\infty,\infty}}+\|\nabla^{j+1} E(\tau)\|_{L^{\infty}}\right) d\tau\\\overset{\eqref{E estimate by rho}, \eqref{E estimate by rho'}}\leq& c_0 \int_{s}^{t}\frac{\|(\rho,U)\|_{S^{{\varepsilon_0}}_a}}{\langle \tau\rangle^{d-1+a}}
\sup_\alpha\frac{\|	\delta_{\alpha}^x\nabla_x^{k+1} Y_{\tau,t}\|_{L^\infty_{x,v}}}{|\alpha|^{b}} +\frac{C(\mathbf{c},b,C_{A,\gamma})}{\langle s\rangle^{d-1+b}}.
\end{align*}
This gives 
\begin{equation}\label{z7}
\sup_{0\leq s\leq t}	\langle s\rangle^{d-1+b}\sup_\alpha\frac{\|	\delta_{\alpha}^x\nabla_x^{k+1} Y_{s,t}\|_{L^\infty_{x,v}}}{|\alpha|^{b}}\lesssim_{\mathbf{c},b,C_{A,\gamma}} 1. 
\end{equation}
By the same argument, we obtain
\begin{equation*}
\sup_{0\leq s\leq t}	\langle s\rangle^{d-1-\mathbf{1}_{k=\gamma}}\sup_{\alpha}	\frac{	\|\delta_{\alpha}^v	\nabla_v^{k+1}W_{s,t}\|_{L^\infty_{x,v}}}{|\alpha|^{a}}+\sup_{0\leq s\leq t}	\langle s\rangle^{d+a}\sup_\alpha\frac{\|	\delta_{\alpha}^x\nabla_x^{k+1} W_{s,t}\|_{L^\infty_{x,v}}}{|\alpha|^{a}}\lesssim_{\mathbf{c},b,C_{A,\gamma}}1.
\end{equation*}
Therefore, the proof is complete. 
\end{proof}
\begin{remark}
For $d\geq4$, we have 
\begin{align*}
\sup_{0\leq s\leq t}\langle s\rangle^{d-3}\sup_{\alpha}	\frac{	\|\delta_{\alpha}^v	\nabla_v^{\gamma+1}Y_{s,t}\|_{L^\infty_{x,v}}}{|\alpha|^{b}}\lesssim_{\mathbf{c},b,C_{A,\gamma}}1.
\end{align*}
\end{remark}
\begin{proposition}\label{Y W difference}
Let $a\in (\frac{\sqrt{5}-1}{2},1)$,  $\widetilde{\varepsilon_0}$ be as in Proposition \ref{estimates about Y and W},  $\rho_1,\rho_2$ be such that $\|(\rho_1,U_1)\|_{S^{{\varepsilon_0}}_a},\|(\rho_2,U_2)\|_{S^{{\varepsilon_0}}_a}\leq\varepsilon_0\leq \widetilde{\varepsilon_0}$ .  Let $(X_{s,t}^1,V_{s,t}^1)$, $(X_{s,t}^2,V_{s,t}^2)$  be solutions to \eqref{X and V} associated to $E^1(s,x)=\nabla_x(1-\Delta)^{-1}(\rho_1+A(U_1))(x)$ and $E^2(s,x)=\nabla_x(1-\Delta)^{-1}(\rho_2+A(U_2))(x)$. Let  $(Y_{s,t}^1,W_{s,t}^1)$, $(Y_{s,t}^2,W_{s,t}^2)$ be such that 
\begin{equation}\label{XiYiViWi def}
X_{s,t}^i(x,v)=x-(t-s)v+Y_{s,t}^i(x-vt,v),V_{s,t}^i(x,v)=v+W_{s,t}^i(x-vt,v),~~i=1,2.
\end{equation}
Then for $\mathbf{Y}_{s,t}=Y_{s,t}^1-Y_{s,t}^2,\mathbf{W}_{s,t}=W_{s,t}^1-W_{s,t}^2$, we have
\begin{align}
&\sup_{0\leq s\leq t}\langle s\rangle^{d-2+a}\left(	\|\mathbf{Y}_{s,t}\|_{L^\infty_{x,v}}+\sup_\alpha\frac{	\|\delta_{\alpha}^x\mathbf{Y}_{s,t}\|_{L^\infty_{x,v}}}{|\alpha|^a}\right)+\sup_{0\leq s\leq t}\langle s\rangle^{d-2}\sup_\alpha\frac{	\|\delta_{\alpha}^v\mathbf{Y}_{s,t}\|_{L^\infty_{x,v}}}{|\alpha|^a}\lesssim_{a} \|(\rho_1-\rho_2,U_1-U_2)\|_{S^{{\varepsilon_0}}_a},\label{results of Y12 small}
\\&	\sup_{0\leq s\leq t}\langle s\rangle^{d-1+a}\left(	\|\mathbf{W}_{s,t}\|_{L^\infty_{x,v}}+\sup_\alpha\frac{	\|\delta_{\alpha}^v\mathbf{W}_{s,t}\|_{L^\infty_{x,v}}}{|\alpha|^a}\right)+\sup_{0\leq s\leq t}	\langle s\rangle^{d-1}\sup_\alpha\frac{	\|\delta_{\alpha}^v\mathbf{W}_{s,t}\|_{L^\infty_{x,v}}}{|\alpha|^a}\lesssim_{a} \|(\rho_1-\rho_2,U_1-U_2)\|_{S^{{\varepsilon_0}}_a}.\label{results of W12 small}
\end{align}
\end{proposition}
\begin{proof}Let $c_0,c_1$  be the constant mentioned in the proof of Theorem \ref{estimates about Y and W}. We have 
\begin{align*}
&Y_{s,t}^k(x,v)=\int_{s}^{t}(\tau-s)E^k(\tau,x+\tau v+Y_{\tau,t}^k(x,v))d\tau,\\&
W_{s,t}^k(x,v)=-\int_{s}^{t}E^k(\tau,x+\tau v+Y_{\tau,t}^k(x,v))d\tau,~~~k=1,2.
\end{align*}
From $\|(\rho_k,U_k)\|_{S^{{\varepsilon_0}}_a}\leq \varepsilon_0$, one has $\|U_k\|_a\leq \varepsilon_0^{\frac{2}{3}}$, for $k=1,2$. Combining this with  Assumption \ref{A assump}, we obtain that
\begin{align*}
&	\|A(U_1)-A(U_2)\|_{\dot{B}^a_{\infty,\infty}}=\left\|(U_1-U_2)\int _0^1 A'(\varpi U_1+(1-\varpi U_2))d\varpi\right\|_{\dot{B}^a_{\infty,\infty}}\\
&	\leq\|U_1-U_2\|_{L^\infty}\left\|\int _0^1 A'(\varpi U_1+(1-\varpi U_2))d\varpi\right\|_{\dot{B}^a_{\infty,\infty}}\\
	&+\|U_1-U_2\|_{\dot{B}^a_{\infty,\infty}}\left\|\int _0^1 A'(\varpi U_1+(1-\varpi U_2))d\varpi\right\|_{L^\infty}\\
	&\leq\|U_1-U_2\|_{L^\infty}\left(\max_{|r|\leq\varepsilon_0^{2/3}}|A'(r)|\right)^{1-a}\left(\varepsilon_0^{\frac{2}{3}}\max_{|r|\leq\varepsilon_0^{2/3}}|A''(r)|\right)^a+\|U_1-U_2\|_{\dot{B}^a_{\infty,\infty}}\max_{|r|\leq\varepsilon_0^{2/3}}|A'(r)|\\
	&\leq C_A\varepsilon_0^{\frac{2}{3}}\left(\|U_1-U_2\|_{L^\infty}+\|U_1-U_2\|_{\dot{B}^a_{\infty,\infty}}\right).
\end{align*}
Therefore,
\begin{align}\label{E12 estimate by rho}
\begin{split}
&\|(E^1-E^2)(\tau)\|_{L^\infty}+\|\nabla_x (E^1-E^2)(\tau)\|_{L^\infty}	+\|\nabla_x (E^1-E^2)(\tau)\|_{\dot{B}^a_{\infty,\infty}}
\\&\leq c_0(a)\|\rho_1(\tau)-\rho_2(\tau)+A(U_1)(\tau)-A(U_2)(\tau)\|_{\dot{B}^a_{\infty,\infty}}\\
&\leq c_0(a)\left(\|\rho_1(\tau)-\rho_2(\tau)\|_a+C_A\varepsilon_0^{\frac{2}{3}}\left(\|U_1(\tau)-U_2(\tau)\|_{L^\infty}+\|U_1(\tau)-U_2(\tau)\|_{\dot{B}^a_{\infty,\infty}}\right)\right)\\&\leq c_0(a)\frac{\|(\rho_1-\rho_2,U_1-U_2)\|_{S_a^{\varepsilon_0}}}{\langle \tau\rangle^{d+a}}.
\end{split}
\end{align} 
These imply
\begin{align*}
\mathbf{Y}_{s,t}(x,v)=&\int_{s}^{t}(\tau-s)(E^1-E^2)(\tau,x+\tau v+Y_{\tau,t}^1(x,v))d\tau\\&+\int_{s}^{t}(\tau-s)\left(E^2(\tau,x+\tau v+Y_{\tau,t}^1(x,v))-E^2(\tau,x+\tau v+Y_{\tau,t}^2(x,v))\right)d\tau.
\end{align*}
Hence we know that 
\begin{align*}
\|\mathbf{Y}_{s,t}\|_{L^\infty_{x,v}}&\leq \int_{s}^{t}(\tau-s)\|(E^1-E^2)(\tau)\|_{L^\infty}d\tau+\int_{s}^{t}(\tau-s)\|\nabla E^2(\tau)\|_{L^\infty} 	\|\mathbf{Y}_{\tau,t}\|_{L^\infty_{x,v}}d\tau\\ &\leq  c_0 \int_{s}^{t}\frac{\|(\rho_1-\rho_2,U_1-U_2)\|_{S_a^{\varepsilon_0}}}{\langle \tau\rangle^{d-1+a}}d\tau+ c_0\int_{s}^{t}\frac{\|(\rho_2,U_2)\|_{S^{{\varepsilon_0}}_a}}{\langle \tau\rangle^{d-1+a}}\|\mathbf{Y}_{\tau,t}\|_{L^\infty_{x,v}}d\tau
\\ &\leq  c_0c_1 \frac{\|(\rho_1-\rho_2,U_1-U_2)\|_{S_a^{\varepsilon_0}}}{\langle s\rangle^{d-2+a}}+c_0c_1\frac{\|(\rho_2,U_2)\|_{S^{{\varepsilon_0}}_a}}{\langle s\rangle^{2d-4+2a}} \sup_{0\leq\tau\leq t} \langle \tau\rangle^{d-2+a} \|\mathbf{Y}_{\tau,t}\|_{L^\infty_{x,v}}.
\end{align*}
Since $\|(\rho_2,U_2)\|_{S^{{\varepsilon_0}}_a}\leq\frac{1}{400_0c_1d^2} $,
\begin{align*}
\sup_{0\leq s\leq t}\langle s\rangle^{d-2+a}	\|\mathbf{Y}_{s,t}\|_{L^\infty_{x,v}}\leq2  c_0c_1\|(\rho_1-\rho_2,U_1-U_2)\|_{S_a^{\varepsilon_0}}\lesssim_a\|(\rho_1-\rho_2,U_1-U_2)\|_{S_a^{\varepsilon_0}}.
\end{align*}
For the next terms, we have
\begin{align*}
\sup_\alpha\frac{	\|\delta_{\alpha}^x\mathbf{Y}_{s,t}\|_{L^\infty_{x,v}}}{|\alpha|^a}\leq& \int_{s}^{t}(\tau-s)\|(E^1-E^2)(\tau)\|_{B^{a}_{\infty,\infty}}d\tau+\int_{s}^{t}(\tau-s)\|\nabla E^2(\tau)\|_{B^a_{\infty,\infty}} 	\|\mathbf{Y}_{\tau,t}\|_{L^\infty_{x,v}}d\tau\\&+\int_{s}^{t}(\tau-s)\|\nabla E^2(\tau)\|_{L^\infty} \sup_\alpha\frac{	\|\delta_{\alpha}^x\mathbf{Y}_{\tau,t}\|_{L^\infty_{x,v}}}{|\alpha|^a}d\tau
\\ \leq& c_0c_1\frac{\|(\rho_1-\rho_2,U_1-U_2)\|_{S_a^{\varepsilon_0}}}{\langle s\rangle^{d-2+a}}+2c_0^2c_1^2\frac{\|(\rho_2,U_2)\|_{S^{{\varepsilon_0}}_a}\|(\rho_1-\rho_2,U_1-U_2)\|_{S_a^{\varepsilon_0}}}{\langle s\rangle^{d-2+a}} \\& +c_0c_1\frac{\|(\rho_2,U_2)\|_{S^{{\varepsilon_0}}_a}}{\langle s\rangle^{2(d-2+a)}} \sup_{0\leq \tau\leq t}\langle \tau\rangle^{d-2+a}\sup_\alpha\frac{	\|\delta_{\alpha}^x\mathbf{Y}_{\tau,t}\|_{L^\infty_{x,v}}}{|\alpha|^a}\\
\leq& c_0c_1\frac{\|(\rho_1-\rho_2,U_1-U_2)\|_{S_a^{\varepsilon_0}}}{\langle s\rangle^{d-2+a}} +c_0c_1\frac{\|(\rho_2,U_2)\|_{S^{{\varepsilon_0}}_a}}{\langle s\rangle^{2(d-2+a)}} \sup_{0\leq \tau\leq t}\langle \tau\rangle^{d-2+a}\sup_\alpha\frac{	\|\delta_{\alpha}^x\mathbf{Y}_{\tau,t}\|_{L^\infty_{x,v}}}{|\alpha|^a},
\end{align*}
and 
\begin{align*}
\sup_\alpha\frac{	\|\delta_{\alpha}^v\mathbf{Y}_{s,t}\|_{L^\infty_{x,v}}}{|\alpha|^a}\leq& \int_{s}^{t}(\tau-s)\|(E^1-E^2)(\tau)\|_{B^{a}_{\infty,\infty}}\langle \tau\rangle^ad\tau+\int_{s}^{t}(\tau-s)\|\nabla E^2(\tau)\|_{B^a_{\infty,\infty}} 	\|\mathbf{Y}_{\tau,t}\|_{L^\infty_{x,v}}\langle \tau\rangle^ad\tau\\&+\int_{s}^{t}(\tau-s)\|\nabla E^2(\tau)\|_{L^\infty} \sup_\alpha\frac{	\|\delta_{\alpha}^v\mathbf{Y}_{\tau,t}\|_{L^\infty_{x,v}}}{|\alpha|^a}d\tau
\\\leq & c_0c_1 \frac{\|(\rho_1-\rho_2,U_1-U_2)\|_{S_a^{\varepsilon_0}}}{\langle s\rangle^{d-2}}+2c_0^2c_1^2\frac{\|(\rho_2,U_2)\|_{S^{{\varepsilon_0}}_a}\|(\rho_1-\rho_2,U_1-U_2)\|_{S_a^{\varepsilon_0}}}{\langle s\rangle^{d-2}}\\&
 +c_0c_1\frac{\|(\rho_2,U_2)\|_{S^{{\varepsilon_0}}_a}}{\langle s\rangle^{2(d-2)+a}} \sup_{0\leq \tau\leq t}\langle \tau\rangle^{d-2}\sup_\alpha\frac{	\|\delta_{\alpha}^x\mathbf{Y}_{\tau,t}\|_{L^\infty_{x,v}}}{|\alpha|^a}
\\\leq & c_0c_1 \frac{\|(\rho_1-\rho_2,U_1-U_2)\|_{S_a^{\varepsilon_0}}}{\langle s\rangle^{d-2}} +c_0c_1\frac{\|(\rho_2,U_2)\|_{S^{{\varepsilon_0}}_a}}{\langle s\rangle^{2(d-2)+a}} \sup_{0\leq \tau\leq t}\langle \tau\rangle^{d-2}\sup_\alpha\frac{	\|\delta_{\alpha}^x\mathbf{Y}_{\tau,t}\|_{L^\infty_{x,v}}}{|\alpha|^a}.
\end{align*}
Therefore, we get \eqref{results of Y12 small} provided $\|(\rho_2,U_2)\|_{S^{{\varepsilon_0}}_a}\leq\frac{1}{400_0c_1d^2}$. Similarly, we also get \eqref{results of W12 small}. The proof is complete.
\end{proof}		
\begin{proposition}\label{proposition of psi and phi}  Let $a\in (\frac{\sqrt{5}-1}{2},1)$, $\widetilde{\varepsilon_0}$  be as in Proposition \ref{estimates about Y and W}. Let $(\rho,U)$ be such that $\|(\rho,U)\|_{S^{{\varepsilon_0}}_a}\leq\varepsilon_0\leq \widetilde{\varepsilon_0}$. Then, for any $0\leq s\leq t< \infty$, we have a $C^{1}$ map~
$
(x,v)\mapsto\Psi_{s,t}(x,v)
$ satisfying for all $x,v\in \mathbb{R}^d$:
\begin{equation}\label{X(x,Psi)}
X_{s,t}(x,\Psi_{s,t}(x,v))=x-(t-s)v,\qquad	
\end{equation}
and 
\begin{equation}\label{z3}
\langle s\rangle^{d-1+a}	\left(\left|	\Psi_{s,t}(x,v)-v\right|+	\left|	\nabla_x\Psi_{s,t}(x,v)\right|\right)+\langle s\rangle^{d-2+a}	\left|	\nabla_v\left(\Psi_{s,t}(x,v)-v\right)\right|\lesssim_a \|(\rho,U)\|_{S^{{\varepsilon_0}}_a}.
\end{equation}
\end{proposition}
\begin{proof}
Define
\begin{align*}
\Phi_{s,t}(x,v)=-\frac{1}{t-s}\int_{s}^{t}(\tau-s)E(\tau,x-(t-\tau)v+Y_{\tau,t}(x-vt,v))d\tau.
\end{align*}
Hence, it follows from \eqref{definition of X and V} that
\begin{align}\label{X Phi}
X_{s,t}(x,v)&=x-v(t-s)+Y_{s,t}(x-vt,v)=x-(t-s)(v+\Phi_{s,t}(x,v)).
\end{align}
For the first estimate, since  $\|\nabla_xY_{\tau,t}\|_{L^\infty_{x,v}}\lesssim 1$, so 
\begin{equation}
\|\Phi_{s,t}\|_{L^\infty_{x,v}}\lesssim\int_{s}^{t}\|E(\tau)\|_{L^\infty_x}d\tau\lesssim_a\frac{\|(\rho,U)\|_{S^{{\varepsilon_0}}_a}}{\langle s\rangle^{d+a-1}}
\end{equation}
and
\begin{equation}
\|\nabla_x\Phi_{s,t}\|_{L^\infty_{x,v}}\leq\frac{1}{t-s}\int_{s}^{t}(\tau-s)\|\nabla_xE(\tau)\|_{L^\infty_{x}}(1+\|\nabla_xY_{\tau,t}\|_{L^\infty_{x,v}})d\tau\lesssim_a\frac{\|(\rho,U)\|_{S^{{\varepsilon_0}}_a}}{\langle s\rangle^{d+a-1}},
\end{equation}
we have
\begin{equation}\label{Phi infty norm eatimates}
\sup_{0\leq s,t\leq \infty}\langle s\rangle^{d+a-1}\left(\|\Phi_{s,t}\|_{L^\infty_{x,v}}+\|\nabla_x \Phi_{s,t}\|_{L^\infty_{x,v}}\right)	\lesssim_a \|(\rho,U)\|_{S^{{\varepsilon_0}}_a}.
\end{equation}
Now one has
\begin{align*}
\|\nabla_v\Phi_{s,t}\|_{L^\infty_{x,v}}\leq&\frac{1}{t-s}\int_{s}^{t}(\tau-s)\|\nabla_xE(\tau)\|_{L^\infty}\left(t-\tau+t\|\nabla_xY_{\tau,t}||_{L^\infty_{x,v}}+\|\nabla_v Y_{\tau,t}\|_{L^\infty_{x,v}}\right)d\tau\\
\leq&\frac{c_0\|(\rho,U)\|_{S^{{\varepsilon_0}}_a}}{t-s}\int_{s}^{t}\frac{\tau-s}{\langle \tau\rangle^{d+a}}\left((t-\tau)+(t-\tau)+\tau+1\right)d\tau\\
\leq&4c_0\|(\rho,U)\|_{S^{{\varepsilon_0}}_a}\int_{s}^{t}\frac{1}{\langle \tau\rangle^{d+a-1}}d\tau.
\end{align*}
Let $c_0$, $c_1$	be in \eqref{E estimate by rho} and \eqref{c1}	by \eqref{results of Y small} and \eqref{E estimate by rho}. One obtains that 
$$
\|\nabla_v\Phi_{s,t}\|_{L^\infty_{x,v}}\leq \frac{4c_0c_1\|(\rho,U)\|_{S^{{\varepsilon_0}}_a}}{\langle s\rangle^{d-2+a}}\leq\frac{1}{100d^2},$$
provided $\|(\rho,U)\|_{S^{{\varepsilon_0}}_a}\leq\frac{1}{400c_0c_1d^2}$,
and it is clear that under the assumption $\|(\rho,U)\|_{S^{{\varepsilon_0}}_a}\leq\frac{1}{400c_0c_1d^2}$, the map $(x,v)\mapsto (x,v+\Phi_{s,t}(x,v))$ is a $C^1$ differomorphism. Thus, there exists a map 		$(x,v)\mapsto\Psi_{s,t}(x,v)$ satisfying \eqref{X(x,Psi)}. Moreover, combining \eqref{X(x,Psi)} with \eqref{X Phi}, we have 
$
\Psi_{s,t}(x,v)+\Phi_{s,t}(x,\Psi_{s,t}(x,v))=v.
$\\
Now we have 
\begin{align*}
&	\left|	\Psi_{s,t}(x,v)-v\right|\leq \|\Phi_{s,t}\|_{L^\infty_{x,v}}\lesssim_a\frac{\|(\rho,U)\|_{S^{{\varepsilon_0}}_a}}{\langle s\rangle^{d+a-1}},\\&
\left|	\nabla_x\Psi_{s,t}(x,v)\right|\leq \|\nabla_x\Phi_{s,t}\|_{L^\infty_{x,v}}+\|\nabla_v\Phi_{s,t}\|_{L^\infty_{x,v}}	|\nabla_x\Psi_{s,t}(x,v)|,\\&
\left|	\nabla_v\left(\Psi_{s,t}(x,v)-v\right)\right|\leq \|\nabla_v\Phi_{s,t}\|_{L^\infty_{x,v}}+\|\nabla_v\Phi_{s,t}\|_{L^\infty_{x,v}}\left(	|\nabla_v(\Psi_{s,t}(x,v)-v)|+1\right).
\end{align*}
Then, the right hand side can be absorbed by the left. Hence, we get
\begin{align*}
&	\left\|	\nabla_x\Psi_{s,t}(x,v)\right\|_{L^\infty_{x,v}}\leq2 \|\nabla_x\Phi_{s,t}\|_{L^\infty_{x,v}}\lesssim_a\frac{\|(\rho,U)\|_{S^{{\varepsilon_0}}_a}}{\langle s\rangle^{d+a-1}},\\
&		\left|	\nabla_v\left(\Psi_{s,t}(x,v)-v\right)\right|\leq3 \|\nabla_v\Phi_{s,t}\|_{L^\infty_{x,v}}\lesssim
_a\frac{\|(\rho,U)\|_{S^{{\varepsilon_0}}_a}}{\langle s\rangle^{d-2+a}}.
\end{align*}
These follow \eqref{z3}.
\end{proof}	
\begin{remark}
	For $a_1\in(0,1)$, if we also have $\|(\rho,U)\|_{S^{{\varepsilon_0}}_{a_1}}\lesssim 1$,
	\begin{equation}\label{z3'}
		\langle s\rangle^{d-2+a_1}	\left|	\nabla_v\left(\Psi_{s,t}(x,v)-v\right)\right|\lesssim_{a_1} \|(\rho,U)\|_{S^{{\varepsilon_0}}_{a_1}}.
	\end{equation}	
\end{remark}
\section{Contribution of the initial data}\label{Contribution of the initial data} 
For $h=h(x,v):\mathbb{R}^d_x\times\mathbb{R}^d_v\to \mathbb{R}$, define
\begin{equation}
\mathcal{I}_h(\rho,U)(t,x,v)= \int_{\mathbb{R}^{d}}h(X_{0,t}(x,v),V_{0,t}(x,v))dv.
\end{equation}
\begin{proposition}\label{I a estimate} Let $a\in (\frac{\sqrt{5}-1}{2},1)$, $\widetilde{\varepsilon_0}$ be as in Proposition \ref{estimates about Y and W}. Let  $(\rho,U)$ be such that $\|(\rho,U)\|_{S^{{\varepsilon_0}}_a}\leq\varepsilon_0\leq \widetilde{\varepsilon_0}$.  Then there holds
\begin{equation}\label{I norm}
\|\mathcal{I}_h(\rho,U)\|_{\sigma}\lesssim_a  \sum_{p=1,\infty}\|\mathcal{D}^\sigma(h)\|_{L^1_{x}L^p_v\cap L^1_{v}L^p_x},	
\end{equation}
for any $\sigma\in [a,1)$.
Moreover,  for $0<\delta_0<2-\sqrt{3}$, if $\|(\rho,U)\|_{l,1-\delta_0}\leq \mathbf{c}$ for some $l\geq0$, then  for any $b\in (0,1-\delta_0)$, we have
\begin{equation}\label{higher I norm}
\|\mathcal{I}_h(\rho,U)\|_{l+1,b}\lesssim_{\mathbf{c},a,b,C_{A,l}}\frac{1}{1-\delta_0-b} \sum_{i=1}^{l+1}\|\mathcal{D}^{b} (\nabla^i_{x,v}h)\|_{L^1_{x}L^p_v\cap L^1_{v}L^p_x},
\end{equation}
provided that $	\sum_{j=2}^{l+1}\sup_{|r|\leq 1}	\left|{A^{(j)}(r)}\right|\leq C_{A,l}.$
\end{proposition}
\begin{remark}
By \eqref{d4}, for $d\geq4$, we can obtain better result as follows
\begin{equation}\label{higher I norm2}
\|\mathcal{I}_h(\rho,U)\|_{l+1,1-\delta_0}\lesssim_{\mathbf{c},a,b,C_{A,l}} \sum_{i=1}^{l+1}\|\mathcal{D}^{1-\delta_0} (\nabla^i_{x,v}h)\|_{L^1_{x}L^p_v\cap L^1_{v}L^p_x}.
\end{equation}
\end{remark}
\begin{remark}
We can not only use $\|h\|_{L^1_{x}L^1_v\cap L^1_{x}L^\infty_v}$ to control $\|I_h(\rho,U)(t)\|_{L^\infty}$ for  $0\leq t\leq 1$.\\
For example we take $$h(x,v)=\frac{1}{|x|^{\theta}(|x|^2+|v|^2+1)^{2d}},~~ X_{0,t(x,v)}=x-tv,~~V_{0,t}(x,v)=v,~~U=0,$$ for $0<\theta <d$. And it is easy to check that for $t\in[0,1]$, we have $$\|h\|_{L^1_{x}L^1_v\cap L^1_{x}L^\infty_v}<\infty,~~\|h\|_{ L^1_{v}L^\infty_x}=\infty, ~\text{and}~~\|I_h(\rho,U)(t)\|_{L^\infty}\sim \frac{1}{t^\theta}.
$$
\end{remark}
\begin{proof} 
$\textbf{(1)}$	Firstly, we deal with $\|	\mathcal{I}_h(\rho,U)(t,x)\|_{L^p_x}$. Note that 
$
(	x,v)\mapsto (X_{0,t}(x,v),V_{0,t}(x,v))$ corresponding to the transport matrix
\begin{align*}
\mbox{$\mathfrak{A}(t,x,v):=$} \left[ \begin{array}{cc}
	I_d+\nabla_xY_{0,t}(x-vt,v) & 	-tI_d+\nabla_vY_{0,t}(x-vt,v) \\ 	\nabla_xW_{0,t}(x-vt,v) & 	I_d+\nabla_vW_{0,t}(x-vt,v)  \end{array} \right].
\end{align*}
Since we have 
\begin{align*}
\|\nabla_{x,v}Y_{0,t}(x-vt,v)\|_{L^\infty_{x,v}}+\|\nabla_{x,v}W_{0,t}(x-vt,v)\|_{L^\infty_{x,v}}\leq 2c_0c_1\|(\rho,U)\|_{S^{{\varepsilon_0}}_a}\leq  2c_0c_1\varepsilon_0\leq\frac{1}{400d^2},
\end{align*} and  $\|\det\mathfrak{A}^{-1}(t,x,v)\|_{L^\infty_{x,v}}\lesssim1$. Then
we directly obtain that
\begin{equation}\label{L^p t small}
\begin{split}
	\|	\mathcal{I}_h(\rho,U)(t,x)\|_{L^1_x}=&\left\|\int_{\mathbb{R}^d}h(x-vt+Y_{0,t}(x-vt,v),v+W_{0,t}(x-vt,v))dv\right\|_{L^1_x}\\
	=&\left\|\int_{\mathbb{R}^d}h(x,v)|\det\mathfrak{A}^{-1}(t,x,v)|dv\right\|_{L^1_x}\lesssim \|h\|_{L^1_{v}L^1_x}.
\end{split}
\end{equation}
For the $\|	\mathcal{I}_h(\rho,U)(t,x)\|_{L^\infty_x}$ with $t\leq 1$, changing variable as $\tilde{v}=v+W_{0,t}(x-vt,v)$, we have
\begin{align}\label{kk1}
	\begin{split}
		\|\mathcal{I}_h(\rho,U)(t,x)\|_{L^\infty_x}=&\left\|\int_{\mathbb{R}^d}h(X_{0,t}(x,v),\tilde{v})\det{(Id+\nabla_vW_{0,t}(x-vt,v))^{-1}}{d\tilde{v}}\right\|_{L^\infty_x}\\
		\lesssim&\left\|\int_{\mathbb{R}^d}\int_{\mathbb{R}^d}h(X_{0,t}(x,v),\tilde{v})d \tilde{v}\right\|_{L^\infty_x}
		\lesssim_a \int_{\mathbb{R}^d}\|h(\cdot,\tilde{v})\|_{L^\infty}d\tilde{v}=\|h\|_{L^1_{v}L^\infty_x}.	
	\end{split}
\end{align}
For $t\geq1$ , we change of variable $w=x-vt$ to get more decay in time:
\begin{align*}
\mathcal{I}_h(\rho,U)(t,x)
=\int_{\mathbb{R}^{d}}h\left(w+Y_{0,t}\left(w,\frac{x-w}{t}\right),\frac{x-w}{t}+W_{0,t}\left(w,\frac{x-w}{t}\right)\right)\frac{dw}{t^d}.
\end{align*} 	
Thus
\begin{equation}\label{L^p t large}
\begin{split}
\|\mathcal{I}_h(\rho,U)(t,x)\|_{L^\infty_x}=&\left\|\int_{\mathbb{R}^d}h(w,V_{0,t}(x,\Psi_{0,t}(x,\frac{x-w}{t})))\det((\nabla_v\Psi_{0,t})(x,\frac{x-w}{t}))\frac{dw}{t^d}\right\|_{L^\infty_x}\\
\lesssim&\left\|\int_{\mathbb{R}^d}h\left(w,V_{0,t}\left(x,\Psi_{0,t}(x,\frac{x-w}{t})\right)\right)\frac{dw}{t^d}\right\|_{L^\infty_x}\left\|\det\left((\nabla_v\Psi_{0,t})(x,\frac{x-w}{t})\right)\right\|_{L^\infty_{x,w}}\\
\lesssim_a& \frac{1}{t^d}\int_{\mathbb{R}^d}\|h(w,.)\|_{L^\infty}dw=\frac{1}{t^d}\|h\|_{L^1_{x}L^\infty_v}.		
\end{split}
\end{equation}
Combine \eqref{L^p t small} with \eqref{kk1} and \eqref{L^p t large} , we have
\begin{align}\label{z5}
\langle 	t\rangle^{\frac{d(p-1)}{p}}\|\mathcal{I}_h(\rho,U)(t)\|_{L^p}\lesssim_a \|h\|_{L^1_{x}L^p_v\cap L^1_vL^p_x},~~p=1,\infty.
\end{align}
$\textbf{(2)}$ Now we deal with the term $\|\mathcal{I}_h(\rho,U)(t)\|_{B^\sigma_{p,\infty}}$. From Proposition \ref{estimates about Y and W}, we have
\begin{align*}
&\|\nabla_x X(x,v)\|_{L^\infty_{x,v}}+	\|\nabla_x V(x,v)\|_{L^\infty_{x,v}}\lesssim_a1,\\
&\sup_{\alpha/t} \frac{\left|\delta_\frac{\alpha}{t} Y_{0,t}\left(w,\cdot\right)\left(\frac{x-w}{t}\right)\right|} {|\alpha/t|}+	\sup_{\alpha/t} \frac{\left|\delta_{\frac{\alpha}{t}}W_{0,t}\left(w,\cdot\right)\left(\frac{x-w}{t}\right)\right|}{|\alpha/t|}\lesssim_a1.
\end{align*}
Thus, we get
\begin{align*}
\frac{|\delta_{\alpha}\mathcal{I}_h(\rho,U)(t,x)|}{|\alpha|^\sigma}\lesssim_a \sum_{z=x,x-\alpha}	\int_{\mathbb{R}^{d}}\dot{\mathcal{D}}^\sigma (h)(X_{0,t}(z,v),V_{0,t}(z,v))dv,
\end{align*}
and	
\begin{align*}
&	\frac{|\delta_{\alpha}\mathcal{I}_h(\rho,U)(t,x)|}{|\alpha|^\sigma}\lesssim_a  \int_{\mathbb{R}^{d}}	\dot{\mathcal{D}}^\sigma_1 (h)\left(w+Y_{0,t}\left(w,\frac{x-w}{t}\right),\frac{x-w}{t}+W_{0,t}\left(w,\frac{x-w}{t}\right)\right)\frac{dw}{t^{d+\sigma}}\\&\qquad\qquad\qquad\quad+ 
\int_{\mathbb{R}^{d}}\dot{\mathcal{D}}^\sigma_2 (h)\left(w+Y_{0,t}\left(w,\frac{x-\alpha-w}{t}\right),\frac{x-\alpha-w}{t}+W_{0,t}\left(w,\frac{x-\alpha-w}{t}\right)\right)\frac{dw}{t^{d+\sigma}}\\&\qquad\qquad\qquad\lesssim_a \frac{1}{t^{\sigma}}\sum_{z=v,v-\frac{\alpha}{t}}	\int_{\mathbb{R}^{d}}\dot{\mathcal{D}}^\sigma (h)(X_{0,t}(x,z),V_{0,t}(x,z))dv	.
\end{align*}
Thanks to \eqref{z5} with $h=\dot{\mathcal{D}}^\sigma (h)$, we get 
\begin{equation}\label{z6}
\langle t\rangle^{\sigma+\frac{d(p-1)}{p}}\|\mathcal{I}_h(\rho,U)(t)\|_{\dot{B}^\sigma_{p,\infty}}\lesssim_a \|\dot{\mathcal{D}}^\sigma (h)\|_{L^1_{x}L^p_v\cap L^1_vL^p_x},~~p=1,\infty.
\end{equation}
Until now we finish the proof of \eqref{I norm}.\\
$\textbf{(3)}$Finally we consider the term $	\|\mathcal{I}_h(\rho,U)\|_{l+1,b}$. Thanks to Proposition \ref{proposition of higher derivative estimates Y W}, one has
\begin{align*}
\sum_{k=1}^{l+1}\sup_{\alpha}	\frac{	\|\delta_{\alpha}^x	\nabla_x^k(Y,W)_{0,t}\|_{L^\infty_{x,v}}}{|\alpha|^{b}}+\sum_{k=1}^{l+1}\sup_{0\leq s\leq t}\|\nabla_x^k(Y,W)_{0,t}\|_{L^\infty_{x,v}}\lesssim_{\mathbf{c},b,C_{A,l}}1.
\end{align*}
Therefore
\begin{align*}
\frac{|\delta_{\alpha}\nabla^{l}\mathcal{I}_h(\rho,U)(t,x)|}{|\alpha|^{b}}\lesssim& _{\mathbf{c},b,C_{A,l}}\sum_{j=1}^{l}\sum_{z=x,x-\alpha}\int_{\mathbb{R}^{d}}	\mathcal{D}^{b} (\nabla^j_{x,v}h)\left(X_{0,t}(z,v),V_{0,t}(z,v)\right){dw}.
\end{align*}
Thus,	\begin{equation}\|\nabla^l\mathcal{I}_h(\rho,U)(t)\|_{\dot{B}^{b}_{p,\infty}}\lesssim_{\mathbf{c},a,b,C_{A,l}}\sum_{j=1}^{l} \|\mathcal{D}^{b} (\nabla^j_{x,v}h)\|_{L^1_{x}L^p_v}.
\end{equation}
On the other hand, changing variable as $w=x-vt$ and using the estimation
\begin{align*}
\sum_{k=1}^{l+1}(1-\delta_0-b)\sup_{\alpha}	\frac{	\|\delta_{\alpha}^v	\nabla_v^k(Y,W)_{0,t}\|_{L^\infty_{x,v}}}{|\alpha|^{b}}+\sum_{k=1}^{l+1}\sup_{0\leq s\leq t}\|\nabla_v^k(Y,W)_{0,t}\|_{L^\infty_{x,v}}\lesssim_{\mathbf{c},b,C_{A,l}}1,
\end{align*}
we have
\begin{align*}
&\frac{|\delta_{\alpha}\nabla^{l}\mathcal{I}_h(\rho,U)(t,x)|}{|\alpha|^{b}}\\&\lesssim _{\mathbf{c},b}\frac{1}{1-\delta_0-b}\frac{1}{t^{d+l+b}}\Bigg[\sum_{j=1}^{l}\int_{\mathbb{R}^{d}}	\mathcal{D}^{b}_1 (\nabla^j_{x,v}h)\left(w+Y_{0,t}\left(w,\frac{x-w}{t}\right),\frac{x-w}{t}+W_{0,t}\left(w,\frac{x-w}{t}\right)\right)dw\\&\quad+ 
\int_{\mathbb{R}^{d}}\mathcal{D}^{b}_2 (\nabla^j_{x,v}h)\left(w+Y_{0,t}\left(w,\frac{x-\alpha-w}{t}\right),\frac{x-\alpha-w}{t}+W_{0,t}\left(w,\frac{x-\alpha-w}{t}\right)\right)dw\Bigg]\\&\lesssim _{\mathbf{c},b} \frac{1}{(1-\delta_0-b )t^{l+b}} \sum_{j=1}^{l}\sum_{z=v,v-\alpha/t}\int_{\mathbb{R}^{d}}	\left(\mathcal{D}^{b} (\nabla^j_{x,v}h)\right)(X_{0,t}(x,z),V_{0,t}(x,z))dv.
\end{align*}
Thanks to \eqref{z5}, we get for $p=1,\infty$
\begin{equation}\label{z8}
t^{l+b-\delta+\frac{d(p-1)}{p}}\|\nabla^l\mathcal{I}_h(\rho,U)(t)\|_{B^{b-\delta}_{p,\infty}}\lesssim_{\mathbf{c},a,b,C_{A,l}} \frac{1}{1-\delta_0-b} \sum_{j=1}^{l}\|\mathcal{D}^{b'} (\nabla^j_{x,v}h)\|_{L^1_{x}L^p_v}.
\end{equation}
Therefore, we obtain \eqref{higher I norm}.
\end{proof}
\begin{proposition}\label{I 12a estimate} Let $a\in (\frac{\sqrt{5}-1}{2},1)$, $\widetilde{\varepsilon}_0$  be as in Proposition \ref{estimates about Y and W}. Suppose that $\|(\rho_1,U_1)\|_{S^{{\varepsilon_0}}_a},\|(\rho_2,U_2)\|_{S^{{\varepsilon_0}}_a}\leq\varepsilon_0\leq\widetilde{\varepsilon}_0$. There holds, 
\begin{equation}\label{I 12norm}
\|\mathcal{I}_h(\rho_1,U_1)-\mathcal{I}_h(\rho_2,U_2)\|_{a}\lesssim_{a}   \|(\rho_1-\rho_2,U_1-U_2)\|_{S_a^{\varepsilon_0}}	\sum_{p=1,\infty}\|\mathcal{D}^a(\nabla_{x,v} h)\|_{L^1_{x}L^p_v\cap L^1_{v}L^p_x}.
\end{equation}
\end{proposition}
\begin{proof}
Let $X_{s,t}^1$, $X_{s,t}^2$, $Y_{s,t}^1$, $Y_{s,t}^2$, $V_{s,t}^1$, $V_{s,t}^2$, $\mathbf{Y}_{s,t}$, $\mathbf{W}_{s,t}$ be in Proposition \ref{Y W difference}. Then we have
\begin{align*}
&X_{s,t}^1(x,v)-X_{s,t}^2(x,v)=Y_{s,t}^1(x-vt,v)-Y_{s,t}^2(x-vt,v)=	\mathbf{Y}_{s,t}(x-vt,v),\\
& V_{s,t}^1(x,v)-V_{s,t}^2(x,v)=W_{s,t}^1(x-vt,v)-W_{s,t}^2(x-vt,v)=	\mathbf{W}_{s,t}(x-vt,v).
\end{align*}
Meanwhile, denote
\begin{align}\label{XV varpi}
\begin{split}
	&X_{s,t}(x,v,\varpi)=(1-\varpi)X_{s,t}^1(x,v)+\varpi X_{s,t}^2(x,v), ~~	V_{s,t}(x,v,\varpi)=(1-\varpi)V_{s,t}^1(x,v)+\varpi V_{s,t}^2(x,v),\\
&Y_{s,t}(x,v,\varpi)=(1-\varpi)Y_{s,t}^1(x,v)+\varpi Y_{s,t}^2(x,v), ~~	W_{s,t}(x,v,\varpi)=(1-\varpi)W_{s,t}^1(x,v)+\varpi W_{s,t}^2(x,v).
\end{split}
\end{align}
$\textbf{(1)}$ For the term $\|\mathcal{I}_h(\rho_1,U_1)(t)-\mathcal{I}_h(\rho_2,U_2)(t)\|_{L^p}$, we have 
\begin{align*}
&|\mathcal{I}_h(\rho_1,U_1)(t,x)-\mathcal{I}_h(\rho_2,U_2)(t,x)|\leq\left(\|\mathbf{Y}_{0,t}\|_{L^\infty_{x,v}}+\|\mathbf{W}_{0,t}\|_{L^\infty_{x,v}}\right) \int_{\mathbb{R}^{d}}\int_{0}^{1}\left|(\nabla_{x,v}h)(		X_{0,t}(x,v,\varpi),	V_{0,t}(x,v,\varpi))\right|d\varpi dv\\
&\overset{\eqref{results of W12 small}\eqref{results of Y12 small}}{\lesssim_a}	\|(\rho_1-\rho_2,U_1-U_2)\|_{S_a^{\varepsilon_0}}\int_{\mathbb{R}^{d}}\sup_{0\leq\varpi\leq1}\left|(\nabla_{x,v}h)(	X_{0,t}(x,v,\varpi),V_{0,t}(x,v,\varpi))\right| dv.
\end{align*}
As the proof of \eqref{z5}, one obtains
\begin{align*}
	\int_{\mathbb{R}^{d}}\sup_{0\leq\varpi\leq1}\left|(\nabla_{x,v}h)(	X_{0,t}(x,v,\varpi),V_{0,t}(x,v,\varpi))\right| dv\lesssim_{a}	\sum_{p=1,\infty}\|\nabla_{x,v}h\|_{L^1_{x}L^p_v\cap L^1_{v}L^p_x}.
\end{align*}
Thus
\begin{align}
\sum_{p=1,\infty}	\langle t\rangle^{\frac{d(p-1)}{p}}	\|\mathcal{I}_h(\rho_1)(t)-\mathcal{I}_h(\rho_2)(t)\|_{L^p}\lesssim_{a}  	\|(\rho_1-\rho_2,U_1-U_2)\|_{S_a^{\varepsilon_0}}	\sum_{p=1,\infty}\|\nabla_{x,v}h\|_{L^1_{x}L^p_v\cap L^1_{v}L^p_x}.
\end{align}
\textbf{(2)} Now we consider the term $\|\mathcal{I}_h(\rho_1,U_1)(t)-\mathcal{I}_h(\rho_2,U_2)(t)\|_{\dot{B}^a_{p,\infty}}$. We compute directly to get
\begin{equation}\label{xiao}
\begin{split}
	&	\frac{|\delta_\alpha\left(\mathcal{I}_h(\rho_1,U_1)-\mathcal{I}_h(\rho_2,U_2)\right)(t,x)|}{|\alpha|^a}\\&\quad\leq\left(\sup_z\frac{\|\delta_{z}^x\mathbf{Y}_{0,t}\|_{L^\infty_{x,v}}}{|z|^a}+\sup_z\frac{\|\delta_{z}^x\mathbf{W}_{0,t}\|_{L^\infty_{x,v}}}{|z|^a}\right) \int_{\mathbb{R}^{d}}\int_{0}^{1}\left|(\nabla_{x,v}h)(X_{0,t}(x,v,\varpi),	V_{0,t}(x,v,\varpi))\right|d\varpi dv\\&\quad\quad+\left(\|\mathbf{Y}_{0,t}\|_{L^\infty_{x,v}}+\|\mathbf{W}_{0,t}\|_{L^\infty_{x,v}}\right) \int_{\mathbb{R}^{d}}\int_{0}^{1}\sum_{z=x,x-\alpha}\dot{\mathcal{D}}^a(\nabla_{x,v}h)(X_{0,t}(z,v,\varpi),	V_{0,t}(z,v,\varpi))d\varpi dv\\
&\overset{\eqref{results of W12 small}\eqref{results of Y12 small}}{\lesssim_a}\|(\rho_1-\rho_2,U_1-U_2)\|_{S_a^{\varepsilon_0}}\sum_{z=x,x-\alpha}\int_{\mathbb{R}^{d}}\sup_{0\leq\varpi\leq1}\mathcal{D}^a(\nabla_{x,v}h)(X_{0,t}(z,v,\varpi),	V_{0,t}(z,v,\varpi))d\varpi dv.
\end{split}
\end{equation}
On the other hand, we  first make the change of variable as $w=x-vt$, then taking $\delta_\alpha$ to get the decay of $t^{-a}$, after this we go back to the original variable to obtain
\begin{align}\label{da}
\begin{split}
	&	\frac{|\delta_\alpha\left(\mathcal{I}_h(\rho_1,U_1)-\mathcal{I}_h(\rho_2,U_2)\right)(t,x)|}{|\alpha|^a}\\&\quad\leq\frac{1}{t^a}\left(\sup_z\frac{\|\delta_{z}^v\mathbf{Y}_{0,t}\|_{L^\infty_{x,v}}}{|z|^a}+\sup_z\frac{\|\delta_{z}^v\mathbf{W}_{0,t}\|_{L^\infty_{x,v}}}{|z|^a}\right) \int_{\mathbb{R}^{d}}\int_{0}^{1}\left|(\nabla_{x,v}h)(X_{0,t}(x,v,\varpi),	V_{0,t}(x,v,\varpi))\right|d\varpi dv\\&\qquad+\frac{1}{t^a}\left(\|\mathbf{Y}_{0,t}\|_{L^\infty_{x,v}}+\|\mathbf{W}_{0,t}\|_{L^\infty_{x,v}}\right) \int_{\mathbb{R}^{d}}\int_{0}^{1}\sum_{z=x,x-\alpha}\dot{\mathcal{D}}^a(\nabla_{x,v}h)(X_{0,t}(z,v,\varpi),	V_{0,t}(z,v,\varpi))d\varpi dv\\
&\overset{\eqref{results of W12 small}\eqref{results of Y12 small}}{\lesssim_a}\frac{\|(\rho_1-\rho_2,U_1-U_2)\|_{S_a^{\varepsilon_0}}}{t^a}\sum_{z=x,x-\alpha}\int_{\mathbb{R}^{d}}\sup_{0\leq\varpi\leq1}\mathcal{D}^a(\nabla_{x,v}h)(X_{0,t}(z,v,\varpi),	V_{0,t}(z,v,\varpi))d\varpi dv.
\end{split}
\end{align}
Combining \eqref{xiao} with \eqref{da}, we have
\begin{align*}
&	\frac{|\delta_\alpha\left(\mathcal{I}_h(\rho_1,U_1)-\mathcal{I}_h(\rho_2,U_2)\right)(t,x)|}{|\alpha|^a}\\&\quad\lesssim_a\frac{\|(\rho_1-\rho_2,U_1-U_2)\|_{S_a^{\varepsilon_0}}}{\langle t\rangle^a}\sum_{z=x,x-\alpha}\int_{\mathbb{R}^{d}}\sup_{0\leq\varpi\leq1}\mathcal{D}^a(\nabla_{x,v}h)(X_{0,t}(z,v,\varpi),	V_{0,t}(z,v,\varpi))d\varpi dv.
\end{align*}
Therefore, applying the the same method in \eqref{z5}, we obtain
\begin{align}\label{mathcalI difference}
\sup_{t>0}\langle t\rangle^{a+\frac{d(p-1)}{p}}\|(\mathcal{I}_h(\rho_1,U_1)-\mathcal{I}_h(\rho_2,U_2))(t)\|_{\dot{B}^a_{p,\infty}}\lesssim_a\|(\rho_1-\rho_2,U_1-U_2)\|_{S_a^{\varepsilon_0}} \sum_{p=1,\infty}\|\mathcal{D}^a(\nabla_{x,v} h)\|_{ L^1_{x}L^p_v}.
\end{align}  
Now the proof is complete.
\end{proof}	
\section{Contribution Of The Reaction Term}\label{Contribution of the reaction term}
In this section, we focus on estimate of the reaction term:
\begin{equation}\label{definition of R}
\mathcal{R}(\rho,U)(t,x)=\int_{0}^{t}\int_{\mathbb{R}^{d}}E(s,x-(t-s)v)\cdot\nabla_v\mu(v)dvds-\int_{0}^{t}\int_{\mathbb{R}^{d}}E(s,X_{s,t}(x,v))\cdot\nabla_v\mu(V_{s,t}(x,v))dvds.
\end{equation}
Denote that for any given $F:\mathbb{R}^+\times \mathbb{R}^d\to \mathbb{R}$ and $\eta:  \mathbb{R}^d\to \mathbb{R}$, 
\begin{align*}
\mathcal{T}[F,\eta](t,x)=-	\mathcal{T}_{NL}[F,\eta](t,x)+	\mathcal{T}_{L}[F,\eta](t,x),
\end{align*}
with 
\begin{align*}
&	\mathcal{T}_{L}[F,\eta](t,x)=\int_{0}^{t}\int_{\mathbb{R}^{d}}F(s,x-(t-s)v)\eta(v)dvds,\\&
\mathcal{T}_{NL}[F,\eta](t,x)=\int_{0}^{t}\int_{\mathbb{R}^{d}}F(s,X_{s,t}(x,v))\eta(V_{s,t}(x,v))dvds.
\end{align*}
Hence, it is clear that
\begin{align*}
\mathcal{R}(\rho)=\sum_{i=1}^{d}\mathcal{T}[E_i,\partial_{x_i}\mu].
\end{align*}
By changing variable, we reformuate as follows: 
\begin{align*}
&	\mathcal{T}_{L}[F,\eta](t,x)=	\int_{0}^{t}\int_{\mathbb{R}^{d}}F(s,w+\frac{s}{t}(x-w))\eta(\frac{x-w}{t})\frac{dwds}{t^d},\\&	\mathcal{T}_{NL}[F,\eta](t,x)=\int_{0}^{t}\int_{\mathbb{R}^d}
F\left(s,Y_{s,t}(w,\frac{x-w}{t})+w+\frac{s(x-w)}{t}\right)\eta\left(W_{s,t}(w,\frac{x-w}{t})+\frac{x-w}{t}\right)
\frac{dwds}{t^{d}}.
\end{align*}
Then, we have following Lemmas.
\begin{lemma} \label{esT}
Let $a\in(\frac{\sqrt{5}-1}{2},1)$ and  $\widetilde{\varepsilon}_0>0$ be as in Proposition \ref{estimates about Y and W}. Assume that $\|(\rho,U)\|_{S^{{\varepsilon_0}}_a}\leq\widetilde{\varepsilon}_0$. Let $(X_{s,t},V_{s,t})$ be  solution to \eqref{X and V} associated to $E(s,x)=\nabla_x(1-\Delta)^{-1}(\rho(s,x))$. Let  $Y_{s,t},W_{s,t}$ be such that 
\begin{equation}\label{XYVW def}
X_{s,t}(x,v)=x-(t-s)v+Y_{s,t}(x-vt,v),V_{s,t}(x,v)=v+W_{s,t}(x-vt,v).
\end{equation} Suppose $\eta$ be such that \begin{equation}\label{z9}
\sum_{j=0}^2\langle v\rangle^N|\nabla^j\eta(v)|\leq \mathbf{c}_1,
\end{equation}
for $N>d$. 
Then, 
\\ \textbf{(1)}
\begin{align}\label{z10'}
\left|	\mathcal{T}[F,\eta](t,x)\right|\lesssim_{a,\mathbf{c}_1} \|(\rho,U)\|_{S^{{\varepsilon_0}}_a} \int_{0}^{t}\int_{\mathbb{R}^{d}}|F(s,x-(t-s)v)|\frac{1}{\langle s\rangle^{d-2+a}}\frac{dvds}{\langle v\rangle^N}.
\end{align}
In particular, for $p=1,\infty$
\begin{align}\label{z21}
\langle t\rangle^\frac{d(p-1)}{p}\|	\mathcal{T}[F,\eta](t)\|_{L^p}\lesssim_{a,\mathbf{c}_1} \|(\rho,U)\|_{S^{{\varepsilon_0}}_a}\left(\sup_{s\in [0,t]}\langle s\rangle^\frac{d(p-1)}{p}\|F(s)\|_{L^p}+ \sup_{s\in [0,t]}\|F(s)\|_{L^1}\right).
\end{align}
\\ \textbf{(2)}
For any	$\sigma\in[a,1)$
\begin{align}
&\sum_{p=1,\infty}\langle t\rangle^{\frac{d(p-1)}{p}+\sigma}\|\mathcal{T}[F,\eta](t)\|_{\dot{B}^\sigma_{p,\infty}}\lesssim_{a,\mathbf{c}_1} \|(\rho,U)\|_{S^{{\varepsilon_0}}_a}^\sigma\sum_{p=1,\infty} \left(\sup_{s\in [0,t]} \langle s\rangle^{\frac{d(p-1)}{p}} \|F(s)\|_{L^p}+\sup_{s\in [0,t]}\langle s\rangle^{\frac{d(p-1)}{p}+\sigma}	\|F(s)\|_{\dot{F}^\sigma_{p,\infty}}\right). \label{z12}
\end{align}
\end{lemma}
\begin{proof} \textbf{(1)}
First of all, we change of variable and  obtain
\begin{align*}
|\mathcal{T}[F,\eta](t,x)|=&\left|-\int_{0}^{t}\int_{\mathbb{R}^{d}}F(s,x-(t-s)v)\eta(V_{s,t}(x,\Psi_{s,t}(x,v)))|det(\nabla_v\Psi_{s,t}(x,v))|dvds\right.\\
&\left.+\int_{0}^{t}\int_{\mathbb{R}^{d}}F(s,x-(t-s)v)\eta(v)dvds\right|\\
\leq& \int_{0}^{t}\int_{\mathbb{R}^{d}}|F(s,x-(t-s)v)|\left|\eta(V_{s,t}(x,\Psi_{s,t}(x,v)))-\eta(v)\right|dvds\\
&+\int_{0}^{t}\int_{\mathbb{R}^{d}}|F(s,x-(t-s)v)|\Big||\det(\nabla_v\Psi_{s,t}(x,v))|-1\Big|\left|\eta(V_{s,t}(x,\Psi_{s,t}(x,v)))\right|dvds.
\end{align*}
Hence, since we have 
\begin{align*}
\left|\eta(V_{s,t}(x,\Psi_{s,t}(x,v)))-\eta(v)\right|&\leq\int_{0}^{1}\left|\nabla\eta(\varpi V_{s,t}(x,v)+(1-\varpi)v)\right|V_{s,t}(x,v)-v|d\varpi\\
&{\lesssim_{\mathbf{c}_1}}\sup_{\varpi\in[0,1]}\frac{|V_{s,t}(x,v)-v|}{\langle \varpi W_{s,t}(x-vt,v)+v\rangle^N}\\
&\sim_{\mathbf{c}_1}\frac{|W_{s,t}(x-vt,v)|}{\langle v\rangle^N}\overset{\eqref{results of W small}, \eqref{z3}}{\lesssim_{a,\mathbf{c}_1}} \frac{\|(\rho,U)\|_{S^{{\varepsilon_0}}_a}}{\langle s\rangle^{{d-1+a}}{\langle v\rangle^N}},
\end{align*}
and
$\left|\eta(V_{s,t}(x,\Psi_{s,t}(t,x)))\right|{\lesssim_{\mathbf{c}_1}}{\langle W_{s,t}(x-t\Psi_{s,t}(t,x),\Psi_{s,t}(t,x))+v\rangle^{-N}}{\lesssim_{\mathbf{c}_1}}{\langle v\rangle^{-N}}$, 
where we use the fact that \\
$ 	\sup_{0\leq s\leq t}\|W_{s,t}\|_{L^\infty_{x,v}}\leq 1 $ in \eqref{results of W small}. Hence
\begin{align*}
	\left|	\mathcal{T}[F,\eta](t,x)\right|
	&\overset{\eqref{z3}}{\lesssim_{a,\mathbf{c}_1}} \int_{0}^{t}\int_{\mathbb{R}^{d}}|F(s,x-(t-s)v)|\left(\frac{\|(\rho,U)\|_{S^{{\varepsilon_0}}_a}}{\langle s\rangle^{d-1+a}}+\frac{\|(\rho,U)\|_{S^{{\varepsilon_0}}_a}}{\langle s\rangle^{d-2+a}}\right)\frac{dvds}{\langle v\rangle^N}\\
&	\lesssim_{a,\mathbf{c}_1} \|(\rho,U)\|_{S^{{\varepsilon_0}}_a} \int_{0}^{t}\int_{\mathbb{R}^{d}}|F(s,x-(t-s)v)|\frac{1}{\langle s\rangle^{d-2+a}}\frac{dvds}{\langle v\rangle^N}.
\end{align*}
This gives \eqref{z10'}. Moreover,
\begin{align*}
\|	\mathcal{T}[F,\eta](t)\|_{L^p}\lesssim_{a,\mathbf{c}_1} \|(\rho,U)\|_{S^{{\varepsilon_0}}_a} \int_{0}^{t}\left\|\int_{\mathbb{R}^{d}}|F(s,x-(t-s)v)|\frac{dv}{\langle v\rangle^N}\right\|_{L^p_x}\frac{ds}{\langle s\rangle^{d-2+a}}.
\end{align*} 
Then, we apply Lemma \ref{mathcal H1H2} with $\mathcal{H}_1=F$ and $\mathcal{H}_2=0$ to obtain
\begin{align*}
\|	\mathcal{T}[F,\eta](t)\|_{L^p}\lesssim_{a,\mathbf{c}_1}& t^{-\frac{d(p-1)}{p}} \|(\rho,U)\|_{S^{{\varepsilon_0}}_a} \int_{0}^{t/2}\|F(s)\|_{L^1}\frac{ds}{\langle s\rangle^{d-2+a}}+\|(\rho,U)\|_{S^{{\varepsilon_0}}_a}\int_{t/2}^{t}\|F(s)\|_{L^p}\frac{ds}{\langle s\rangle^{d-2+a}}\\
\lesssim_{a,\mathbf{c}_1}&\underbrace{ t^{-\frac{d(p-1)}{p}} \|(\rho,U)\|_{S^{{\varepsilon_0}}_a}\sup_{s\in [0,t]}\|F(s)\|_{L^1} \int_{0}^{t/2}\frac{ds}{\langle s\rangle^{d-2+a}}}_{\text{Thanks to }\eqref{mathcal H 1}}\\
&+\underbrace{\|(\rho,U)\|_{S^{{\varepsilon_0}}_a}\left(\sup_{s\in [0,t]}\langle s\rangle^\frac{d(p-1)}{p}\|F(s)\|_{L^p}\right)\int_{t/2}^{t}\frac{ds}{\langle s\rangle^{d-2+a}\langle s\rangle^\frac{d(p-1)}{p}}}_{\text{Thanks to }\eqref{mathcal H p}}\\
\lesssim_{a,\mathbf{c}_1}& \frac{1}{t^\frac{d(p-1)}{p}}\|(\rho,U)\|_{S^{{\varepsilon_0}}_a} \left(\sup_{s\in [0,t]}\|F(s)\|_{L^1}+ \sup_{s\in [0,t]}\langle s\rangle^\frac{d(p-1)}{p}\|F(s)\|_{L^p}\right).
\end{align*}
Combining this with 
\begin{align*}
\|	\mathcal{T}[F,\eta](t)\|_{L^p}\lesssim_{a,\mathbf{c}_1} \|(\rho,U)\|_{S^{{\varepsilon_0}}_a} \int_{0}^{t}\|F(s)\|_{L^p}\frac{1}{\langle s\rangle^{d-2+a}}ds\lesssim_{a,\mathbf{c}_1} \|(\rho,U)\|_{S^{{\varepsilon_0}}_a}\sup_{s\in [0,t]}\|F(s)\|_{L^p},
\end{align*} to obtain \eqref{z21}. \vspace{0.25cm}\\
\textbf{(2)} Now we estimate $\|\mathcal{T}[F,\eta](t)\|_{\dot{B}^\sigma_{p,\infty}}.$ \vspace{0.15cm}\\
\textbf{2.1)} Estimate $\|\mathcal{T}[F,\eta](t)\|_{\dot{B}^\sigma_{p,\infty}}$ for $0\leq t<1$. \\We divide $\delta_{\alpha}\mathcal{T}[F,\eta](t,x)$ into three terms:
\begin{align*}
\delta_{\alpha}\mathcal{T}[F,\eta](t,x)=\overline{\mathcal{T}}_{\alpha}^1[F,\eta](t,x)|+\overline{\mathcal{T}}_{\alpha}^2[F,\eta](t,x)+	\overline{\mathcal{T}}_{\alpha}^3[F,\eta](t,x),
\end{align*}
where
\begin{align*}
&	\overline{\mathcal{T}}_{\alpha}^1[F,\eta](t,x)=-\int_{0}^{t}\int_{\mathbb{R}^d}
\delta_{\alpha}F(s,.)\left(X_{s,t}(x,v)\right)\eta\left(V_{s,t}(x,v)\right)
dvds
+	\int_{0}^{t}\int_{\mathbb{R}^{d}}	\delta_{\alpha}F(s,.)(x-(t-s)v)\eta(v)dvds,
\\	&\overline{\mathcal{T}_{\alpha}}^2[F,\eta](t,x)=\int_{0}^{t}\int_{\mathbb{R}^d}
\left[F\left(s,X_{s,t}(x-\alpha,v)\right)-F\left(s,,X_{s,t}(x,v)-\alpha\right) \right]\eta\left(V_{s,t}(x,v)\right)dvds,
\\
&	\overline{\mathcal{T}}_{\alpha}^3[F,\eta](t,x)=\int_{0}^{t}\int_{\mathbb{R}^d}
F\left(s,X_{s,t}(x-\alpha,v)\right) \left[\eta\left(V_{s,t}(x-\alpha,v)\right)-\eta\left(V_{s,t}(x,v)\right) \right]
dvds.
\end{align*}
For the first term, using \eqref{z9} and \eqref{z10'} we have
\begin{align*}
\|\overline{\mathcal{T}}_{\alpha}^1[F,\eta](t)\|_{L^p}&=\|\mathcal{T}[\delta_{\alpha}F,\eta](t)\|_{L^p}\lesssim_{a,\mathbf{c}_1} \|(\rho,U)\|_{S^{{\varepsilon_0}}_a} \left(\sup_{s\in [0,t]}\|\delta_{\alpha}F(s)\|_{L^1}+\sup_{s\in [0,t]}\|\delta_{\alpha}F(s)\|_{L^p} \right)\\
&	\lesssim_{a,\mathbf{c}_1}|\alpha|^\sigma \|(\rho,U)\|_{S^{{\varepsilon_0}}_a} \left(\sup_{s\in [0,t]}\|F(s)\|_{\dot{B}^\sigma_{1,\infty}}+
\sup_{s\in [0,t]}\|F(s)\|_{\dot{B}^\sigma_{p,\infty}}\right).
\end{align*}
Since 
\begin{align*}
\left|X_{s,t}(x-\alpha,v)-\left(X_{s,t}(x-\alpha,v)-\alpha\right)\right|=\left|Y_{s,t}(x-tv,v)-Y_{s,t}(x-tv,v)\right|\leq|\alpha|\|\nabla_x Y_{s,t}\|_{L^\infty},
\end{align*} 
we can estimate  $	\overline{\mathcal{T}}_{\alpha}^2[F,\eta](t,x)$ as follows
\begin{align*}
\left\|\overline{\mathcal{T}}_{\alpha}^2[F,\eta](t,x)\right\|_{L^p_x}
\lesssim_{\mathbf{c}_1}&|\alpha|^\sigma\int_{0}^{t}\int_{\mathbb{R}^d}\|\sup_z	|z|^{-\sigma}|\delta_{z}F(s,.)\|_{L^p}\|\nabla_x Y_{{s,t}}\|_{L^\infty_{x,v}}^\sigma\frac{dv}{\langle v \rangle^N}ds
\\
\lesssim_{a,\mathbf{c}_1}&|\alpha|^\sigma\|(\rho,U)\|_{S^{{\varepsilon_0}}_a}^\sigma\left(\sup_{s\in [0,t]}\|F(s)\|_{\dot{F}^\sigma_{p,\infty}}\right)
\int_{0}^{t}\langle s\rangle^{-(d+a-2)\sigma}ds
\lesssim_{a,\mathbf{c}_1}|\alpha|^\sigma\|(\rho,U)\|_{S^{{\varepsilon_0}}_a}^\sigma\sup_{s\in [0,t]}\|F(s)\|_{\dot{F}^\sigma_{p,\infty}},
\end{align*}
where in the last inequality, we use $(d+a-3)\sigma+a\geq(a+1)a>1$.\\
Note that by \eqref{results of W small}, one has
$
\sup_{s,t}\|W_{s,t}\|_{L^\infty_{x,v}}\leq 1.
$
Then,
\begin{align}\label{V sim v}
\langle \varpi V_{s,t}(x,v)+(1-\varpi)V_{s,t}(x-\alpha,v)\rangle= \langle v+\varpi W_{s,t}(x-tv,v)+(1-\varpi)W_{s,t}(x-\alpha-tv,v)\rangle
\sim\langle v\rangle.
\end{align}	
We can estimate $\overline{\mathcal{T}}_{\alpha}^3[F,\eta](t)$.
\begin{align*}				
\|\overline{\mathcal{T}}_{\alpha}^3[F,\eta](t,x)\|_{L^p}&\leq \int_{0}^{t}\int_{\mathbb{R}^d}\int_0^1\|F(s)\|_{L^p}\left|\nabla\eta(\cdot)\right|(\varpi V_{s,t}(x,v)+(1-\varpi)V_{s,t}(x-\alpha,v)) \left|\delta_\alpha^xV_{s,t}(x,v)\right| dvdsd\varpi\\
&\lesssim_{\mathbf{c}_1} |\alpha|^\sigma\int_{0}^{t}\int_{\mathbb{R}^d}\|F(s)\|_{L^p}\frac{1}{\langle v\rangle^N}\sup_\alpha\frac{\left|\delta_\alpha^xV_{s,t}(x,v)\right|}{|\alpha|^\sigma} dvds\\
&\lesssim_{a,\mathbf{c}_1}|\alpha|^\sigma\|(\rho,U)\|_{S^{{\varepsilon_0}}_a}\left(\sup_{s\in [0,t]}\|F(s)\|_{L^p}\right)\int_0^t \frac{ds}{\langle s\rangle^{d-1+a}}
\lesssim_{a,\mathbf{c}_1}|\alpha|^\sigma\|(\rho,U)\|_{S^{{\varepsilon_0}}_a}^\sigma \sup_{s\in [0,t]} \|F(s)\|_{L^p} ,
\end{align*}
where we apply \eqref{results of W small}.\\
In conclusion,
for $0<t<1$, we get
\begin{align}\label{d3}
\|\mathcal{T}[F,\eta](t)\|_{\dot{B}^\sigma_{p,\infty}} 
\lesssim_{a,\mathbf{c}_1}& \|(\rho,U)\|_{S^{{\varepsilon_0}}_a}  ^\sigma \left(\sup_{s\in [0,t]} \|F(s)\|_{\dot{F}^\sigma_{p,\infty}}
+\sup_{s\in [0,t]}\|F(s)\|_{L^p}\right).
\end{align}
\textbf{2.2)}Estimate $\|\mathcal{T}[F,\eta](t)\|_{\dot{B}^\sigma_{p,\infty}}$ for $t\geq1$. \\
Set 
\begin{align}\label{Z1234}
\begin{split}
	&Z_1(x)=Y_{s,t}(w,\frac{x-w}{t})+w+\frac{s(x-w)}{t},~	Z_2(x)=W_{s,t}(w,\frac{x-w}{t})+\frac{x-w}{t},\\
	&Z_3(x)=W_{s,t}(w,\frac{x-\alpha-w}{t})+\frac{x-w}{t},~~
	Z_4(x)=Y_{s,t}(w,\frac{x-w}{t})+w+\frac{s(x-\alpha-w)}{t}.
\end{split}
\end{align}
We have 
\begin{align*}
\delta_{\alpha}\mathcal{T}[F,\eta](t,x)=\mathcal{T}_{\alpha}^1[F,\eta](t,x)|+\mathcal{T}_{\alpha}^2[F,\eta](t,x)+	\mathcal{T}_{\alpha}^3[F,\eta](t,x)+	\mathcal{T}_{\alpha}^4[F,\eta](t,x),
\end{align*}
where 
\begin{align*}
&	\mathcal{T}_{\alpha}^1[F,\eta](t,x)=-\int_{0}^{t}\int_{\mathbb{R}^d}
\delta_{\frac{s\alpha}{t}}F(s,.)\left(Z_1(x)\right)\eta\left(Z_2(x)\right)
\frac{dwds}{t^{d}} +	\int_{0}^{t}\int_{\mathbb{R}^{d}}	\delta_{\frac{s\alpha}{t}}F(s,.)(w+\frac{s}{t}(x-w))\eta(\frac{x-w}{t})\frac{dwds}{t^d},\\
&	\mathcal{T}_{\alpha}^2[F,\eta](t,x)=\int_{0}^{t}\int_{\mathbb{R}^d}
F\left(s,Z_1(x-\alpha)\right)(\delta_{\frac{\alpha}{t}}\eta)\left(Z_2(x)\right)
\frac{dwds}{t^{d}}
-	\int_{0}^{t}\int_{\mathbb{R}^{d}}F(s,w+\frac{s}{t}(x-w-\alpha))(\delta_{\frac{\alpha}{t}}\eta)(\frac{x-w}{t})\frac{dwds}{t^d},\\
&	\mathcal{T}_{\alpha}^3[F,\eta](t,x)=\int_{0}^{t}\int_{\mathbb{R}^d}
\left(F\left(s,Z_1(x-\alpha)\right)-F\left(s,Z_4(x)\right)\right) \eta\left(Z_2(x)\right)
\frac{dwds}{t^{d}},\\
&	\mathcal{T}_{\alpha}^4[F,\eta](t,x)=\int_{0}^{t}\int_{\mathbb{R}^d}
F\left(s,Z_1(x-\alpha)\right) \left(\eta\left(Z_3(x)\right)-\eta\left(Z_2(x)\right)\right)
\frac{dwds}{t^{d}}.
\end{align*}
Going back to the variable $v=\frac{x-w}{t}$,
and denote
\begin{align*}
Z_5(x)=Y_{s,t}(x-tv,v-\frac{\alpha}{t})+x-(t-s)v-\frac{s\alpha}{t},~~
Z_6(x)=W_{s,t}(x-tv,v-\frac{\alpha}{t})+v,
\end{align*}
we reformulate as follows:
\begin{align*}
&\mathcal{T}_{\alpha}^1[F,\eta](t,x)=\int_{0}^{t}\int_{\mathbb{R}^d}
\left(-	\delta_{\frac{s\alpha}{t}}F(s,.)\left(X_{s,t}(x,v)\right)\eta\left(V_{s,t}(x,v)\right)+	\delta_{\frac{s\alpha}{t}}F(s,.)(x-(t-s)v)\eta(v)\right)dvds,\\&
\mathcal{T}_{\alpha}^2[F,\eta](t,x)=\int_{0}^{t}\int_{\mathbb{R}^d}
\left(	F\left(s,X_{s,t}(x,v)\right)\left(\delta_{\frac{\alpha}{t}}\eta\right)\left(V_{s,t}(x,v)\right)-F(s,x-(t-s)v)\left(\delta_{\frac{\alpha}{t}}\eta\right)(v)\right)dvds,
\\
&	\mathcal{T}_{\alpha}^3[F,\eta](t,x)= \int_{0}^{t}\int_{\mathbb{R}^d}
\left(F\left(s,X_{s,t}(x,v)-\frac{s\alpha}{t}\right)-F\left(s,Z_5(x)\right)\right)
\eta\left(V_{s,t}(x,v)\right)dvds,
\\&
\mathcal{T}_{\alpha}^4[F,\eta](t,x)= \int_{0}^{t}\int_{\mathbb{R}^d}
F\left(s,Z_5(x)\right)
\left(\eta\left(V_{s,t}(x,v)\right)-\eta\left(Z_6(x)\right)\right)
dvds.
\end{align*}
	Note that for $p=1,\infty$ and $|\alpha|\geq t$
\begin{align*}
	\langle t\rangle^{\frac{d(p-1)}{p}+a}\frac{\|\delta_{\alpha}\mathcal{T}[F,\eta](t)\|_{L^p}}{|\alpha|^a}&\leq 2	\langle t\rangle^{\frac{d(p-1)}{p}}\|\delta_{\alpha}\mathcal{T}[F,\eta](t)\|_{L^p}\leq 4	\langle t\rangle^{\frac{d(p-1)}{p}}\|\mathcal{T}[F,\eta](t)\|_{L^p}\\& \overset{\eqref{z21}}{\lesssim_{a,\mathbf{c}_1}} \|(\rho,U)\|_{S^{{\varepsilon_0}}_a}\left(\sup_{s\in [0,t]}\langle s\rangle^\frac{d(p-1)}{p}\|F(s)\|_{L^p}+ \sup_{s\in [0,t]}\|F(s)\|_{L^1}\right).
\end{align*}
So, it is enough to consider $|\alpha|\leq t$.\\
\textbf{2.2.1)} Estimate $	\mathcal{T}_{\alpha}^1[F,\eta](t,x)$.\\Applying  \eqref{z21} with $\delta_{\frac{s\alpha}{t}}F$, we have
\begin{align*}
\|\mathcal{T}_{\alpha}^1[F,\eta](t)\|_{L^p}&=\|\mathcal{T}[\delta_{\frac{s\alpha}{t}}F,\eta](t)\|_{L^p}\lesssim_{a,\mathbf{c}_1} t^{-\frac{d(p-1)}{p}}\|(\rho,U)\|_{S^{{\varepsilon_0}}_a} \left(\sup_{s\in [0,t]}\|\delta_{\frac{s\alpha}{t}}F(s)\|_{L^1}+\sup_{s\in [0,t]}\langle s\rangle^\frac{d(p-1)}{p}\|\delta_{\frac{s\alpha}{t}}F(s)\|_{L^p} \right)\\
&	\lesssim_{a,\mathbf{c}_1}|\alpha|^\sigma t^{-\frac{d(p-1)}{p}-\sigma}\|(\rho,U)\|_{S^{{\varepsilon_0}}_a} \left(\sup_{s\in [0,t]}\langle s\rangle^a\|F(s)\|_{\dot{B}^\sigma_{1,\infty}}+
\sup_{s\in [0,t]}\langle s\rangle^{\frac{d(p-1)}{p}+a}\|F(s)\|_{\dot{B}^\sigma_{p,\infty}}\right),
\end{align*}
where we apply $\sigma<2a$. 
Thus, we can yield
\begin{align}\label{T alpha lower 1}
t^{\frac{d(p-1)}{p}+\sigma}\sup_\alpha	\frac{	\|\mathcal{T}_{\alpha}^1[F,\eta](t)\|_{L^p}}{|\alpha|^\sigma}\lesssim_{a,\mathbf{c}_1} \|(\rho,U)\|_{S^{{\varepsilon_0}}_a}\left(\sup_{s\in [0,t]}\langle s\rangle^a\|F(s)\|_{\dot{B}^\sigma_{1,\infty}}+
\sup_{s\in [0,t]}\langle s\rangle^{\frac{d(p-1)}{p}+a}\|F(s)\|_{\dot{B}^\sigma_{p,\infty}}\right).
\end{align}
\textbf{2.2.2)} Estimate $	\mathcal{T}_{\alpha}^2[F,\eta](t,x)$.  \\Note that 
\begin{align*}
|\delta_{\frac{\alpha}{t}}\eta(v)|+|\nabla_v(\delta_{\frac{\alpha}{t}}\eta)(v)|\lesssim_{\mathbf{c}_1} \min\left\{\frac{|\alpha|}{t},1\right\}\left(\frac{1}{\langle v\rangle^N}+\frac{1}{\langle v-\frac{\alpha}{t}\rangle^N}\right)\lesssim_{\mathbf{c}_1}\frac{|\alpha|^\sigma}{t^\sigma}\frac{1}{\langle v\rangle^N},
\end{align*}
since $|\alpha|\leq t$.
Thus, we have
\begin{align*}
	\|\mathcal{T}_{\alpha}^2[F,\eta](t)\|_{L^p}&=\frac{|\alpha|^\sigma}{t^\sigma}\left\|\mathcal{T}\left[F,\frac{t^\sigma}{|\alpha|^\sigma}\delta_{\frac{\alpha}{t}}\eta(v)\right]\right\|_{L^p}\\&\lesssim_{a,\mathbf{c}_1}\frac{|\alpha|^\sigma}{t^{\sigma+\frac{d(p-1)}{p}}}\|(\rho,U)\|_{S^{{\varepsilon_0}}_a} \left(\sup_{s\in [0,t]}\|F(s)\|_{L^1}+\sup_{s\in [0,t]}\langle s\rangle^\frac{d(p-1)}{p}\|F(s)\|_{L^p} \right),
\end{align*}
which implies 
\begin{align*}
t^{\frac{d(p-1)}{p}+\sigma}\sup_\alpha	\frac{	\|\mathcal{T}_{\alpha}^2[F,\eta](t)\|_{L^p}}{|\alpha|^\sigma}\lesssim_{a,\mathbf{c}_1} \|(\rho,U)\|_{S^{{\varepsilon_0}}_a}\left(\sup_{s\in [0,t]}\langle s\rangle^\frac{d(p-1)}{p}\|F(s)\|_{L^p}+\sup_{s\in [0,t]}\|F(s)\|_{L^1}\right).
\end{align*}
\textbf{2.2.3)} Estimate $	\mathcal{T}_{\alpha}^3[F,\eta](t,x)$.
\begin{align}\nonumber
&	|\mathcal{T}_{\alpha}^3[F,\eta](t,x)|\\\nonumber
&\lesssim_{\mathbf{c}_1}\int_{0}^{t}\int_{\mathbb{R}^d}
\left|F\left(s,X_{s,t}(x,v)-\frac{s\alpha}{t}\right)-F\left(s,Z_5(x)\right)\right|\frac{dvds}{\langle v\rangle^N}\\\nonumber
&\lesssim_{\mathbf{c}_1}\int_{0}^{t}\int_{\mathbb{R}^d}
\left(	\sup_z	|z|^{-\sigma}\left|\delta_{z}F(s,.)\left(Y_{s,t}(x-tv,v)+x-(t-s)v-\frac{s\alpha}{t}\right)\right|\right)\left|Y_{s,t}(x-tv,v)-Y_{s,t}(x-tv,v-\frac{\alpha}{t})\right|^\sigma\frac{dvds}{\langle v\rangle^N}\\\label{T3fml 1}
&\lesssim_{a,\mathbf{c}_1}\frac{|\alpha|^\sigma}{t^\sigma} \|(\rho,U)\|_{S^{{\varepsilon_0}}_a}^\sigma\int_{0}^{t}\int_{\mathbb{R}^d}
\sup_z	|z|^{-\sigma}\left|\delta_{z}F(s,.)\left(Y_{s,t}(x-tv,v)+x-(t-s)v-\frac{s\alpha}{t}\right)\right|\frac{ds}{\langle s\rangle^{\sigma(d-3+a)}} \frac{dv}{\langle v\rangle^N}.
\end{align}
Thus applying Lemma \ref{mathcal H1H2} with $\mathcal{H}=\sup_z	|z|^{-\sigma}|\delta_{z}F(s,.)|$ and $\varphi(x,v)=Y_{s,t}(x-tv,v-\frac{\alpha}{t})-\frac{s\alpha}{t}$ to obtain 
\begin{align*}
&\|\mathcal{T}_{\alpha}^3[F,\eta](t)\|_{L^p}\\&\lesssim _{a,\mathbf{c}_1}  \underbrace{\frac{|\alpha|^\sigma}{t^{\frac{d(p-1)}{p}+\sigma}}\|(\rho,U)\|_{S^{{\varepsilon_0}}_a}^\sigma\int_{0}^{t/2}
	\|F(s)\|_{\dot{F}^\sigma_{1,\infty}}\frac{1}{\langle s\rangle^{\sigma(d-3+a)}} ds}_{\text{Thanks to }\eqref{mathcal H 1}}+\underbrace{\frac{|\alpha|^\sigma}{t^\sigma}\|(\rho,U)\|_{S^{{\varepsilon_0}}_a}^\sigma\int_{t/2}^{t}
	\|F(s)\|_{\dot{F}^\sigma_{p,\infty}}\frac{1}{\langle s\rangle^{\sigma(d-3+a)}} ds}_{\text{Thanks to }\eqref{mathcal H p}}\\&
\lesssim  _{a,\mathbf{c}_1} \frac{|\alpha|^\sigma}{t^{\frac{d(p-1)}{p}+\sigma}}\|(\rho,U)\|_{S^{{\varepsilon_0}}_a}^\sigma\left(
\sup_{s\in [0,t]}\langle s\rangle^a\|F(s)\|_{\dot{F}^\sigma_{1,\infty}}+\sup_{s\in [0,t]} \langle s\rangle^{\frac{d(p-1)}{p}+a}	\|F(s)\|_{\dot{F}^\sigma_{p,\infty}}\right)\int_{0}^{t}\frac{1}{\langle s\rangle^{\sigma(d-3+a)+a}}ds\\
&\lesssim_{a,\mathbf{c}_1}\frac{|\alpha|^\sigma}{t^{\frac{d(p-1)}{p}+\sigma}}\|(\rho,U)\|_{S^{{\varepsilon_0}}_a}^\sigma\left(
\sup_{s\in [0,t]}\langle s\rangle^a\|F(s)\|_{\dot{F}^\sigma_{1,\infty}}+\sup_{s\in [0,t]} \langle s\rangle^{\frac{d(p-1)}{p}+a}	\|F(s)\|_{\dot{F}^\sigma_{p,\infty}}\right),
\end{align*}
where in the last line, we use the fact that $a^2+a>1$.
Thus, 
\begin{align*}
t^{\frac{d(p-1)}{p}+\sigma}\sup_\alpha\frac{	\|\mathcal{T}_{\alpha}^3[F,\eta](t)\|_{L^p}}{|\alpha|^\sigma}&\lesssim _{a,\mathbf{c}_1} \|(\rho,U)\|_{S^{{\varepsilon_0}}_a}^\sigma\left(
\sup_{s\in [0,t]}\langle s\rangle^a\|F(s)\|_{\dot{F}^\sigma_{1,\infty}}+\sup_{s\in [0,t]} \langle s\rangle^{\frac{d(p-1)}{p}+a}	\|F(s)\|_{\dot{F}^\sigma_{p,\infty}}\right).
\end{align*}
\textbf{2.2.4)} Estimate $	\mathcal{T}_{\alpha}^4[F,\eta](t,x)$.\\  One has
\begin{align}\nonumber
&
|	\mathcal{T}_{\alpha}^4[F,\eta](t,x)|\lesssim _{\mathbf{c}_1}\int_{0}^{t}\int_{\mathbb{R}^d}
\left|	F\left(s,Z_5(x)\right)\right| \left|W_{s,t}(x-tv,v)-W_{s,t}(x-tv,v-\frac{\alpha}{t})\right|
\frac{	dvds}{\langle v\rangle^N}\\\nonumber
&\lesssim_{a,\mathbf{c}_1} \min\left\{1,\frac{|\alpha|}{t}\right\} \|(\rho,U)\|_{S^{{\varepsilon_0}}_a} \int_{0}^{t}\int_{\mathbb{R}^d}
\left|F\left(s,Y_{s,t}(x-tv,v-\frac{\alpha}{t})+x-(t-s)v-\frac{s\alpha}{t}\right)\right|  \frac{1}{\langle s\rangle^{d-2+a}}
\frac{	dvds}{\langle v\rangle^N}\\\label{T4fml 1}
&\lesssim_{a,\mathbf{c}_1} \frac{|\alpha|^\sigma}{\ t
	^\sigma}   \|(\rho,U)\|_{S^{{\varepsilon_0}}_a} \int_{0}^{t}\int_{\mathbb{R}^d}
\left|F\left(s,Y_{s,t}(x-tv,v-\frac{\alpha}{t})+x-(t-s)v-\frac{s\alpha}{t}\right)\right|  \frac{dv}{\langle v\rangle^N}\frac{ds}{\langle s\rangle^{d-2+a}}.	
\end{align}
Then, applying Lemma \ref{mathcal H1H2} with $\mathcal{H}=F$ and $\varphi(x,v)=Y_{s,t}(x-tv,v-\frac{\alpha}{t})-\frac{s\alpha}{t}$, one gets
\begin{align*}
\|\mathcal{T}_{\alpha}^4[F,\eta](t)\|_{L^p}&\lesssim_{a,\mathbf{c}_1} \underbrace{\frac{|\alpha|^\sigma}{\ t
		^{\sigma+\frac{d(p-1)}{p}}}  \|(\rho,U)\|_{S^{{\varepsilon_0}}_a} \int_{0}^{t/2}
	\|F(s)\|_{L^1}  \frac{ds}{\langle s\rangle^{d-2+a}}}_{\text{Thanks to }\eqref{mathcal H 1}}+ \underbrace{\frac{|\alpha|^{\sigma}}{\ t
		^\sigma}  \|(\rho,U)\|_{S^{{\varepsilon_0}}_a} \int_{t/2}^{t}
	\|F(s)\|_{L^p}  \frac{ds}{\langle s\rangle^{d-2+a}}}_{\text{Thanks to }\eqref{mathcal H p}} \\&\lesssim_{a,\mathbf{c}_1} 
\frac{|\alpha|^\sigma}{t^{\frac{d(p-1)}{p}+\sigma}} \|(\rho,U)\|_{S^{{\varepsilon_0}}_a}  \left( \sup_{s\in [0,t]} \langle s\rangle^\frac{d(p-1)}{p} \|F(s)\|_{L^p}+\sup_{s\in [0,t]}\|F(s)\|_{L^1}\right).
\end{align*}
Hence, we obtain
\begin{align*}
t^{\frac{d(p-1)}{p}+\sigma}\sup_\alpha\frac{	\|\mathcal{T}_{\alpha}^4[F,\eta](t)\|_{L^p}}{|\alpha|^\sigma}&\lesssim_{a,\mathbf{c}_1}  \|(\rho,U)\|_{S^{{\varepsilon_0}}_a}^\sigma\left(
\sup_{s\in [0,t]}\langle s\rangle^a\|F(s)\|_{\dot{F}^\sigma_{1,\infty}}+\sup_{s\in [0,t]} \langle s\rangle^{\frac{d(p-1)}{p}+a}	\|F(s)\|_{\dot{F}^\sigma_{p,\infty}}\right).
\end{align*}
Summing up, we conclude that for $t\geq1$,
\begin{align}\label{d2}
&\sum_{p=1,\infty} t^{\frac{d(p-1)}{p}+\sigma}\|\mathcal{T}[F,\eta](t)\|_{\dot{B}^\sigma_{p,\infty}}\lesssim _{a,\mathbf{c}_1} \|(\rho,U)\|_{S^{{\varepsilon_0}}_a}^\sigma\sum_{p=1,\infty} \left(\sup_{s\in [0,t]} \langle s\rangle^{\frac{d(p-1)}{p}} \|F(s)\|_{L^p}+\sup_{s\in [0,t]}\langle s\rangle^{\frac{d(p-1)}{p}+a}	\|F(s)\|_{\dot{F}^\sigma_{p,\infty}}\right). 
\end{align}
Hence, we proved the Lemma with \eqref{d3} and  \eqref{d2}.
\end{proof}\\
\begin{remark}
	For $a_1\in(0,1)$, if we also have $\|(\rho,U)\|_{S^{{\varepsilon_0}}_{a_1}}\lesssim 1$, 
	\begin{align}
		\left|	\mathcal{T}^i_\alpha[F,\eta](t,x)\right|\lesssim_{a_1,\mathbf{c}_1} \|(\rho,U)\|_{S^{{\varepsilon_0}}_{a_1}} \int_{0}^{t}\int_{\mathbb{R}^{d}}|F(s,x-(t-s)v)|\frac{1}{\langle s\rangle^{d-2+a_1}}\frac{dvds}{\langle v\rangle^N},~~i=1,2.
	\end{align} 
\end{remark}
\begin{proposition}\label{mathcalR sigma est}
Let $a$, $\sigma$ and $\rho$ satisfy the same conditions  in Lemma \ref{esT}. Then we have
\begin{equation}\label{z13}
\|\mathcal{R}(\rho,U)\|_\sigma\lesssim_{a,M^*}  \|(\rho,U)\|_{S^{{\varepsilon_0}}_a}^{1+\sigma}.
\end{equation}
\end{proposition}
\begin{proof} Since $
\langle v\rangle^N	\sum_{j=0}^2\left|\nabla^j\left(\nabla \mu(v)\right)\right|\leq M^*$, thus, by Lemma \ref{esT}, we have 
$$
\|\mathcal{R}(\rho,U)\|_\sigma\lesssim_{a,M^*}	\|(\rho,U)\|_{S^{{\varepsilon_0}}_a}^\sigma\sum_{p=1,\infty}\left(\langle s\rangle^\frac{d(p-1)}{p}\|E\|_{L^p}+\sup_{s\in [0,t]}\langle s\rangle^{a+\frac{d(p-1)}{p}}	\|E(s)\|_{\dot{F}^\sigma_{1,\infty}}\right)\lesssim_{a,M^*}  	\|(\rho,U)\|_{S^{{\varepsilon_0}}_a}^{1+\sigma}.$$
This implies the result.
\end{proof}
\begin{proposition}\label{mathcalR difference}
Let $a\in (\frac{\sqrt{5}-1}{2},1)$ and  $\widetilde{\varepsilon}_0>0$ be as in Proposition \eqref{estimates about Y and W}.  Assume $\|(\rho_1,U_1)\|_{S^{{\varepsilon_0}}_a}, \|(\rho_2,U_2)\|_{S^{{\varepsilon_0}}_a}\leq \varepsilon_0\leq\widetilde{\varepsilon}_0$. There holds, 
\begin{equation}
\|\mathcal{R}(\rho_1,U_1)-\mathcal{R}(\rho_2,U_2)\|_a\lesssim_{a,M^*}  \Big(\|(\rho_1,U_1)\|_{S^{{\varepsilon_0}}_a}^a+\|(\rho_2,U_2)\|_{S^{{\varepsilon_0}}_a}\Big) \|(\rho_1-\rho_2,U_1-U_2)\|_{S^{{\varepsilon_0}}_a}.
\end{equation}
\end{proposition}
\begin{proof}
	In this proof, we recall the notations $X_{s,t}^1$, $X_{s,t}^2$, $Y_{s,t}^1$, $Y_{s,t}^2$, $V_{s,t}^1$, $V_{s,t}^2$, $\mathbf{Y}_{s,t}$, $\mathbf{W}_{s,t}$ in Proposition \ref{Y W difference}, and $X_{s,t}(x,v,\varpi),~Y_{s,t}(x,v,\varpi),~V_{s,t}(x,v,\varpi), W_{s,t}(x,v,\varpi)$ in \eqref{XV varpi}. Then we have
\begin{align*}
&	\mathbf{R}(\rho_1,\rho_2,U_1,U_2)(t,x):=\mathcal{R}(\rho_1,U_1)(t,x)-	\mathcal{R}(\rho_2,U_2)(t,x)\\=&\int_{0}^{t}\int_{\mathbb{R}^{d}}(E^1-E^2)(s,x-(t-s)v)\cdot\nabla_v\mu(v)dvds-\int_{0}^{t}\int_{\mathbb{R}^{d}}(E^1-E^2)(s,X_{s,t}^1(x,v))\cdot\nabla_v\mu(V_{s,t}^1(x,v))dvds\\&+
\int_{0}^{t}\int_{\mathbb{R}^{d}}\left(E^2(s,X_{s,t}^1(x,v))-E^2(s,X_{s,t}^2(x,v))\right).\nabla\mu(V_{s,t}^1(x,v))\\&+
\int_{0}^{t}\int_{\mathbb{R}^{d}}E^2(s,X_{s,t}^2(x,v)).\left(\nabla\mu(V_{s,t}^1(x,v))-\nabla\mu(V_{s,t}^2(x,v))\right)
\\=&:\sum_{i=1}^{d}\mathcal{T}[(E^1-E^2)_i,\partial_i\mu](t,x)+ 	\mathbf{R}^1(\rho_1,\rho_2,U_1,U_2)(t,x)+\mathbf{R}^2(\rho_1,\rho_2,U_1,U_2)(t,x).
\end{align*}
Thanks to Lemma \ref{esT} with $F=(E^1-E^2)_i$, $\eta=\partial_i\mu$ and $\sigma=a$, we get
\begin{equation}
\sum_{i=1}^{d}\|\mathcal{T}[(E^1-E^2)_i,\partial_i\mu]\|_{a}\lesssim_{a,M^*}  \|(\rho_1,U_1)\|_{S^{{\varepsilon_0}}_a}^a\|(\rho_1-\rho_2,U_1-U_2)\|_{S^{{\varepsilon_0}}_a}.
\end{equation}
\textbf{(1)} Estimate on $\mathbf{R}^1(\rho_1,\rho_2,U_1,U_2)(t,x)$, one has 
\begin{align*}
\mathbf{R}^1(\rho_1,\rho_2,U_1,U_2)(t,x)=-\sum_{i=1}^{d}	\int_{0}^{1}\int_{0}^{t}\int_{\mathbb{R}^{d}}\mathbf{Y}_{s,t}(x-vt,v)\nabla E^2_i(s,X_{s,t}(x,v,\varpi))\partial_i\mu(V_{s,t}^1(x,v))d\varpi ds dv.
\end{align*}
Hence, 
\begin{align*}
&|\mathbf{R}^1(\rho_1,\rho_2,U_1,U_2)(t,x)|\leq  \sup_{0\leq s\leq t}\langle s\rangle ^{d-2+a}\|\mathbf{Y}_{s,t}\|_{L^\infty}	\int_{0}^{1}\int_{0}^{t}\int_{\mathbb{R}^{d}} \frac{|\nabla E^2(s,X_{s,t}(x,v,\varpi))|}{\langle s\rangle ^{d-2+a}}|\nabla\mu(V_{s,t}^1(x,v))| dsdvd\varpi\\
&\qquad\overset{\eqref{results of Y12 small}}{\lesssim_{a,M^*}} \|(\rho_1-\rho_2,U_1-U_2)\|_{S^{{\varepsilon_0}}_a}\int_{0}^{1}\int_{0}^{t}\int_{\mathbb{R}^{d}} \langle s\rangle ^{-d+2-a}|\nabla E^2(s,X_{s,t}(x,v,\varpi))|\frac{1}{\langle v \rangle^N} dsdvd\varpi.
\end{align*}
Note that $X_{s,t}(x,v,\varpi)=x-(t-s)v+Y_{s,t}(x-tv,v,\varpi)$, then we  apply Lemma \ref{mathcal Halpha} with $\mathcal{H}=\nabla E^2$ and $\varphi(x,v)=Y_{s,t}(x-tv,v,\varpi)$ to obtain that for $t> 1$,
\begin{align}\label{mathcalR1 Linfty}
	\begin{split}
		&\|\mathbf{R}^1(\rho_1,\rho_2,U_1,U_2)(t,x)\|_{L^p}
	\\&\quad	\lesssim_{a,M^*}\|(\rho_1-\rho_2,U_1-U_2)\|_{S^{{\varepsilon_0}}_a}\int_{0}^{1}\int_{0}^{t} \langle s\rangle ^{-d+2-a}\left \|\int_{\mathbb{R}^{d}}\nabla E^2(s,X_{s,t}(x,v,\varpi))\frac{1}{\langle v \rangle^N} dv\right\|_{L^p_x} dsd\varpi\\
		&\quad\lesssim_{a,M^*}\|(\rho_1-\rho_2,U_1-U_2)\|_{S^{{\varepsilon_0}}_a}\left( \int_{0}^{t/2}\langle s\rangle ^{-d+2-a} t ^{-\frac{d(p-1)}{p}}\left \|\nabla E^2(s)\right\|_{L^1_x} ds+\int_{t/2}^{t} \langle s\rangle ^{-d+2-a} \left \|\nabla E^2(s)\right\|_{L^p_x} ds\right)\\
		&\quad\lesssim_{a,M^*}\frac{\|(\rho_1-\rho_2,U_1-U_2)\|_{S^{{\varepsilon_0}}_a}}{t^\frac{d(p-1)}{p}}\left(\int_{0}^{t/2}+\int_{t/2}^{t}\right)\frac{\langle s\rangle^{2-a}}{\langle s\rangle^{d+a}}ds\lesssim\frac{\|(\rho_1-\rho_2,U_1-U_2)\|_{S^{{\varepsilon_0}}_a}}{t^\frac{d(p-1)}{p}}.
	\end{split}
\end{align}
Meanwhile, for $0\leq t\leq 1$
\begin{align}\nonumber
	\|\mathbf{R}^1(\rho_1,\rho_2,U_1,U_2)(t,x)\|_{L^p}&\lesssim_{a,M^*}\|(\rho_1-\rho_2,U_1-U_2)\|_{S^{{\varepsilon_0}}_a}\int_{0}^{t}\langle s\rangle ^{-d+2-a}\|\nabla E^2(s)\|_{L^p} ds\\&\lesssim_{a,M^*}\|(\rho_2,U_2)\|_{S^{{\varepsilon_0}}_a}\|(\rho_1-\rho_2,U_1-U_2)\|_{S^{{\varepsilon_0}}_a}.\label{mathcal R1 L1}
\end{align}
Then, with \eqref{mathcal R1 L1} and \eqref{mathcalR1 Linfty}, we can obtain 
\begin{align}\label{mathbfR1 1}
\sum_{p=1,\infty}\langle t\rangle^{\frac{d(p-1)}{p}} \|\mathbf{R}^1(\rho_1,\rho_2)\|_{L^p}\lesssim_{a,M^*} \|(\rho_2,U_2)\|_{S^{{\varepsilon_0}}_a}\|(\rho_1-\rho_2,U_1-U_2)\|_{S^{{\varepsilon_0}}_a}.
\end{align}
Now we deal with $\frac{|\delta_{\alpha}\mathbf{R}^1(\rho_1,\rho_2)(t,x)|}{|\alpha|^a}$, for the case $t\leq 1$ we have
\begin{align*}
&\frac{|\delta_{\alpha}\mathbf{R}^1(\rho_1,\rho_2)(t,x)|}{|\alpha|^a}\lesssim_{a,M^*} \|(\rho_1-\rho_2,U_1-U_2)\|_{S^{{\varepsilon_0}}_a}	\int_{0}^{1}\int_{0}^{t}\int_{\mathbb{R}^{d}}\langle s\rangle^{-d+2+a}|\nabla E^2(s,X_{s,t}(x,v,\varpi))|\frac{dv}{\langle v\rangle^N}dsd\varpi\\&\qquad
+ \|(\rho_1-\rho_2,U_1-U_2)\|_{S^{{\varepsilon_0}}_a}	\int_{0}^{1}\int_{0}^{t}\int_{\mathbb{R}^{d}}\langle s\rangle^{-d+2}\sup_z\frac{|\delta_{z}\nabla E^2(s,.)(X_{s,t}(x,v,\varpi))|}{|z|^a}\frac{dv}{\langle v\rangle^N}dsd\varpi\\&\qquad\quad
+\|(\rho_1-\rho_2,U_1-U_2)\|_{S^{{\varepsilon_0}}_a}	\int_{0}^{1}\int_{0}^{t}\int_{\mathbb{R}^{d}}\langle s\rangle^{-d+2-a}|\nabla E^2(s,X_{s,t}(x-\alpha,v,\varpi))| \frac{dv}{\langle v\rangle^N} dsd\varpi.
\end{align*}
Then thanks to \eqref{mathcal H alpha p}, we deduce that 
\begin{align*}
\frac{\|\delta_{\alpha}\mathbf{R}^1(\rho_1,\rho_2)\|_{L^p}}{|\alpha|^a}\lesssim_{a,M^*}&\|(\rho_1-\rho_2,U_1-U_2)\|_{S^{{\varepsilon_0}}_a}	\int_{0}^{t}\left(\langle s\rangle^{-d+2-a}\|\nabla E^2(s)\|_{L^p} +\langle s\rangle^{-d+2}\|\nabla E^2(s)\|_{\dot{F}^a_{p,\infty}}\right) ds\\\lesssim_{a,M^*}&
\|(\rho_1-\rho_2,U_1-U_2)\|_{S^{{\varepsilon_0}}_a}\|(\rho_2,U_2)\|_{S^{{\varepsilon_0}}_a}	\int_{0}^{t}\left(\frac{\langle s\rangle^{-d+2-a}}{ \langle s\rangle^{\frac{d(p-1)}{p}+a}} +\frac{\langle s\rangle^{-d+2}}{ \langle s\rangle^{\frac{d(p-1)}{p}+a}}\right)ds\\\lesssim_{a,M^*}&
\|(\rho_1-\rho_2,U_1-U_2)\|_{S^{{\varepsilon_0}}_a}\|(\rho_2,U_2)\|_{S^{{\varepsilon_0}}_a}.
\end{align*}
For the case $t>1$, $\frac{|\delta_{\alpha}\mathbf{R}^1(\rho_1,\rho_2,U_1,U_2)(t,x)|}{|\alpha|^a}$ can be recast as 
\begin{align*}
&\frac{|\delta_{\alpha}\mathbf{R}^1(\rho_1,\rho_2,U_1,U_2)(t,x)|}{|\alpha|^a}\lesssim_{a,M^*} \frac{1}{t^a}\|(\rho_1-\rho_2,U_1-U_2)\|_{S^{{\varepsilon_0}}_a}	\int_{0}^{1}\int_{0}^{t}\int_{\mathbb{R}^{d}}\langle s\rangle^{-d+2}|\nabla E^2(s,X_{s,t}(x,v,\varpi))|\frac{dv}{\langle v\rangle^N}dsd\varpi\\&
+\frac{1}{t^a} \|(\rho_1-\rho_2,U_1-U_2)\|_{S^{{\varepsilon_0}}_a}	\int_{0}^{1}\int_{0}^{t}\int_{\mathbb{R}^{d}}\langle s\rangle^{-d+2}\sup_z\frac{|\delta_{z}\nabla E^2(s,.)(X_{s,t}(x,v,\varpi))|}{|z|^a}\frac{dv}{\langle v\rangle^N}dsd\varpi\\&+\frac{1}{t^a} \|(\rho_1-\rho_2,U_1-U_2)\|_{S^{{\varepsilon_0}}_a}
\int_{0}^{1}\int_{0}^{t}\int_{\mathbb{R}^{d}}\langle s\rangle^{-d+2-a}|\nabla E^2_i(s,X_{s,t}(x-\alpha,v,\varpi))|\left(\frac{1}{\langle v\rangle^{N}}+\frac{1}{\langle v-\frac{\alpha}{t}\rangle^{N}}\right)dvdwd\varpi.
\end{align*}
Thus, it holds that for $t>1$,
\begin{align*}
&	\frac{\|\delta_{\alpha}\mathbf{R}^1(\rho_1,\rho_2,U_1,U_2)\|_{L^p}}{|\alpha|^a}\\&\quad\lesssim_{a,M^*}\frac{1}{t^a} \|(\rho_1-\rho_2,U_1-U_2)\|_{S^{{\varepsilon_0}}_a}\int_{0}^{t/2}\left(\langle s\rangle^{-d+2-a}\|\nabla E^2(s)\|_{L^1} +\langle s\rangle^{-d+2}\|\nabla E^2(s)\|_{\dot{F}^a_{1,\infty}}\right) ds\\
	&\qquad+\frac{1}{t^{a+\frac{d(p-1)}{p}}} \|(\rho_1-\rho_2,U_1-U_2)\|_{S^{{\varepsilon_0}}_a}\int_{t/2}^{t}\left(\langle s\rangle^{-d+2-a}\|\nabla E^2(s)\|_{L^p} +\langle s\rangle^{-d+2}\|\nabla E^2(s)\|_{\dot{F}^a_{p,\infty}}\right) ds\\
&\quad\lesssim_{a,M^*}\frac{1}{t^{a+\frac{d(p-1)}{p}}} \|(\rho_1-\rho_2,U_1-U_2)\|_{S^{{\varepsilon_0}}_a}\|(\rho_2,U_2)\|_{S^{{\varepsilon_0}}_a}.
\end{align*}
Then, 
\begin{align}
\sum_{p=1,\infty}\langle t\rangle^{a+\frac{d(p-1)}{p}} \frac{\|\delta_{\alpha}\mathbf{R}^1(\rho_1,\rho_2,U_1,U_2)\|_{L^p}}{|\alpha|^a}\lesssim_{a,M^*} \|(\rho_2,U_2)\|_{S^{{\varepsilon_0}}_a}\|(\rho_1-\rho_2,U_1-U_2)\|_{S^{{\varepsilon_0}}_a}.
\end{align}
Combining this with \eqref{mathbfR1 1}, one gets
\begin{equation}
\|\mathbf{R}^1(\rho_1,\rho_2,U_1,U_2)\|_{S^{{\varepsilon_0}}_a}\lesssim_{a,M^*} \|(\rho_2,U_2)\|_{S^{{\varepsilon_0}}_a}\|(\rho_1-\rho_2,U_1-U_2)\|_{S^{{\varepsilon_0}}_a}.
\end{equation}
\textbf{(2)} Estimate on $\mathbf{R}^2(\rho_1,\rho_2,U_1,U_2)(t,x)$.
We infer that
\begin{align*}
&\mathbf{R}^2(\rho_1,\rho_2,U_1,U_2)(t,x)\\&=-\sum_{i=1}^{d}	\int_{0}^{1}\int_{0}^{t}\int_{\mathbb{R}^{d}}\mathbf{W}_{s,t}(x-vt,v)E^2_i(s,X_{s,t}^2(x,v))(\nabla\partial_i\mu)(V_{s,t}(x,v,\varpi))dvdsd\varpi.
\end{align*}
Analogously, we  obtain 
\begin{equation}
\|	\mathbf{R}^2(\rho_1,\rho_2,U_1,U_2)\|_{S^{{\varepsilon_0}}_a}\lesssim_{a,M^*} \|(\rho_2,U_2)\|_{S^{{\varepsilon_0}}_a}\|(\rho_1-\rho_2,U_1-U_2)\|_{S^{{\varepsilon_0}}_a},
\end{equation}
which complete the proof.
\end{proof}\\

\begin{lemma} \label{lem1}Let $H_0:\mathbb{R}^d\to \mathbb{R}$ and $A_0:\mathbb{R}^d\to \mathbb{R}^d$ and $\lambda\in \mathbb{R}$. Then, for any muti-index $\alpha\not =0$, 
\begin{align*}
&	|\partial^\gamma_x\left[H_0(\lambda x+A_0(x))\right]-(\partial_x^\gamma H_0)(\lambda x+A_0(x))\lambda^{|\gamma|}|\lesssim |(\nabla^{|\gamma|} H_0)(\lambda x+A_0(x))||\nabla A_0(x)|(|\lambda|+|\nabla A_0(x)|)^{|\gamma|-1}\\&\quad\quad\quad+\sum_{k=1}^{|\gamma|-1}\sum_{j=0}^{k-1}\sum_{i=2}^{|\gamma|-j}|(\nabla^k H_0)(\lambda x+A_0(x))|(|\lambda|+|\nabla A_0(x)|)^j |\nabla^iA_0(x)|^{\frac{|\gamma|-j}{i}}.
\end{align*}
\end{lemma}
\begin{proof}  Assume $\partial_x^{\gamma}=\partial_{x_1}^{\gamma_1}...\partial_{x_d}^{\gamma_d}$. We know that, for $\tilde{A}_0(x)=\lambda x+A_0(x)$, we have
\begin{align*}
&	\left|\partial^\gamma_x\left[H_0( \lambda x+A_0(x))\right]-	\mathbf{Q}\right|\lesssim \sum_{k=1}^{|\gamma|-1} |(\nabla^kH_0)(\tilde{A}_0(x))|  \sum_{\substack{m_1,...,m_k\geq 1\\m_1+..+m_k=|\gamma|}}\prod_{j=1}^k |\nabla^{m_j} \tilde{A}_0(x)|,
\end{align*}
with 
\begin{equation}\label{23}
\mathbf{Q}=(\nabla^{|\gamma|}H_0)(\lambda x+A_0(x)):\left[ (\partial_{x_1} \tilde{A}_0(x))^{\otimes{\gamma_1}}\otimes (\partial_{x_2} \tilde{A}_0(x))^{\otimes{\gamma_2}}\otimes...\otimes (\partial_{x_d} \tilde{A}_0(x))^{\otimes{\gamma_d}}\right].
\end{equation}
It is easy to check that 
\begin{align*}
&	\sum_{k=1}^{|\gamma|-1} |(\nabla^kH_0)(\tilde{A}_0(x))|  \sum_{\substack{m_1,...,m_k\geq 1\\m_1+..+m_k=|\gamma|}}\prod_{j=1}^k |\nabla^{m_j} \tilde{A}_0(x)|\\&\qquad\lesssim \sum_{k=1}^{|\gamma|-1}\sum_{j=0}^{k-1}\sum_{i=2}^{|\gamma|-j}|(\nabla^k H_0)(\lambda x+A_0(x))|(|\lambda|+|\nabla A_0(x)|)^j |\nabla^iA_0(x)|^{\frac{|\gamma|-j}{i}},
\end{align*}
and 
\begin{align}\label{233}
\left|	\mathbf{Q}-(\partial_x^\gamma H_0)(\lambda x+A_0(x))\lambda^{|\gamma|}\right|\lesssim |(\nabla^{|\gamma|} H_0)(\lambda x+A_0(x))||\nabla A_0(x)|(|\lambda|+|\nabla A_0(x)|)^{|\gamma|-1}.
\end{align}
Combining \eqref{23} with \eqref{233}, we finish the proof.
\end{proof}	
\begin{proposition}\label{R estimate}Let $a\in (\frac{\sqrt{5}-1}{2},1)$ and  assume $\|(\rho,U)\|_{S^{{\varepsilon_0}}_a}\leq \varepsilon_0\leq\widetilde{\varepsilon}_0$ with $\widetilde{\varepsilon}_0$ mentioned in Proposition \ref{estimates about Y and W}. Then we have for $b\in(\frac{1+\delta_0}{2-\delta_0},1-2\delta_0)$,
\begin{equation}\label{higherRnorm}
\|	\mathcal{R}(\rho,U)\|_{l+1,b}\lesssim_{a,b,\delta_0,\mathbf{c},\mathbf{c}',C_{A,l}}1,
\end{equation} provided $\|(\rho,U)\|_{l,1-\delta_0}\leq \mathbf{c}$,  $	\sum_{j=2}^{l+1} \sup_{|r|\leq 1}	\left|{A^{(j)}(r)}\right|\leq C_{A,l}$  and $\mu$ satisfies  \begin{equation}\label{z14}
\sum_{i=1}^{l+3}\langle v\rangle^N\left|\nabla^i\mu(v)\right|\leq\mathbf{c}',
\end{equation}
where $l\geq 0$ , $0<\delta_0<\frac{4-\sqrt{13}}{3}$ and $N>d$. 
\end{proposition}
\begin{proof} In this proof, we denote $C=C(a,b,\delta_0,\mathbf{c},\mathbf{c}',C_{A,l})$.\\
\textbf{(1)}We firstly consider $t\geq 1$.
Set 
\begin{align*}
&F_{1,1}(x)=	\delta_{\frac{s\alpha}{t}}E_i(s,.)\left(Z_1(x)\right),~~\eta_{1,1}(x)=\partial_i\mu\left(Z_2\right),\\& 
F_{1,2}(x)=\delta_{\frac{s\alpha}{t}}E_i(s,.)(w+\frac{s}{t}(x-w)),
\eta_{1,2}(x)=\partial_i\mu(\frac{x-w}{t}),\\&
F_{2,1}(x)=	E_i\left(s,Z_1(x)\right),
\eta_{2,1}(x)=\delta_{\frac{\alpha}{t}}\partial_i\mu\left(Z_2(x)\right),\\&
F_{2,2}(x)=E_i(s,w+\frac{s}{t}(x-w)),~~
\eta_{2,2}(x)=\delta_{\frac{\alpha}{t}}\partial_i\mu(\frac{x-w}{t}),\\&
F_{3}(x)=E_i\left(s,Z_1(x-\alpha)\right)-E_i\left(s,Z_4(x)\right),~~
\eta_3(x)=\partial_i\mu\left(Z_3(x)\right)-\partial_i\mu\left(Z_2(x)\right),
\end{align*}
where $Z_1,Z_2,Z_3,Z_4$ are defined	 in \eqref{Z1234} and  $w=x-vt$, then we have, 
\begin{align}\nonumber
&	\delta_{\alpha}\mathcal{T}[E_i,\partial_i\mu](t,x)= \sum_{j=1}^{4} \mathcal{T}_{\alpha}^j[E_i,\partial_i\mu](t,x),
\end{align} 
where
\begin{align*}			
&\mathcal{T}_{\alpha}^1[E_i,\partial_i\mu](t,x)=-\int_{0}^{t}\int_{\mathbb{R}^d}
\left(F_{1,1}(x)\eta_{1,1}(x)-	F_{1,2}(x)\eta_{1,2}(x)\right)\frac{dwds}{t^d},\\&
\mathcal{T}_{\alpha}^2[E_i,\partial_i\mu](t,x)=\int_{0}^{t}\int_{\mathbb{R}^d}
\left(F_{2,1}(x)\eta_{2,1}(x)-	F_{2,2}(x)\eta_{2,2}(x)\right)\frac{dwds}{t^d},\\&
\mathcal{T}_{\alpha}^3[E_i,\partial_i\mu](t,x)=\int_{0}^{t}\int_{\mathbb{R}^d} F_3(x)\eta_{1,1}(x)
\frac{dwds}{t^{d}},\\&
\mathcal{T}_{\alpha}^4[E_i,\partial_i\mu](t,x)=\int_{0}^{t}\int_{\mathbb{R}^d} F_{2,1}(x-\alpha)\eta_{3}(x)
\frac{dwds}{t^{d}}.
\end{align*}
Hence, we obtain	
\begin{align}\nonumber
&|\nabla^{l+1}_x	\delta_{\alpha}\mathcal{T}^j[E_i,\partial_i\mu](t,x)|\lesssim\sum_{|\beta_1|+|\beta_2|=l+1} \left|\mathcal{T}_{\alpha}^j[E_i,\partial_i\mu,\beta_1,\beta_2](t,x)\right|,
\end{align} 
where 
\begin{align*}
&\mathcal{T}_{\alpha}^1[E_i,\partial_i\mu,\beta_1,\beta_2](t,x)=-\int_{0}^{t}\int_{\mathbb{R}^d}
\left(\partial_x^{\beta_1}F_{1,1}(x)\partial_x^{\beta_2}\eta_{1,1}(x)-	\partial_x^{\beta_1}F_{1,2}(x)\partial_x^{\beta_2}\eta_{1,2}(x)\right)\frac{dwds}{t^d},\\&
\mathcal{T}_{\alpha}^2[E_i,\partial_i\mu,\beta_1,\beta_2](t,x)=\int_{0}^{t}\int_{\mathbb{R}^d}
\left(\partial_x^{\beta_1}F_{2,1}(x)\partial_x^{\beta_2}\eta_{2,1}(x)-	\partial_x^{\beta_1}F_{2,2}(x)\partial_x^{\beta_2}\eta_{2,2}(x)\right)\frac{dwds}{t^d},\\&
\mathcal{T}_{\alpha}^3[E_i,\partial_i\mu,\beta_1,\beta_2](t,x)=\int_{0}^{t}\int_{\mathbb{R}^d} \partial_x^{\beta_1}F_3(x)\partial_x^{\beta_2}\eta_{1,1}(x)
\frac{dwds}{t^{d}},\\&
\mathcal{T}_{\alpha}^4[E_i,\partial_i\mu,\beta_1,\beta_2](t,x)=\int_{0}^{t}\int_{\mathbb{R}^d} \partial_x^{\beta_1}F_{2,1}(x-\alpha)\partial_x^{\beta_2}\eta_{3}(x)
\frac{dwds}{t^{d}}.
\end{align*}		
\textbf{1.1)} 	 Estimate of terms in $\nabla_x^{l+1}\mathcal{T}^1_\alpha[E_i,\partial_i\mu]$. \\
Note that
\begin{align*}
&\partial_x^{\beta_1}F_{1,1}(x)\partial_x^{\beta_2}\eta_{1,1}(x)-	\partial_x^{\beta_1}F_{1,2}(x)\partial_x^{\beta_2}\eta_{1,2}(x)\\				&\quad=\left(\partial_x^{\beta_1}F_{1,1}(x)-\left(\frac{s}{t}\right)^{|\beta_1|}\partial_x^{\beta_1}	\delta_{\frac{s\alpha}{t}}E_i(s,.)(Z_1)\right)
\left(	\partial_x^{\beta_2}\eta_{1,1}(x)-\left(\frac{1}{t}\right)^{|\beta_2|}(\partial_x^{\beta_2}\partial_i\mu)(Z_2)\right)\\
&\qquad+\left(\partial_x^{\beta_1}F_{1,1}(x)-\left(\frac{s}{t}\right)^{|\beta_1|}\partial_x^{\beta_1}	\delta_{\frac{s\alpha}{t}}E_i(s,.)(Z_1)\right)\left(\frac{1}{t}\right)^{|\beta_2|}(\partial_x^{\beta_2}\partial_i\mu)(Z_2)\\
&\qquad+\left(\frac{s}{t}\right)^{|\beta_1|}\partial_x^{\beta_1}	\delta_{\frac{s\alpha}{t}}E_i(s,.)(Z_1)\left(	\partial_x^{\beta_2}\eta_{1,1}(x)-\left(\frac{1}{t}\right)^{|\beta_2|}(\partial_x^{\beta_2}\partial_i\mu)(Z_2)\right)\\
&\qquad+\left(\left(\frac{s}{t}\right)^{|\beta_1|}\partial_x^{\beta_1}	\delta_{\frac{s\alpha}{t}}E_i(s,.)(Z_1)\left(\frac{1}{t}\right)^{|\beta_2|}(\partial_x^{\beta_2}\partial_i\mu)(Z_2)-\partial_x^{\beta_1}F_{1,2}(x)\partial_x^{\beta_2}\eta_{1,2}(x)\right)\\
&\quad	:=\sum_{k=1}^{4}\mathcal{T}^{1,\beta_1,\beta_2}_{\alpha,k}(t,x,\frac{x-w}{t}).
\end{align*}
We first deal with the fourth term directly
\begin{align}\label{T alpha 1 higher 4}
\begin{split}
	&\sum_{|\beta_1|+|\beta_2|=l+1}\left\|\int_{0}^{t}\int_{\mathbb{R}^d}\mathcal{T}^{1,\beta_1,\beta_2}_{\alpha,4}(t,x,\frac{x-w}{t})\frac{dwds}{t^d}\right\|_{L^p_x}=\frac{1}{t^{l+1}}\sum_{|\beta_1|+|\beta_2|=l+1}\left\|\mathcal{T}^1_\alpha\big[s^{|\beta_1|}\partial^{\beta_1}\delta_{\frac{s\alpha}{t}}E_i,\partial^{\beta_2}\partial_i\mu\big]\right\|_{L^p_x}\\
	&\quad\overset{\eqref{T alpha lower 1}}\lesssim_{\delta_0,\mathbf{c}}\frac{|\alpha|^b}{t^{d+l+1+b}}\sum_{|\beta_1|=0}^{l+1}\left(\int_{0}^{t/2}|s|^{|\beta_1|+b}\|\nabla^{|\beta_1|}E_i(s)\|_{\dot{F}^b_{1,\infty}}\frac{ds}{\langle s\rangle^{d-1-\delta_0}}
\right.\\&\left.	+\int_{t/2}^{t}\sup_{0\leq s\leq t}|s|^{|\beta_1|+b+d}\|\nabla^{|\beta_1|}E_i(s)\|_{\dot{F}^b_{\infty,\infty}}\frac{ds}{\langle s\rangle^{d-1-\delta_0}}\right)\\
	&\quad\overset{\eqref{E estimate by rho'}}
	{	\lesssim_{b,\delta_0,\mathbf{c},C_{A,l}}}\frac{|\alpha|^b}{t^{d+l+1+b}}\int_{0}^{t}\frac{|s|^{l+1+b}}{\langle s\rangle^{(l+1-\delta_0)+(d-1-\delta_0)}}ds	\lesssim_{b,\delta_0,\mathbf{c},C_{A,l}}\frac{|\alpha|^b}{t^{d+l+1+b}}
	,
\end{split}
\end{align}
since $b<1-2\delta_0$.\\
Then we start to estimate $\mathcal{T}^{1,\beta_1,\beta_2}_{\alpha,k}$, where $k=1,2,3.$ 
Applying  Lemma \ref{lem1} and Proposition \ref{results of Y small},  \ref{proposition of higher derivative estimates Y W}, we get
\begin{align*}
&\left|\partial_x^{\beta_1}F_{1,1}(x)-\left(\frac{s}{t}\right)^{|\beta_1|}\partial_x^{\beta_1}	\delta_{\frac{s\alpha}{t}}E_i(s,.)(Z_1)\right|\lesssim \frac{\mathbf{1}_{|\beta_1|\geq1}}{t^{|\beta_1|}} |\nabla^{|\beta_1|} 	\delta_{\frac{s\alpha}{t}}E_i(s,.)(Z_1)|  |\nabla_vY_{s,t}|\left(s+|\nabla_vY_{s,t}|\right)^{|\beta_1|-1}\\&\qquad+\frac{\mathbf{1}_{|\beta_1|\geq1}}{t^{|\beta_1|}}\sum_{m=1}^{|\beta_1|-1}\sum_{j=0}^{m-1}\sum_{i=2}^{|\beta_1|-j}|\nabla^m \delta_{\frac{s\alpha}{t}}E_i(s,.)(Z_1)|(|s|+|\nabla_vY_{s,t}|)^j |\nabla_v^iY_{s,t}|^{\frac{|\beta_1|-j}{i}}
\\&\lesssim_{C}\mathbf{1}_{|\beta_1|\geq1}\left(\frac{\|(\rho,U)\|_{S^{{\varepsilon_0}}_a}}{t^{|\beta_1|}} |\nabla^{|\beta_1|} 	\delta_{\frac{s\alpha}{t}}E_i(s,.)(Z_1)| \langle s\rangle^{|\beta_1|-d+\delta_0}+ \frac{\|(\rho,U)\|_{S^{{\varepsilon_0}}_a}}{t^{|\beta_1|}}\sum_{m=1}^{|\beta_1|-1}|\nabla^m \delta_{\frac{s\alpha}{t}}E_i(s,.)(Z_1)|\langle s\rangle^{m-d+1}\right)\\
&\lesssim_{C}\mathbf{1}_{|\beta_1|\geq1}\frac{\|(\rho,U)\|_{S^{{\varepsilon_0}}_a}}{t^{|\beta_1|}}\sum_{m=1}^{|\beta_1|}|\nabla^m \delta_{\frac{s\alpha}{t}}E_i(s,.)(Z_1)|\langle s\rangle^{m-d+1}	,
\end{align*}
with $
\nabla_v^iY_{s,t}:=(\nabla_v^iY_{s,t})(w,\frac{x-w}{t}).$
Similarly, we also can estimate 
\begin{align*}
&\left|	\partial_x^{\beta_2}\eta_{1,1}(x)-\left(\frac{1}{t}\right)^{|\beta_2|}(\partial_x^{\beta_2}\partial_i\mu)(Z_2)\right|\lesssim \frac{\mathbf{1}_{|\beta_2|\geq1}}{t^{|\beta_2|}} |\nabla^{|\beta_2|} \partial_i\mu(Z_2)|  |\nabla_vW_{s,t}|(1+|\nabla_vW_{s,t}|)^{|\beta_2|-1}\\&\qquad+\frac{\mathbf{1}_{|\beta_2|\geq1}}{t^{|\beta_1|}}\sum_{m=1}^{|\beta_2|-1}\sum_{j=0}^{m-1}\sum_{i=2}^{|\beta_2|-j}|\nabla^m \partial_i\mu(Z_2)|(1+|\nabla_vW_{s,t}|)^j |\nabla_v^iW_{s,t}|^{\frac{|\beta_2|-j}{i}}\\&\overset{\eqref{results of Y small},  \eqref{proposition of higher derivative estimates Y W}}{\lesssim_{C}}\mathbf{1}_{|\beta_2|\geq1}\left(\frac{\|(\rho,U)\|_{S^{{\varepsilon_0}}_a}}{t^{|\beta_2|}}|\nabla^{|\beta_2|} \partial_i\mu(Z_2)| \langle s\rangle^{-d+\delta_0}+ \frac{\|(\rho,U)\|_{S^{{\varepsilon_0}}_a}}{t^{|\beta_2|}}\sum_{m=1}^{|\beta_2|}|\nabla^m \partial_i\mu(Z_2)|\langle s\rangle^{-d+1}\right)\\&
\quad\lesssim_{C}\mathbf{1}_{|\beta_2|\geq1}\frac{\|(\rho,U)\|_{S^{{\varepsilon_0}}_a}}{t^{|\beta_2|}}\frac{1}{\langle\frac{x-w}{t}\rangle^N}\langle s\rangle^{-d+1}.
\end{align*}
Thus, we have
\begin{align*}
\sum_{|\beta_1|+|\beta_2|=l+1}|\mathcal{T}^{1,\beta_1,\beta_2}_{\alpha,1}(t,x,\frac{x-w}{t})|\lesssim_{C}&\frac{\|(\rho,U)\|_{S^{{\varepsilon_0}}_a}^2}{t^{l+1}}\sum_{|\beta_1|=1}^{l}\sum_{m=1}^{|\beta_1|}|\nabla^m \delta_{\frac{s\alpha}{t}}E_i(s,.)(Z_1)|\langle s\rangle^{m-2d+2}\frac{1}{\langle\frac{x-w}{t}\rangle^N},\\
\sum_{|\beta_1|+|\beta_2|=l+1}|\mathcal{T}^{1,\beta_1,\beta_2}_{\alpha,2}(t,x,\frac{x-w}{t})|\lesssim_{C}&\frac{\|(\rho,U)\|_{S^{{\varepsilon_0}}_a}^2}{t^{l+1}}\sum_{|\beta_1|=1}^{l+1}\sum_{m=1}^{|\beta_1|}|\nabla^m \delta_{\frac{s\alpha}{t}}E_i(s,.)(Z_1)|\langle s\rangle^{m-d+1}\frac{1}{\langle\frac{x-w}{t}\rangle^N},\\
\sum_{|\beta_1|+|\beta_2|=l+1}|\mathcal{T}^{1,\beta_1,\beta_2}_{\alpha,3}(t,x,\frac{x-w}{t})|\lesssim_{C}&\frac{\|(\rho,U)\|_{S^{{\varepsilon_0}}_a}^2}{t^{l+1}}\sum_{|\beta_1|=0}^{l}
\left|\partial_x^{\beta_1}	\delta_{\frac{s\alpha}{t}}E_i(s,.)(Z_1)\right|
\langle s\rangle^{|\beta_1|-d+1}\frac{1}{\langle\frac{x-w}{t}\rangle^N}.
\end{align*}
Hence, we get
\begin{align}\label{T11}
\begin{split}
	\sum_{k=1}^{3}\sum_{|\beta_1|+|\beta_2|=l+1}|\mathcal{T}^{1,\beta_1,\beta_2}_{\alpha,k}(t,x,\frac{x-w}{t})|\lesssim_{C}&\frac{\|(\rho,U)\|_{S^{{\varepsilon_0}}_a}^2}{t^{l+1}}\sum_{|\beta_1|=1}^{l+1}\sum_{m=1}^{|\beta_1|}|\nabla^m \delta_{\frac{s\alpha}{t}}E_i(s,.)(Z_1)|\langle s\rangle^{m-d+1}\frac{1}{\langle\frac{x-w}{t}\rangle^N}.
\end{split}
\end{align} 
Note that we have \eqref{mathcal H p} and \eqref{mathcal H 1} in the Appendix with $\mathcal{H}=\nabla^m\delta_{\frac{s\alpha}{t}}E_i$ and $\varphi=Y_{s,t}(x,v)$, we make the change of variables $v=\frac{x-w}{t}$, then we apply \eqref{E estimate by rho'}  to obtain
\begin{align}\label{T alpha 1 higher 123}
\begin{split}
	&\left\|\int_{0}^{t}\int_{\mathbb{R}^d}\sum_{k=1}^{3}\sum_{|\beta_1|+|\beta_2|=l+1}|\mathcal{T}^{1,\beta_1,\beta_2}_{\alpha,k}(t,x,\frac{x-w}{t})|\frac{dwds}{t^d}\right\|_{L^p_x}\\
	&\quad\lesssim_{C}\frac{1}{t^{l+1}}\sum_{|\beta_1|=1}^{l+1}\sum_{m=1}^{|\beta_1|}\int_{0}^{t}\langle s\rangle^{m-d+1}\left\|\int_{\mathbb{R}^d}|\nabla^m \delta_{\frac{s\alpha}{t}}E_i(s,.)(X_{s,t}(x,v))|\frac{1}{\langle v\rangle^N}dv\right\|_{L^p_x}ds\\
	&\quad\lesssim_{C}\frac{1}{t^{l+1+\frac{d(p-1)}{p}}}\sum_{|\beta_1|=1}^{l+1}\sum_{m=1}^{|\beta_1|}\int_{0}^{t/2}\langle s\rangle^{m-d+1}\|\nabla^m \delta_{\frac{s\alpha}{t}}E_i(s)\|_{L^1}ds+\frac{1}{t^{l+1}}\sum_{|\beta_1|=1}^{l+1}\sum_{m=1}^{|\beta_1|}\int_{t/2}^{t}\langle s\rangle^{m-d+1}\|\nabla^m \delta_{\frac{s\alpha}{t}}E_i(s)\|_{L^p}ds
	\\&\quad\lesssim_{C}\frac{|\alpha|^b}{t^{\frac{d(p-1)}{p}+l+1+b}}\sum_{|\beta_1|=1}^{l+1}\sum_{m=1}^{|\beta_1|}\left(\int_{0}^{t/2}+\int_{t/2}^{t}\right)\frac{|s|^{b}\langle s\rangle^{m-d+1}}{\langle s\rangle^{m+1-\delta_0}}ds
	\lesssim_{C}\frac{|\alpha|^b}{t^{\frac{d(p-1)}{p}+l+1+b}}.
\end{split}
\end{align}
Thanks to \eqref{T alpha 1 higher 4} and \eqref{T alpha 1 higher 123}, we have
\begin{align}\label{T1p}
t^{\frac{d(p-1)}{p}+l+1+b}\sup_\alpha\frac{	\left\|\nabla^{l+1}_x\mathcal{T}^1_\alpha[E_i,\partial_i\mu](t)\right\|_{L^p_x}}{|\alpha|^b}\lesssim_{a,b,\delta_0,\mathbf{c},\mathbf{c}' ,C_{A,l}}1.
\end{align}
\textbf{1.2)} Estimate of terms in $\nabla_x^{l+1}\mathcal{T}^2_\alpha[E_i,\partial_i\mu]$. It is clear that
\begin{align*}
&\partial_x^{\beta_1}F_{2,1}(x)\partial_x^{\beta_2}\eta_{2,1}(x)-	\partial_x^{\beta_1}F_{2,2}(x)\partial_x^{\beta_2}\eta_{2,2}(x)\\
&\quad=\left(\partial_x^{\beta_1}F_{2,1}(x)-\left(\frac{s}{t}\right)^{|\beta_1|}\partial_x^{\beta_1}E_i(s,.)(Z_1)\right)
\left(	\partial_x^{\beta_2}\eta_{2,1}(x)-\left(\frac{1}{t}\right)^{|\beta_2|}\partial_x^{\beta_2}(\delta_{\frac{\alpha}{t}}\partial_i\mu)(Z_2)\right)\\
&\qquad+\left(\partial_x^{\beta_1}F_{2,1}(x)-\left(\frac{s}{t}\right)^{|\beta_1|}\partial_x^{\beta_1}E_i(s,.)(Z_1)\right)\left(\frac{1}{t}\right)^{|\beta_2|}\partial_x^{\beta_2}(\delta_{\frac{\alpha}{t}}\partial_i\mu)(Z_2)\\
&\qquad+\left(\frac{s}{t}\right)^{|\beta_1|}\partial_x^{\beta_1}E_i(s,.)(Z_1)\left(	\partial_x^{\beta_2}\eta_{2,1}(x)-\left(\frac{1}{t}\right)^{|\beta_2|}\partial_x^{\beta_2}(\delta_{\frac{\alpha}{t}}\partial_i\mu)(Z_2)\right)\\
&\qquad+\left(\left(\frac{s}{t}\right)^{|\beta_1|}\partial_x^{\beta_1}E_i(s,.)(Z_1)\left(\frac{1}{t}\right)^{|\beta_2|}\partial_x^{\beta_2}(\delta_{\frac{\alpha}{t}}\partial_i\mu)(Z_2)-\partial_x^{\beta_1}F_{2,2}(x)\partial_x^{\beta_2}\eta_{2,2}(x)\right)\\
&\quad	:=\sum_{k=1}^{4}\mathcal{T}^{2,\beta_1,\beta_2}_{\alpha,k}(t,x,\frac{x-w}{t}).
\end{align*}
Then, we also compute the last term first as follows
\begin{align}\nonumber
	&\sum_{|\beta_1|+|\beta_2|=l+1}\left\|\int_{0}^{t}\int_{\mathbb{R}^d}\mathcal{T}^{2,\beta_1,\beta_2}_{\alpha,4}(t,x,\frac{x-w}{t})\frac{dwds}{t^d}\right\|_{L^p_x}=\frac{1}{t^{l+1}}\sum_{|\beta_1|+|\beta_2|=l+1}\left\|\mathcal{T}^2_\alpha\big[s^{|\beta_1|}\partial^{\beta_1}E_i,\partial^{\beta_2}\partial_i\mu\big]\right\|_{L^p_x}\\
&\nonumber\overset{\eqref{T alpha lower 1}}{\lesssim_{\delta_0,\mathbf{c}}}\frac{|\alpha|^b}{t^{d+l+1+b}}\sum_{|\beta_1|=0}^{l+1}\left(\int_{0}^{t/2}|s|^{|\beta_1|}\|\nabla^{|\beta_1|}E_i(s)\|_{L^1}\frac{ds}{\langle s\rangle^{d-1-\delta_0}}+\int_{t/2}^{t}|s|^{|\beta_1|+d}\|\nabla^{|\beta_1|}E_i(s)\|_{L^\infty}\frac{ds}{\langle s\rangle^{d-1-\delta_0}}\right)\\
&\overset{\eqref{E estimate by rho'}}
\lesssim_{b,\delta_0,\mathbf{c},C_{A,l}}\frac{|\alpha|^b}{t^{d+l+1+b}}\int_{0}^{t}\frac{|s|^{l+1}}{\langle s\rangle^{(l+1-\delta_0)+(d-1-\delta_0)}}ds\lesssim_{b,\delta_0,\mathbf{c},C_{A,l}}\frac{|\alpha|^b}{t^{d+l+1+b}}.\label{T alpha 2 higher 4}
\end{align}
Now, we turn to $\mathcal{T}^{2,\beta_1,\beta_2}_{\alpha,k}$, where $k=1,2,3.$ 
Applying Proposition \ref{proposition of higher derivative estimates Y W}, we can get 
\begin{align*}	
&\left|\partial_x^{\beta_1}F_{2,1}(x)-\left(\frac{s}{t}\right)^{|\beta_1|}\partial_x^{\beta_1}E_i(s,.)(Z_1)\right|\lesssim \frac{\mathbf{1}_{|\beta_1|\geq1}}{t^{|\beta_1|}} |\nabla^{|\beta_1|} E_i(s,.)(Z_1)|  |(\nabla_vY_{s,t})|\left(s+|\nabla_vY_{s,t}|\right)^{|\beta_1|-1}\\&\qquad+\frac{\mathbf{1}_{|\beta_1|\geq1}}{t^{|\beta_1|}}\sum_{m=1}^{|\beta_1|-1}\sum_{j=0}^{m-1}\sum_{i=2}^{|\beta_1|-j}|\nabla^m E_i(s,.)(Z_1)|(|s|+|\nabla_vY_{s,t}|)^j |\nabla_v^iY_{s,t}|^{\frac{|\beta_1|-j}{i}}
\\&\overset{\eqref{results of Y small},  \eqref{proposition of higher derivative estimates Y W}}{\lesssim_{C}}\mathbf{1}_{|\beta_1|\geq1}\frac{\|(\rho,U)\|_{S^{{\varepsilon_0}}_a}}{t^{|\beta_1|}}\left(|\nabla^{|\beta_1|} E_i(s,.)(Z_1)|\langle s\rangle^{|\beta_1|-d+\delta_0} +\sum_{m=1}^{|\beta_1|-1}|\nabla^m E_i(s,.)(Z_1)|\langle s\rangle^{m-d+1}\right)\\&\lesssim_{C}\mathbf{1}_{|\beta_1|\geq1}\frac{\|(\rho,U)\|_{S^{{\varepsilon_0}}_a}}{t^{|\beta_1|}}\sum_{m=1}^{|\beta_1|}|\nabla^m E_i(s,.)(Z_1)|\langle s\rangle^{m-d+1},
\end{align*}
and 
\begin{align*}
&\left|	\partial_x^{\beta_2}\eta_{2,1}(x)-\left(\frac{1}{t}\right)^{|\beta_2|}\partial_x^{\beta_2}(\delta_{\frac{\alpha}{t}}\partial_i\mu)(Z_2)\right|\lesssim \frac{\mathbf{1}_{|\beta_2|\geq1}}{t^{|\beta_2|}} |\nabla^{|\beta_2|} (\delta_{\frac{\alpha}{t}}\partial_i\mu)(Z_2)|  |\nabla_vW_{s,t}|(1+|\nabla_vW_{s,t}|)^{|\beta_2|-1}\\&\qquad+\frac{\mathbf{1}_{|\beta_2|\geq1}}{t^{|\beta_2|}}\sum_{m=1}^{|\beta_2|}\sum_{j=0}^{m-1}\sum_{i=2}^{|\beta_2|-j}|\nabla^m (\delta_{\frac{\alpha}{t}}\partial_i\mu)(Z_2)|(1+|\nabla_vW_{s,t}|)^j |\nabla_v^iW_{s,t}|^{\frac{|\beta_2|-j}{i}}\\&\overset{\eqref{proposition of higher derivative estimates Y W}}{\lesssim_{C}}\mathbf{1}_{|\beta_2|\geq1}\frac{\|(\rho,U)\|_{S^{{\varepsilon_0}}_a}}{t^{|\beta_2|}} \left(|\nabla^{|\beta_2|} (\delta_{\frac{\alpha}{t}}\partial_i\mu)(Z_2)|\langle s\rangle^{-d+\delta_0}
+ \sum_{m=1}^{|\beta_2|-1}|\nabla^m (\delta_{\frac{\alpha}{t}}\partial_i\mu)(Z_2)|\langle s\rangle^{-d+1}\right)\\
&\lesssim_{C}\mathbf{1}_{|\beta_2|\geq1}\frac{|\alpha|^b\|(\rho,U)\|_{S^{{\varepsilon_0}}_a}}{t^{|\beta_2|+b}}\left(\frac{1}{\langle\frac{x-w}{t}\rangle^N}+\frac{1}{\langle\frac{x-w-\alpha}{t}\rangle^N}\right)\langle s\rangle^{-d+1}.
\end{align*}
Define 
\begin{equation*}
	\Xi(x,w,\alpha,t)=\frac{1}{\langle\frac{x-w}{t}\rangle^N}+\frac{1}{\langle\frac{x-w-\alpha}{t}\rangle^N}.
\end{equation*}
Thus, we have
\begin{align*}
\sum_{|\beta_1|+|\beta_2|=l+1}|\mathcal{T}^{2,\beta_1,\beta_2}_{\alpha,1}(t,x,\frac{x-w}{t})|\lesssim_{C}&\frac{|\alpha|^b}{t^{l+1+b}}\sum_{|\beta_1|=1}^{l}\sum_{m=1}^{|\beta_1|}|\nabla^mE_i(s,.)(Z_1)|\langle s\rangle^{m-2d+2}\Xi(x,w,\alpha,t),\\
\sum_{|\beta_1|+|\beta_2|=l+1}|\mathcal{T}^{2,\beta_1,\beta_2}_{\alpha,2}(t,x,\frac{x-w}{t})|\lesssim_{C}&\frac{|\alpha|^b}{t^{l+1+b}}\sum_{|\beta_1|=1}^{l+1}\sum_{m=1}^{|\beta_1|}|\nabla^m E_i(s,.)(Z_1)|\langle s\rangle^{m-d+1}	\Xi(x,w,\alpha,t),\\
\sum_{|\beta_1|+|\beta_2|=l+1}|\mathcal{T}^{2,\beta_1,\beta_2}_{\alpha,3}(t,x,\frac{x-w}{t})|\lesssim_{C}&\frac{|\alpha|^b}{t^{l+1+b}}\sum_{|\beta_1|=0}^{l}
\left|\partial_x^{\beta_1}	E_i(s,.)(Z_1)\right|
\langle s\rangle^{|\beta_1|-d+1}\Xi(x,w,\alpha,t).
\end{align*}
We observe to get
\begin{align}
\begin{split}
	\sum_{k=1}^{3}\sum_{|\beta_1|+|\beta_2|=l+1}|\mathcal{T}^{2,\beta_1,\beta_2}_{\alpha,k}(t,x,\frac{x-w}{t})|\lesssim_{C}\frac{|\alpha|^b}{t^{l+1+b}}\sum_{|\beta_1|=1}^{l+1}\sum_{m=1}^{|\beta_1|}|\nabla^m E_i(s,.)(Z_1)|\langle s\rangle^{m-d+1}	\Xi(x,w,\alpha,t).
\end{split}
\end{align} 
As the proof of \eqref{T alpha 1 higher 123},  we have Remark \ref{mathcal Halpha} in the Appendix with $\mathcal{H}=\nabla^mE_i$ and $\varphi=Y_{s,t}(x,v)$, we make the change of variable $v=\frac{x-w}{t}$, then we apply \eqref{E estimate by rho'}  to obtain
\begin{align}\label{T alpha 2 higher 123}
\begin{split}
	&\left\|\int_{0}^{t}\int_{\mathbb{R}^d}\sum_{k=1}^{3}\sum_{|\beta_1|+|\beta_2|=l+1}|\mathcal{T}^{1,\beta_1,\beta_2}_{\alpha,k}(t,x,\frac{x-w}{t})|\frac{dwds}{t^d}\right\|_{L^p_x}\\
	&\quad\lesssim_{C}\frac{|\alpha|^b}{t^{l+1+b}}\sum_{|\beta_1|=1}^{l+1}\sum_{m=1}^{|\beta_1|}\int_{0}^{t}\langle s\rangle^{m-d+1}\left\|\int_{\mathbb{R}^d}|\nabla^m E_i(s,.)(X_{s,t}(x,v))|\left(\frac{1}{\langle v\rangle^N}+\frac{1}{\langle v-\frac{\alpha}{t}\rangle^N}\right)dv\right\|_{L^p_x}ds\\
	&\quad\lesssim_{C}\frac{|\alpha|^b}{t^{l+1+b+\frac{d(p-1)}{p}}}\sum_{|\beta_1|=1}^{l+1}\sum_{m=1}^{|\beta_1|}\int_{0}^{t/2}\langle s\rangle^{m-d+1}\|\nabla^m E_i(s)\|_{L^1}ds\\
	&\qquad+\frac{|\alpha|^b}{t^{l+1+b}}\sum_{|\beta_1|=1}^{l+1}\sum_{m=1}^{|\beta_1|}\int_{t/2}^{t}\langle s\rangle^{m-d+1}\|\nabla^m E_i(s)\|_{L^p}ds
	\\&\quad\lesssim_{C}\frac{|\alpha|^b}{t^{\frac{d(p-1)}{p}+l+1+b}}\sum_{|\beta_1|=1}^{l+1}\sum_{m=1}^{|\beta_1|}\left(\int_{0}^{t/2}+\int_{t/2}^{t}\right)\frac{|s|^{b}\langle s\rangle^{m-d+1}}{\langle s\rangle^{m+1-\delta_0}}ds
	\lesssim_{C}\frac{|\alpha|^b}{t^{\frac{d(p-1)}{p}+l+1+b}}.
\end{split}
\end{align}
Thanks to \eqref{T alpha 2 higher 4} and \eqref{T alpha 2 higher 123}, we have
\begin{align}\label{T2p}
t^{\frac{d(p-1)}{p}+l+1+b}\sup_\alpha\frac{	\left\|\nabla^{l+1}_x\mathcal{T}^2_\alpha[E_i,\partial_i\mu](t)\right\|_{L^p_x}}{|\alpha|^b}\lesssim_{C}1.
\end{align}
\textbf{1.3)} Estimate of terms in $\mathcal{T}^3_\alpha[E_i,\partial_i\mu]$.
\\		
Firstly, we have
\begin{align*}
	|\partial^{\beta_1}_xF_{3}(x)|&\lesssim \sum_{\substack{1\leq |\beta_3|\leq |\beta_1|,|\vartheta_i|>0\\|\vartheta_1|+...+|\vartheta_{|\beta_3|}|=|\beta_1|}}\Bigg| \prod_{j=1}^{|\beta_3|}(\partial_x^{\vartheta_j}Z_4(x))(\partial^{\beta_3}E_i)(s,Z_4(x))-\prod_{j=1}^{|\beta_3|}(\partial_x^{\vartheta_j}Z_4(x-\alpha))(\partial^{\beta_3}E_i)(s,Z_1(x-\alpha))\Bigg|
\\&\leq \sum_{\substack{1\leq |\beta_3|\leq |\beta_1|,|\vartheta_i|>0\\|\vartheta_1|+...+|\vartheta_{|\beta_3|}|=|\beta_1|}}\left| \prod_{j=1}^{|\beta_3|}(\partial_x^{\vartheta_j}Z_4(x))-\prod_{j=1}^{|\beta_3|}(\partial_x^{\vartheta_j}Z_4(x-\alpha))\right||(\nabla^{|\beta_3|}E_i)(s,Z_4(x))|\\&\qquad+
\prod_{j=1}^{|\beta_3|}|\partial_x^{\vartheta_j}Z_4(x-\alpha)|\left|(\nabla^{|\beta_3|}E_i)(s,Z_4(x))-(\nabla^{|\beta_3|}E_i)(s,Z_1(x-\alpha))\right|.
\end{align*}
Then, by Proposition \ref{estimates about Y and W} and \ref{proposition of higher derivative estimates Y W},
for $|\beta_1|\geq1$,
\begin{align*}
	|\partial^{\beta_1}_xF_{3}(x)|&\lesssim_{b,\delta_0,\mathbf{c}} \frac{|\alpha|^{b}}{t^{|\beta_1|+b}}\sum_{1\leq |\beta_3|\leq \beta_1|}^{|\beta_1|}\left(\left[\langle s\rangle ^{|\beta_3|-d+1}\mathbf{1}_{(|\beta_3|,|\beta_1|)\not= (1,l+1)}+\langle s\rangle ^{-d+2+\delta_0+b}\mathbf{1}_{(|\beta_3|,|\beta_1|)= (1,l+1)}\right]|(\nabla^{|\beta_3|}E_i)(s,Z_4)\right.\\&\quad\left.+
\langle s\rangle ^{|\beta_3|-(d-1-\delta_0)b-\mathbf{1}_{|\beta_3|<|\beta_1|}}\sup_z\frac{\left|(\delta_{z}\nabla^{|\beta_3|}E_i)(s,Z_4)\right|}{|z|^{b}} \right).
\end{align*}
And if $|\beta_1|=0$, we can obtain
\begin{align*}
|F_{3}(x)|\leq\left|\frac{\alpha}{t}\right|^{b}\frac{|\delta_z  E_i(s,\cdot)|}{|z|^{b}}\|\nabla_vY_{s,t}\|_{L^\infty}^{b}\lesssim_{b,\delta_0,\mathbf{c}}\left|\frac{\alpha}{t}\right|^{b}\sup_z\frac{|\delta_z  E_i(s,\cdot)|}{|z|^{b}}\langle s\rangle^{-b(d-1-\delta_0)}.
\end{align*}
Now we estimate the $L^1$ and $L^\infty$ norm of 	$\partial^{\beta_1}_xF_{3}(x)$	as follows,	\begin{align}\nonumber
	\|\partial_x^{\beta_1}F_3(x)\|_{L^p_x}
	\lesssim&_{b,\delta_0,\mathbf{c}}\frac{|\alpha|^{b}}{t^{|\beta_1|+b}}\sum_{1\leq |\beta_3|\leq \beta_1|} 
	\left[\langle s\rangle ^{|\beta_3|-d+1}\mathbf{1}_{(|\beta_3|,|\beta_1|)\not= (1,l+1)}+\langle s\rangle ^{-d+2+\delta_0+b}\mathbf{1}_{(|\beta_3|,|\beta_1|)= (1,l+1)}\right]\|(\nabla^{|\beta_3|}E_i)(s)\|_{L^p}\\\nonumber
	&+\frac{|\alpha|^{b}}{t^{|\beta_1|+b}}\left(\sum_{1\leq |\beta_3|\leq \beta_1|}
	\langle s\rangle ^{|\beta_3|-(d-1-\delta_0)b-\mathbf{1}_{|\beta_3|<|\beta_1|}}      \|\nabla^{|\beta_1|}
	E_i(s)\|_{\dot{F}^b_{p,\infty}}+\| E_i(s)\|_{\dot{F}^b_{p,\infty}}\langle s\rangle^{-b(d-1-\delta_0)}\right)
	\\\nonumber
	\lesssim&_{C}\frac{|\alpha|^{b}}{t^{|\beta_1|+b}}
	\left(\frac{1}{\langle s\rangle^{d-\delta_0+\frac{d(p-1)}{p}}}
	+\frac{1}{\langle s\rangle^{d-b-2\delta_0+\frac{d(p-1)}{p}}}
	+\frac{1}{\langle s\rangle^{-\delta_0+(d-1-\delta_0)b+\frac{d(p-1)}{p}}}
	\right)\\
	\lesssim&_{C}\frac{|\alpha|^{b}}{t^{|\beta_1|+b}}\langle s\rangle^{\delta_0-(d-1-\delta_0)b-\frac{d(p-1)}{p}}\label{LpF3}
	.
\end{align}
On the other hand, for $|\beta_2|\geq 0$, we have
\begin{align*}
|\partial_x^{\beta_2}\eta_{1,1}(x)|
\lesssim	&\frac{1}{t^{|\beta_2|}}\sum_{\substack{1\leq |\beta_3|\leq |\beta_2|,|\vartheta_i|>0\\|\vartheta_1|+...+|\vartheta_{|\beta_3|}|=|\beta_2|}} \prod_{j=1}^{|\beta_3|}|\partial_x^{\vartheta_j}Z_2(x)\|\partial^{\beta_3}\partial_i\mu|(Z_2(x))
\lesssim_{C}\frac{1}{t^{|\beta_2|}}
\frac{1}{\langle\frac{x-w}{t}\rangle^N}.
\end{align*}
Since $(d-1-\delta_0)b-\delta_0>1$, $\delta_0\leq\frac{4-\sqrt{13}}{3}$, we apply \eqref{LpF3} to get
\begin{align*}
\begin{split}
	\|\mathcal{T}_{\alpha}^3[E_i,\partial_i\mu](t)\|_{L^p}
	\lesssim_{b,\delta_0,\mathbf{c},\mathbf{c}'}&\int_{0}^{t/2}\frac{1}{t^{|\beta_2|}}
	\|\partial_x^{\beta_1}F_3(x)\|_{L^1_x}\frac{ds}{(t-s)^\frac{d(p-1)}{p}}
	+\int_{t/2}^{t}\int_{\mathbb{R}^d}\frac{1}{t^{|\beta_2|}}
	\|\partial_x^{\beta_1}F_3(x)\|_{L^p_x}\frac{1}{\langle v\rangle^N}dvds\\
	\lesssim_{C}&\frac{|\alpha|^{b}}{t^{\frac{d(p-1)}{p}+l+1+b}}\left(\int_{0}^{t/2}+\int_{t/2}^{t}\right)\frac{1}{\langle s\rangle^{-\delta_0+(d-1-\delta_0)b}}ds
	\lesssim_{C}\frac{|\alpha|^{b}}{t^{\frac{d(p-1)}{p}+l+1+b}}.
\end{split}
\end{align*}
Thus, we have
\begin{align}\label{T3p}
t^{\frac{d(p-1)}{p}+l+1+b}\sup_\alpha\frac{	\left\|\mathcal{T}^3_\alpha[E_i,\partial_i\mu](t)\right\|_{L^p_x}}{|\alpha|^b}\lesssim_{C}1.
\end{align}
\textbf{1.4)} Estimate of terms in $\mathcal{T}^4_\alpha[E_i,\partial_i\mu]$.\\
Firstly, we have for $|\beta_1|\geq1$,
\begin{align*}
|\partial^{\beta_1}_x F_{2,1}(x-\alpha)|
\lesssim&\sum_{\substack{1\leq |\beta_3|\leq |\beta_1|-1,|\vartheta_i|>0\\|\vartheta_1|+...+|\vartheta_{|\beta_3|}|=|\beta_1|}}| \prod_{j=1}^{|\beta_3|}(\partial_x^{\vartheta_j}Z_1(x-\alpha))(\partial^{\beta_3}E_i)(s,Z_1(x-\alpha))|\\&
+\sum_{|\vartheta_1|+...+|\vartheta_{|\beta_1|}|=|\beta_1|,|\vartheta_i|>0}| \prod_{j=1}^{|\beta_1|}(\partial_x^{\vartheta_j}Z_1(x-\alpha))(\partial^{\beta_1}E_i)(s,Z_1(x-\alpha))|\\
\lesssim&_{b,\delta_0,\mathbf{c}} \frac{\|(\rho,U)\|_{S^{{\varepsilon_0}}_a}}{t^{|\beta_1|}}	\sum_{1\leq |\beta_3|\leq |\beta_1|}
\langle s\rangle^{|\beta_3|-\mathbf{1}_{|\beta_3|<|\beta_1|}}
\left|(\partial^{\beta_3}E_i)(s,Z_1(x-\alpha))\right|,
\end{align*}
and for $|\beta_1|=0$, we obtain
\begin{align*}
| F_{2,1}(x-\alpha)|=	\left|E_i(s,Z_1(x-\alpha))\right|.
\end{align*}
Then, we estimate the $L^1$ and $L^\infty$ norm of $\partial^{\beta_1}_x F_{2,1}$
\begin{align}
\nonumber\|\partial^{\beta_1}_x F_{2,1}\|_{L^p}\lesssim_{b,\delta_0,\mathbf{c}}&\frac{\mathbf{1}_{|\beta_1|\geq1}}{t^{|\beta_1|}}	\sum_{1\leq |\beta_3|\leq |\beta_1|}
	\langle s\rangle^{|\beta_3|-\mathbf{1}_{|\beta_3|<|\beta_1|}}
	\left\|(\partial^{\beta_3}E_i)(s,Z_1(x-\alpha))\right\|_{L^p}+\mathbf{1}_{|\beta_1|=0}\left\|E_i(s,Z_1(x-\alpha))\right\|_{L^p}\\
	\lesssim_{C}&\frac{1}{t^{|\beta_1|}}	\langle s\rangle^{\delta_0-1-\frac{d(p-1)}{p}}\lesssim_{C}\frac{1}{t^{|\beta_1|}}	\langle s\rangle^{\delta_0-1-\frac{d(p-1)}{p}}.\label{LpF12}
\end{align}
Meanwhile, for $|\beta_2|\geq1$, we have
\begin{align*}	|\partial^{\beta_2}_x\eta_{3}(x)|&\lesssim \sum_{\substack{1\leq |\beta_3|\leq |\beta_2|,|\vartheta_i|>0\\|\vartheta_1|+...+|\vartheta_{|\beta_3|}|=|\beta_2|}}| \prod_{j=1}^{|\beta_3|}(\partial_x^{\vartheta_j}Z_2(x))(\partial^{\beta_3}\partial_i\mu)(s,Z_2(x))-\prod_{j=1}^{|\beta_3|}(\partial_x^{\vartheta_j}Z_2(x-\alpha))(\partial^{\beta_3}\partial_i\mu)(s,Z_3(x))|
\\&\leq \sum_{\substack{1\leq |\beta_3|\leq |\beta_2|,|\vartheta_i|>0\\|\vartheta_1|+...+|\vartheta_{|\beta_3|}|=|\beta_2|}}\left(\left| \prod_{j=1}^{|\beta_3|}(\partial_x^{\vartheta_j}Z_2(x))-\prod_{j=1}^{|\beta_3|}(\partial_x^{\vartheta_j}Z_2(x-\alpha))\right||(\nabla^{|\beta_3|}\partial_i\mu)(s,Z_2(x))|\right.\\&\quad\left.+
\prod_{j=1}^{|\beta_3|}|\partial_x^{\vartheta_j}Z_2(x-\alpha)|\left|(\nabla^{|\beta_3|}\partial_i\mu)(s,Z_2(x))-(\nabla^{|\beta_3|}\partial_i\mu)(s,Z_3(x))\right|\right)
\\&\lesssim_{C}\frac{|\alpha|^b\|(\rho,U)\|_{S^{{\varepsilon_0}}_a}}{t^{|\beta_2|+b}}\frac{\langle s\rangle^{-d+2}}{\langle \frac{x-w}{t}\rangle^N}+\frac{\min\left\{\frac{|\alpha|}{t},1\right\}}{t^{|\beta_2|}}\frac{\langle s\rangle^{-d+2}}{\langle \frac{x-w}{t}\rangle^N}
\lesssim_{C}\frac{|\alpha|^b}{t^{|\beta_2|+b}}\frac{\langle s\rangle^{-d+2}}{\langle \frac{x-w}{t}\rangle^N}.
\end{align*}
For $|\beta_2|=0$, we have 
\begin{align*}
|\eta_{3}(x)|=&\left|(\partial_i\mu)(Z_2(x))-(\partial_i\mu)(Z_3(x))\right|=\frac{|\alpha|^{b}}{t^{b}}\frac{|\partial_i\mu(Z_2(x))-\partial_i\mu(Z_3(x))|}{|Z_2(x)-Z_3(x)|^{b}}\left|\frac{Z_2(x)-Z_3(x)}{|\alpha|/t}\right|^{b}\\
\lesssim&_{C}\frac{|\alpha|^{b}\|(\rho,U)\|_{S^{{\varepsilon_0}}_a}^b}{t^{b}}\langle s\rangle^{(\delta_0-d)b}\frac{1}{\langle\frac{x-w}{t}\rangle^N}\lesssim_{C}\frac{|\alpha|^{b}}{t^{b}}\frac{\langle s\rangle^{(\delta_0-d)b}}{\langle\frac{x-w}{t}\rangle^N}.
\end{align*}
Thus, with \eqref{LpF12}, we have
\begin{align*}
	\left\|\mathcal{T}_{\alpha}^4[E_i,\partial_i\mu](t,x)\right\|_{L^p}\lesssim_{C}&\sum_{|\beta_1|+|\beta_2|=l+1}\sum_{|\beta_1|=0}^{l}\int_{0}^{t/2}\frac{|\alpha|^b}{t^{|\beta_2|+b}}\langle s\rangle^{-d+2}
	\|\partial_x^{\beta_1}F_{2,1}(x)\|_{L^1_x}\frac{ds}{(t-s)^\frac{d(p-1)}{p}}\\
	&+\int_{0}^{t/2}\frac{|\alpha|^{b}}{t^{b}}\langle s\rangle^{(\delta_0-d)b}
	\|\partial_x^{l+1}F_{2,1}(x)\|_{L^1_x}\frac{ds}{(t-s)^\frac{d(p-1)}{p}}\\
	&+\sum_{|\beta_1|+|\beta_2|=l+1}\sum_{|\beta_1|=0}^{l}\int_{0}^{t/2}\int_{\mathbb{R}^d}\frac{|\alpha|^b}{t^{|\beta_2|+b}}\frac{\langle s\rangle^{-d+2}}{\langle v\rangle^N}
	\|\nabla_x^{\beta_1}F_{2,1}(x)\|_{L^p_x}dvds\\
	&+\int_{0}^{t/2}\int_{\mathbb{R}^d}\frac{|\alpha|^{b}}{t^{b}}\frac{\langle s\rangle^{(\delta_0-d)b}}{\langle v\rangle^N}
	\|\partial_x^{l+1}F_{2,1}(x)\|_{L^p_x}dvds
	\\
	\lesssim_{C}&\frac{|\alpha|^{b}}{t^{\frac{d(p-1)}{p}+l+1+b}}\left(\int_{0}^{t/2}+\int_{t/2}^{t}\right)\left(\frac{1}{\langle s\rangle^{d-1-\delta_0}}+\frac{1}{\langle s\rangle^{(d-\delta_0)b+1-\delta_0}}\right)ds\\
	\lesssim_{C}&\frac{|\alpha|^{b}}{t^{\frac{d(p-1)}{p}+l+1+b}}.
\end{align*}
This implies 
\begin{align}\label{T4p}
t^{\frac{d(p-1)}{p}+l+1+b}\sup_\alpha\frac{	\left\|\mathcal{T}^4_\alpha[E_i,\partial_i\mu](t)\right\|_{L^p_x}}{|\alpha|^b}\lesssim_{C}1.
\end{align}
Hence, combining \eqref{T1p}, \eqref{T2p}, \eqref{T3p} with \eqref{T4p} we obtain
\begin{align}\label{T t large}
t^{\frac{d(p-1)}{p}l+1+b}\|\nabla^{l+1}_x\mathcal{T}[E_i,\partial_i\mu](t)\|_{\dot{B}^b_{p,\infty}}\lesssim_{C}1.
\end{align}
\textbf{(2)}We now consider $t\leq 1$.
Denote
\begin{align*}
&\overline{F}_{1,1}(x)=	\delta_{\alpha}E_i(s,.)\left(X_{s,t}(x,v)\right),~~\overline{\eta}_{1,1}(x)=\partial_i\mu\left(V_{s,t}(x,v)\right),\\& 
\overline{F}_{1,2}(x)=\delta_{\alpha}E_i(s,.)(x-(t-s)v),~~
\overline{\eta}_{1,2}(x)=\partial_i\mu(v),\\&
\overline{F}_{2,1}(x)=	E_i\left(s,X_{s,t}(x,v)\right),~~\overline{F}_{3}(x)=E_i\left(s,X_{s,t}(x-\alpha,v)\right)-E_i\left(s,X_{s,t}(x,v)-\alpha\right),
\\&
\overline{\eta}_3(x)=\partial_i\mu\left(V_{s,t}(x-\alpha,v)\right)-\partial_i\mu\left(V_{s,t}(x,v)\right).
\end{align*}
Thus,  we have 
\begin{align}\nonumber
&	\delta_{\alpha}\mathcal{T}[E_i,\partial_i\mu](t,x)= \sum_{j=1}^{3} \overline{\mathcal{T}}_{\alpha}^j[E_i,\partial_i\mu](t,x),
\end{align} 
with
\begin{align*}			
&\overline{\mathcal{T}}_{\alpha}^1[E_i,\partial_i\mu](t,x)=-\int_{0}^{t}\int_{\mathbb{R}^d}
\left(\overline{F}_{1,1}(x)\overline{\eta}_{1,1}(x)-	 \overline{F}_{1,2}(x) \overline{\eta}_{1,2}(x)\right)dvds,\\
&
\overline{\mathcal{T}}_{\alpha}^2[E_i,\partial_i\mu](t,x)=\int_{0}^{t}\int_{\mathbb{R}^d}  \overline{F}_3(x) \overline{\eta}_{1,1}(x)
dvds,\\
&
\overline{\mathcal{T}}_{\alpha}^3[E_i,\partial_i\mu](t,x)=\int_{0}^{t}\int_{\mathbb{R}^d}  \overline{F}_{2,1}(x-\alpha) \overline{\eta}_{3}(x)dvds.
\end{align*}
Hence, we obtain	
\begin{align}\nonumber
&|\nabla^{l+1}_x	\delta_{\alpha}\overline{	\mathcal{T}}^j[E_i,\partial_i\mu](t,x)|\lesssim\sum_{|\beta_1|+|\beta_2|=l+1} \left|\overline{	\mathcal{T}}_{\alpha}^j[E_i,\partial_i\mu,\beta_1,\beta_2](t,x)\right|,
\end{align} 
where 
\begin{align*}
&\overline{\mathcal{T}}_{\alpha}^1[E_i,\partial_i\mu,\beta_1,\beta_2](t,x)=-\int_{0}^{t}\int_{\mathbb{R}^d}
\left(\partial_x^{\beta_1}\overline{F}_{1,1}(x)\partial_x^{\beta_2}\overline{\eta}_{1,1}(x)-	\partial_x^{\beta_1}\overline{F}_{1,2}(x)\partial_x^{\beta_2}\overline{\eta}_{1,2}(x)\right)dvds,\\
&
\overline{\mathcal{T}_{\alpha}}^2[E_i,\partial_i\mu,\beta_1,\beta_2](t,x)=\int_{0}^{t}\int_{\mathbb{R}^d} \partial_x^{\beta_1}\overline{F}_3(x)\partial_x^{\beta_2}\overline{\eta}_{1,1}(x)dvds,\\
&
\overline{\mathcal{T}}_{\alpha}^3[E_i,\partial_i\mu,\beta_1,\beta_2](t,x)=\int_{0}^{t}\int_{\mathbb{R}^d} \partial_x^{\beta_1}\overline{F}_{2,1}(x-\alpha)\partial_x^{\beta_2}\overline{\eta}_{3}(x)dvds.
\end{align*}
\textbf{2.1)} Estimate of  $\overline{\mathcal{T}}^1_\alpha[E_i,\partial_i\mu]$.\\
We can write
\begin{align*}
&\left|\partial_x^{\beta_1}\overline{F}_{1,1}(x)\partial_x^{\beta_2}\overline{\eta}_{1,1}(x)-	\partial_x^{\beta_1}\overline{F}_{1,2}(x)\partial_x^{\beta_2}\overline{\eta}_{1,2}(x)\right|\\				   &\quad\leq\left|\partial_x^{\beta_1}\overline{F}_{1,1}(x)- \partial_x^{\beta_1}	\delta_{\alpha}E_i(s,.)(X_{s,t}(x,v))\right|
\left|	\partial_x^{\beta_2}\overline{\eta}_{1,1}(x)-(\partial_x^{\beta_2}\partial_i\mu)(V_{s,t}(x,v))\right|\\
&\qquad
+\left|\partial_x^{\beta_1}\overline{F}_{1,1}(x)-\partial_x^{\beta_1}	\delta_{\alpha}E_i(s,.)(X_{s,t}(x,v))\right|\left|(\partial_x^{\beta_2}\partial_i\mu)(V_{s,t}(x,v))\right|\\
&\qquad+\left| \partial_x^{\beta_1}	\delta_{\alpha}E_i(s,.)(X_{s,t}(x,v))\right|\left|	\partial_x^{\beta_2}\overline{\eta}_{1,1}(x)- (\partial_x^{\beta_2}\partial_i\mu)(V_{s,t}(x,v))\right|\\
&\qquad 
+\left| \partial_x^{\beta_1}	\delta_{\alpha}E_i(s,.)(X_{s,t}(x,v)) (\partial_x^{\beta_2}\partial_i\mu)(V_{s,t}(x,v))-\overline{F}_{1,2}(x)\overline{\eta}_{1,2}(x)\right|.
\end{align*}
Thanks to Lemma \ref{lem1}	we  have 
\begin{align*}
&\left|\partial_x^{\beta_1}\overline{F}_{1,1}(x)-\partial_x^{\beta_1}	\delta_{\alpha}F(s,.)(X_{s,t}(x,v))\right|\\
&\lesssim \mathbf{1}_{|\beta_1|\geq1}|\nabla^{|\beta_1|} 	\delta_{\alpha}F(s,.)(X_{s,t}(x,v))|  |(\nabla_xY_{s,t})|\left(1+|\nabla_xY_{s,t}|\right)^{|\beta_1|-1}\\&\quad+\mathbf{1}_{|\beta_1|\geq1}\sum_{m=1}^{|\beta_1|-1}\sum_{j=0}^{m-1}\sum_{i=2}^{|\beta_1|-j}|\nabla^m \delta_{\alpha}F(s,.)(X_{s,t}(x,v))|(1+|\nabla_xY_{s,t}|)^j |\nabla_x^iY_{s,t}|^{\frac{|\beta_1|-j}{i}}
\\&\lesssim_{C} \mathbf{1}_{|\beta_1|\geq1} \sum_{m=1}^{|\beta_1|}|\nabla^m \delta_{\alpha}F(s,.)(X_{s,t}(x,v))|\langle s\rangle^{-d+\delta_0}.
\end{align*}
Then, for $|\beta_2|\geq1$,
\begin{align*}
&	\left|	\partial_x^{\beta_2}\overline{\eta}_{1,1}(x)-(\partial_x^{\beta_2}\partial_i\mu)(V_{s,t}(x,v))\right|\\&\qquad\lesssim\mathbf{1}_{|\beta_2|\geq1} \Bigg(|\nabla^{|\beta_2|} \partial_i\mu(V_{s,t}(x,v))|  |\nabla_xW_{s,t}|^{|\beta_2|}
+\sum_{m=1}^{|\beta_2|}\sum_{j=0}^{m-1}\sum_{i=2}^{|\beta_2|-j}|\nabla^m \partial_i\mu(V_{s,t}(x,v))\|\nabla_xW_{s,t}|^j |\nabla_x^iW_{s,t}|^{\frac{|\beta_2|-j}{i}}\Bigg)\\&\quad\overset{\eqref{results of Y small},  \eqref{proposition of higher derivative estimates Y W}}{\lesssim_{C}}\mathbf{1}_{|\beta_2|\geq1} \sum_{m=1}^{|\beta_2|}|\nabla^m \partial_i\mu(V_{s,t}(x,v))|\langle s\rangle^{-d-1+\delta_0}\|(\rho,U)\|_{S^{{\varepsilon_0}}_a}
\lesssim_{C}\frac{\mathbf{1}_{|\beta_2|\geq1}}{\langle v\rangle^N}\langle s\rangle^{-1-d+\delta_0}.
\end{align*}
Then,  we have for $p=1,\infty$,
\begin{align}\nonumber
	\left\|\overline{\mathcal{T}}^1_\alpha[E_i,\partial_i\mu]\right\|_{L^p_x}	\lesssim_{C}&\sum_{|\beta_1|=0}^{l+1}\sum_{m=0}^{|\beta_1|}\int_{0}^{t}\int_{\mathbb{R}^d}
	\|\nabla^m \delta_{\alpha}E_i(s)\|_{L^p}\langle s\rangle^{-d+\delta_0}\frac{dvds}{\langle v\rangle^N}
	+\sum_{|\beta_1|+|\beta_2|=l+1}\|\mathcal{T}[\partial^{\beta_1}_x\delta_\alpha E_i,\partial^{\beta_2}_x\partial_i\mu]\|_{L^p_x}
	\\
	\lesssim_{C}&|\alpha|^{b}\sum_{|\beta_1|+1}^{l+1}\sum_{m=0}^{|\beta_1|}\left(\int_{0}^{t}
	\frac{1}{\langle s\rangle^{m+1-\delta_0}}\langle s\rangle^{-d+\delta_0}ds
	+\|\nabla^{|\beta_1|}E_i(s)\|_{\dot{B}^b_{p,\infty}}\right)
	\lesssim_{C}|\alpha|^{b}.\label{Tbar1}
\end{align}
\textbf{2.2)} Estimate of  $\overline{\mathcal{T}}^2_\alpha[E_i,\partial_i\mu]$.\\
Firstly, we consider 	$\partial^{\beta_1}_x\overline{F}_{3}(x)$ as follows.
If $|\beta_1|\geq1$, we get
\begin{align*}
&|\partial^{\beta_1}_x\overline{F}_{3}(x)|\lesssim \sum_{\substack{1\leq |\beta_3|\leq |\beta_1|,|\vartheta_i|>0\\|\vartheta_1|+...+|\vartheta_{|\beta_3|}|=|\beta_1|}}\Bigg| \prod_{j=1}^{|\beta_3|}(\partial_x^{\vartheta_j}X_{s,t}(x,v))(\partial^{\beta_3}E_i)(s,X_{s,t}(x,v)-\alpha)\\
&\qquad\qquad\qquad\qquad\qquad\qquad\qquad-\prod_{j=1}^{|\beta_3|}(\partial_x^{\vartheta_j}X_{s,t}(x-\alpha,v))(\partial^{\beta_3}E_i)(s,X_{s,t}(x-\alpha,v))\Bigg|
\\
&\leq\sum_{\substack{1\leq |\beta_3|\leq |\beta_1|,|\vartheta_i|>0\\|\vartheta_1|+...+|\vartheta_{|\beta_3|}|=|\beta_1|}}\Bigg(\left| \prod_{j=1}^{|\beta_3|}(\partial_x^{\vartheta_j}X_{s,t}(x,v))-\prod_{j=1}^{|\beta_3|}(\partial_x^{\vartheta_j}X_{s,t}(x-\alpha,v))\right||(\nabla^{|\beta_3|}E_i)(s,X_{s,t}(x,v)-\alpha)|\\
&\quad+\prod_{j=1}^{|\beta_3|}|\partial_x^{\vartheta_j}X_{s,t}(x-\alpha,v)|\left|(\nabla^{|\beta_3|}E_i)(s,X_{s,t}(x,v)-\alpha)-(\nabla^{|\beta_3|}E_i)(s,X_{s,t}(x-\alpha,v))\right|\Bigg)\\
&\lesssim_{C}|\alpha|^{b} \sum_{1\leq |\beta_3|\leq \beta_1|}\left(\langle s\rangle^{-d+\delta_0}|\nabla^{|\beta_3|}E_i|(s,X_{s,t}(x,v)-\alpha)
+\langle s\rangle ^{-(d-\delta_0)b}\sup_{z} \frac{\left|\delta_{z}\nabla^{|\beta_3|}E_i(s,\cdot)(X_{s,t}(x-\alpha,v))\right|}{|z|^{b}} \right).
\end{align*}
If $|\beta_1|=0$, we can yield
\begin{align*}
|\overline{F}_{3}(x)|\leq\frac{|\overline{F}_3(x)|}{|\delta_\alpha Y_{s,t}(x,v)|^{b}}\left\|\nabla_xY_{s,t}\right\|_{L^\infty}^{b}\lesssim_{b,\delta_0,\mathbf{c}}|\alpha|^b\sup_{z}\frac{|\delta_zE_i(s,\cdot)(X_{s,t}(x-\alpha,v)|}{|z|^{b}}\langle s\rangle^{-(d-\delta_0)b}.
\end{align*}
Then, we consider another term. For $|\beta_2|\geq0$,
\begin{align*}
|\partial_x^{\beta_2}\overline{\eta}_{1,1}(x)|
&\lesssim\sum_{\substack{1\leq |\beta_3|\leq |\beta_2|,|\vartheta_i|>0\\|\vartheta_1|+...+|\vartheta_{|\beta_3|}|=|\beta_2|,}} \prod_{j=1}^{|\beta_3|}|\partial_x^{\vartheta_j}W_{s,t}(x-vt,v)\|\partial^{\beta_3}\partial_i\mu|(V_{s,t}(x,v))|
\lesssim_{b,\delta_0,\mathbf{c},\mathbf{c}'}
\frac{1}{\langle v\rangle^N}.
\end{align*}
Hence, we obtain
\begin{align}\nonumber
	\|\overline{\mathcal{T}}_{\alpha}^2[E_i,\partial_i\mu](t)\|_{L^p}
	&\lesssim_{C}|\alpha|^{b} \int_{0}^{t}\int_{\mathbb{R}^d}
	\sum_{\substack{1\leq|\beta_1|\leq l+1\\1\leq |\beta_3|\leq |\beta_1|}}\left(\langle s\rangle^{-d+\delta_0}\|(\nabla^{|\beta_3|}E_i)(s)\|_{L^p}
	+\langle s\rangle ^{-(d-\delta_0)b} \|\nabla^{|\beta_3|}E_i\|_{\dot{F}^b_{p,\infty}}\right)
	\frac{dvds}{\langle v\rangle^N}\\\nonumber
	&\quad+|\alpha|^{b} \int_{0}^{t}\int_{\mathbb{R}^d}\|E_i(s)\|_{\dot{F}^b_{p,\infty}}\langle s\rangle^{-(d-\delta_0)b}\frac{1}{\langle v\rangle^N}dvds
	\\
	&\lesssim_{C}|\alpha|^{b}\int_{0}^{t}\left(\frac{1}{\langle s\rangle^{d-2\delta_0+2}}+\frac{1}{\langle s\rangle^{(d-\delta_0)b+2-\delta_0}}\right)ds\lesssim_{C}|\alpha|^{b}.\label{Tbar2}
\end{align}
\textbf{2.3)} Estimate of  $\overline{\mathcal{T}}^3_\alpha[E_i,\partial_i\mu]$.\\
Firstly, we estimate $\partial^{\beta_1}_x \overline{F}_{2,1}(x-\alpha)$. If $|\beta_1|\geq1$,
\begin{align*}
|\partial^{\beta_1}_x \overline{F}_{2,1}(x-\alpha)|
&\lesssim	\sum_{\substack{1\leq |\beta_3|\leq |\beta_1|,|\vartheta_i|>0\\|\vartheta_1|+...+|\vartheta_{|\beta_3|}|=|\beta_1|}}| \prod_{j=1}^{|\beta_3|}(\partial_x^{\vartheta_j}X_{s,t}(x-\alpha,v))(\partial^{\beta_3}E_i)(s,X_{s,t}(x-\alpha,v))|\\
&\lesssim _{C}	\sum_{1\leq |\beta_3|\leq |\beta_1|}
\left|(\partial^{\beta_3}E_i)(s,X_{s,t}(x-\alpha,v))\right|.
\end{align*}
Combining above with the definition of $\partial^{\beta_1}_x \overline{F}_{2,1}(x-\alpha)$, we get for any $|\beta_1|\geq0$,
\begin{align*}
|\partial^{\beta_1}_x \overline{F}_{2,1}(x-\alpha)|\lesssim_{C} 	\sum_{0\leq |\beta_3|\leq |\beta_1|}
\left|(\partial^{\beta_3}E_i)(s,X_{s,t}(x-\alpha,v))\right|.
\end{align*}
Then, we compute the other term.
For $|\beta_2|\geq1$,
\begin{align*}
|\partial^{\beta_2}_x\overline{\eta}_{3}(x)|&\lesssim \sum_{\substack{1\leq |\beta_3|\leq \beta_2|,|\vartheta_i|>0\\|\vartheta_1|+...+|\vartheta_{|\beta_3|}|=|\beta_2|}}\left| \prod_{j=1}^{|\beta_3|}\partial_x^{\vartheta_j}W_{s,t}(x-vt,v)(\partial^{\beta_3}\partial_i\mu)(V_{s,t}(x,v))\right.\\
&\qquad\qquad\qquad\qquad\qquad\left.-\prod_{j=1}^{|\beta_3|}\partial_x^{\vartheta_j}W_{s,t}(x-\alpha-vt,v)(\partial^{\beta_3}\partial_i\mu)(V_{s,t}(x-\alpha,v))\right|
\\&\leq \sum_{\substack{1\leq |\beta_3|\leq \beta_2|,|\vartheta_i|>0\\|\vartheta_1|+...+|\vartheta_{|\beta_3|}|=|\beta_2|}}\left| \prod_{j=1}^{|\beta_3|}(\partial_x^{\vartheta_j}W_{s,t}(x-vt,v))-\prod_{j=1}^{|\beta_3|}(\partial_x^{\vartheta_j}W_{s,t}(x-\alpha-vt,v))\right||(\nabla^{|\beta_3|}\partial_i\mu)(V_{s,t}(x,v))|\\&\quad+
\prod_{j=1}^{|\beta_3|}|\partial_x^{\vartheta_j}W_{s,t}(x-\alpha-vt,v)|\left|(\nabla^{|\beta_3|}\partial_i\mu)(V_{s,t}(x,v))-(\nabla^{|\beta_3|}\partial_i\mu)(V_{s,t}(x-\alpha,v))\right|
\\&\lesssim_{C}|\alpha|^b\frac{\langle s\rangle^{-d+\delta_0-1}}{\langle v\rangle^N}+\min\{|\alpha|,1\}\left(\frac{1}{\langle v\rangle^N}+\frac{1}{\langle v-\alpha\rangle^N}\right)\langle s\rangle^{-d+\delta_0-1}\\
&\lesssim_{C}|\alpha|^b\left(\frac{1}{\langle v\rangle^N}+\frac{1}{\langle v-\alpha\rangle^N}\right)\langle s\rangle^{-d+\delta_0-1}.
\end{align*}			
If $|\beta|_2=0$, we have
\begin{align*}
|\overline{\eta}_{3}(x)|\lesssim_{C}\frac{|\alpha|^b}{\langle v\rangle^N}\langle s\rangle^{-d+\delta_0-1}.
\end{align*}
Hence, we deduce that
\begin{align}\nonumber
	\left\|\overline{\mathcal{T}}_{\alpha}^3[E_i,\partial_i\mu](t,x)\right\|_{L^p}
	\lesssim_{b,\delta_0,\mathbf{c},\mathbf{c}'}&|\alpha|^b\int_{0}^{t}\int_{\mathbb{R}^d}\sum_{\substack{0\leq|\beta_1|\leq l\\0\leq |\beta_3|\leq |\beta_1|}}
	\left\|(\partial^{\beta_3}E_i)(s)\right\|_{L^p}\left(\frac{1}{\langle v\rangle^N}+\frac{1}{\langle v-\alpha\rangle^N}\right)\langle s\rangle^{-d+\delta_0-1}dvds\\
	&\nonumber+|\alpha|^{b}\int_{0}^{t}\int_{\mathbb{R}^d}\sum_{1\leq |\beta_3|\leq l+1}
	\left\|(\partial^{\beta_3}E_i)(s)\right\|_{L^p}\langle s\rangle^{\delta_0-d-1}\frac{1}{\langle v\rangle^N}dvds\\
	\lesssim_{C}& |\alpha|^{b}\int_{0}^{t}\langle s\rangle^{-d+\delta_0-1}ds\lesssim_{C}|\alpha|^b.\label{Tbar3}
\end{align}
Thus, combining \eqref{Tbar1}, \eqref{Tbar2} with \eqref{Tbar3}, we have for $0\leq t\leq 1$,
\begin{align}\label{T t small}
\sup_\alpha\frac{\|\overline{	\mathcal{T}}_{\alpha}[E_i,\partial_i\mu](t)\|_{L^\infty}}{|\alpha|^{b}}\lesssim_{C}1.
\end{align}
In conclusion, thanks to \eqref{T t large} and \eqref{T t small}, we finish the proof.
\end{proof}
\section{Proof of the Theorem \ref{Thm1} and Theorem \ref{Thm2}}
Recall the operator $\mathcal{F}$ as:
$$\mathcal{F}(\rho,U)=\Big(\mathcal{F}_1(\rho,U),\mathcal{F}_2(\rho,U)\Big),
$$
where
$
\mathcal{F}_1(\rho,U)=G*_{(t,x)}(\mathcal{I}(\rho,U)+\mathcal{R}(\rho,U)+A(U))+\mathcal{I}(\rho,U)+\mathcal{R}(\rho,U)
$ and $ \mathcal{F}_2(\rho,U)=(1-\Delta)^{-1}(\rho+A(U)).$\\
Now we start to prove Theorem \ref{Thm1} and  \ref{Thm2}. \\
\begin{proof}[Proof of Theorem \ref{Thm1}]Let $\widetilde{\varepsilon}_0$ be in Proposition \ref{estimates about Y and W} and take $\varepsilon\leq\widetilde{\varepsilon}_0$. Assume $$\|(\rho,U)\|_{S^{{\varepsilon}}_a},\|(\rho_1,U_1)\|_{S^{{\varepsilon}}_a},\|(\rho_2,U_2)\|_{S^{{\varepsilon}}_a}\leq\varepsilon.$$Applying Theorem \ref{t modulus estimation} with $\mathbf{f}=\mathcal{I}_{f_0}(\rho,U)+\mathcal{R}(\rho,U)$ and $\mathbf{f}=\mathcal{I}_{f_0}(\rho_1,U_1)-\mathcal{I}_{f_0}(\rho_2,U_2)+\mathcal{R}(\rho_1,U_1)-\mathcal{R}(\rho_2,U_2)$, there exists a constant $M=M(\text{\r{c}},d,M^*)$ such that
\begin{align}\label{C2}
&\|\mathcal{F}_1(\rho,U)	\|_{a}\leq  M\Big(\| (\mathcal{I}_{f_0}+\mathcal{R})(\rho,U)\|_{a}+\|A(U)\|_a\Big),\\  &\|\mathcal{F}_1(\rho_1,U_1)-\mathcal{F}_1(\rho_2,U_2)\|_a\leq M\bigg(\|(\mathcal{I}_{f_0}+\mathcal{R})(\rho_1,U_1)-(\mathcal{I}_{f_0}+\mathcal{R})(\rho_2,U_2)\|_a+\|A(U_1)-A(U_2)\|_a\bigg)
.
\end{align}
Thanks to Lemma \ref{varrho},  Proposition \ref{I a estimate} , Proposition \ref{mathcalR sigma est} and \eqref{AU leq U}, we get that there exsists a constant $C_{3,1}=C_{3,1}(a,\text{\r{c}},d,M^*)>0$ and $C_{3,2}=C_{3,2}(a,d)>0$ such that
\begin{align}\label{F1}
\begin{split}
	\|\mathcal{F}_1(\rho,U)	\|_{a}\leq& C_{3,1}\left(	\sum_{p=1,\infty}\left\|\mathcal{D}^{a} (f_0)\right\|_{L^1_{x}L^p_v\cap L^1_{v}L^p_x}+ \|(\rho,U)\|_{S^{{\varepsilon}}_a}^{1+a}+C_A\varepsilon^{\frac{1}{3}}\|(\rho,U)\|_{S^{{\varepsilon}}_a}\right)\\
\leq&C_{3,1}\left(\frac{\varepsilon}{C_2}+\varepsilon^{1+a}+{\varepsilon}\varepsilon^{\frac{1}{3}}\right)\leq\left(\frac{C_{3,1}}{C_2}+C_{3,1}(1+C_A)\varepsilon^{\frac{1}{3}}\right)\varepsilon,
\end{split}
\end{align}
and
\begin{align}\label{F_2}
	\|\mathcal{F}_2(\rho,U)	\|_{a}\leq& C_{3,2} \|\rho +A(U)\|_a\leq C_{3,2} \|(\rho ,U)\|_{S^{{\varepsilon}}_a}\leq C_{3,2}\varepsilon, \\\label{F2 diff}
		\|\mathcal{F}_2(\rho_1,U_1)-\mathcal{F}_2(\rho_2,U_2)\|_a\leq&C_{3,2}\|\rho_1-\rho_2+A(U_1)-A(U_2)\|_a\leq C_{3,2} \|(\rho_1-\rho_2 ,U_1-U_2)\|_{S^{{\varepsilon}}_a}.
\end{align}
Moreover, with Proposition \ref{I 12a estimate} and Proposition \ref{mathcalR difference}, there exists a constant $C_4=C_4(a,\text{\r{c}},d,M^*)>0$ such that
\begin{align}\label{F1 diff}
\begin{split}
&	\|\mathcal{F}_1(\rho_1,U_1)-\mathcal{F}_1(\rho_2,U_2)\|_a\leq C_4\|A(U_1)-A(U_2)\|_a\\&+ C_4\|(\rho_1-\rho_2,U_1-U_2)\|_{S^{{\varepsilon}}_a} \left(	\sum_{p=1,\infty}\|\mathcal{D}^a(\nabla_{x,v} f_0)\|_{L^1_{x}L^p_v\cap L^1_{v}L^p_x}+	
\|(\rho_1,U_1)\|_{S^{{\varepsilon}}_a}^a+\|(\rho_2,U_2)\|_{S^{{\varepsilon}}_a}\right)\\
&\leq\|(\rho_1-\rho_2,U_1-U_2)\|_{S^{{\varepsilon}}_a}\left(\frac{C_4}{C_2}+2C_4\right) {\varepsilon}^{a} +{C_4}C_A\varepsilon^{\frac{1}{3}}\|(\rho_1-\rho_2,U_1-U_2)\|_{S^{{\varepsilon}}_a}\\
&\leq\|(\rho_1-\rho_2,U_1-U_2)\|_{S^{{\varepsilon}}_a}\left(\frac{C_4}{C_2}+(2+C_A)C_4\right) \varepsilon^{\frac{1}{3}}.
\end{split}
\end{align}
Combining \eqref{F1}, \eqref{F1 diff}, \eqref{F2 diff} with  \eqref{F_2}, we obtain that
\begin{align}\nonumber
\|\mathcal{F}(\rho,U)\|_{S^{{\varepsilon}}_a}=&\|\mathcal{F}_1(\rho,U)\|_a+\varepsilon^{\frac{1}{3}}\|\mathcal{F}_2(\rho,U)\|_a\leq \left(\frac{C_{3,1}}{C_2}+(C_{3,1}(1+C_A)+C_{3,2})\varepsilon^{\frac{1}{3}}\right)\varepsilon,\\
	\|\mathcal{F}(\rho_1,U_1)-\mathcal{F}(\rho_2,U_2)\|_{S^{{\varepsilon}}_a}=&	\|\mathcal{F}_1(\rho_1,U_1)-\mathcal{F}_1(\rho_2,U_2)\|_a+\varepsilon^{\frac{1}{3}}	\|\mathcal{F}_2(\rho_1,U_1)-\mathcal{F}_2(\rho_2,U_2)\|_a\nonumber\\\nonumber
	\leq&\|(\rho_1-\rho_2 ,U_1-U_2)\|_a\left(\frac{C_4}{C_2}+(2+C_A)C_4+C_{3,2}\right) \varepsilon^{\frac{1}{3}}\\:=&C_5 \varepsilon^{\frac{1}{3}}\|(\rho_1-\rho_2 ,U_1-U_2)\|_a.\label{diff12}
\end{align}
Hence, we choose $C_2=\max\{1,8C_{3,1},8C_4\}$,  and $\tilde{\varepsilon}_1=\min\left\{\left(\frac{1}{4(C_{3,1}(1+C_A)+C_{3,2})}\right)^3,\left(\frac{1}{100+4((2+C_A)C_4+C_{3,2})}\right)^{3},\tilde{\varepsilon}_0\right\},$ we can yield that
\begin{equation}\label{z111}
\|\mathcal{F}(\rho,U)	\|_{a}\leq   \varepsilon,~~~~~~
\|\mathcal{F}(\rho_1,U_1)-\mathcal{F}_1(\rho_2,U_2)\|_a\leq\frac{1}{2}\|(\rho_1-\rho_2,U_1-U_2)\|_a,~~\forall~\varepsilon\leq \tilde{\varepsilon}_1.
\end{equation}
Therefore, we conclude  that $\mathcal{F}$ is a compressed mapping in $\mathfrak{B}(\varepsilon)$ for $\varepsilon\leq \tilde{\varepsilon}_1$
Now we consider the iterative sequence $(\varrho_n,\mathcal{U}_n)=\mathcal{F}(\varrho_{n-1},\mathcal{U}_{n-1}) $ with $\varrho_0(t,x)=U_0(t,x)=0$. So, $\varrho_1(t,x)=\mathcal{F}_1(0,0)(t,x)=G*_{t,x}\tilde{\varrho}_0(t,x)+\tilde{\varrho}_0(t,x)$ and $\mathcal{U}_1=0$ with $\tilde{\varrho}_0(t,x)=\int_{\mathbb{R}^d}f_0(x-tv,v)dv$. From  \eqref{z111} we obtain that for $n\in\mathbb{N}$,
\begin{equation}\label{guji}
	\|\mathcal{F}(\varrho_{n+1},\mathcal{U}_{n+1})-\mathcal{F}(\varrho_{n},\mathcal{U}_n)\|_a\leq C_5 \varepsilon^{\frac{1}{3}}\|(\varrho_{n}-\varrho_{n-1},\mathcal{U}_n-\mathcal{U}_{n-1})\|_a=C_5 \varepsilon^{\frac{1}{3}}	\|\mathcal{F}(\varrho_{n},\mathcal{U}_{n})-\mathcal{F}(\varrho_{n-1},\mathcal{U}_{n-1})\|_a.
\end{equation} 
Thus $\{(\varrho_n,\mathcal{U}_n)\}_{n=1}^\infty$ is a Cauchy sequence. We denote 	$$(\rho,U)=\Big(\sum_{n=2}^\infty (\varrho_n-\varrho_{n-1})+\varrho_1,\sum_{n=2}^\infty (\mathcal{U}_n-\mathcal{U}_{n-1})+\mathcal{U}_1\Big),$$ it is clear that $\Big(\mathcal{F}_1(\rho,U),\mathcal{F}_2(\rho,U)\Big)=(\rho,U)\in \mathfrak{B}(\varepsilon).$ Moreover by \eqref{guji}, we have
\begin{align*}		&\|(\rho-\varrho_1,U-\mathcal{U}_1)\|_a\leq\sum_{n=2}^\infty \|(\varrho_n-\varrho_{n-1},\mathcal{U}_n-\mathcal{U}_{n-1})\|_a
	\\
&\qquad	=\sum_{n=2}^\infty \|\mathcal{F}(\varrho_{n-1},\mathcal{U}_{n-1})-\mathcal{F}(\varrho_{n-2},\mathcal{U}_{n-2})\|_a\leq \varepsilon\sum_{n=2}^\infty (C_5 \varepsilon^{\frac{1}{3}})^{n-1}\leq 2C_5  \varepsilon^\frac{4}{3}.
\end{align*}
Recall the definition of $\mathcal{F}(\rho)$ in \eqref{mathcal F}, we have that the fixed point is exactly the solution of \eqref{equation rho}.
Now we consider $\|(\rho,U)\|_{1,a}$, note that
\begin{align*}
	\| U\|_{1,a}&=\|\mathcal{F}_2(\rho)\|_{1,a}\leq \|\rho\|_{1,a}+C_A\|U\|_a\| U\|_{1,a}\leq  \|\rho\|_{1,a}+ \varepsilon^{\frac{1}{3}}\| U\|_{1,a} ,
	\\\|\rho\|_{1,a}&=\|\mathcal{F}_1(\rho,U)\|_{1,a}=
	\Big\| G*_{(t,x)}\Big(\mathcal{I}_{f_0}(\rho,U)+\mathcal{R}(\rho,U)+ A(U)\Big)+\Big(\mathcal{I}_{f_0}(\rho,U)+\mathcal{R}(\rho,U)\Big)\Big\|_{1,a}\\
	&\leq M\left(\|\mathcal{I}_{f_0}(\rho,U)\|_{1,a}+\|\mathcal{R}(\rho,U)\|_{1,a}+\| A(U)\|_{1,a}\right)\\
	&\leq M\left(\|\mathcal{I}_{f_0}(\rho,U)\|_{1,a}+\|\mathcal{R}(\rho,U)\|_{1,a}+\varepsilon^{\frac{1}{3}}\|U\|_{1,a}\right)
	.
\end{align*}
Therefore we have 
\begin{align}\label{tr11}
	\|\rho\|_{1,a}+\frac{1}{4}\|U\|_{1,a}\leq M\left(\|\mathcal{I}_{f_0}(\rho,U)\|_{1,a}+\|\mathcal{R}(\rho,U)\|_{1,a}\right)+M\varepsilon^{\frac{1}{3}}\| U\|_{1,a}+\frac{1}{4}\|\rho\|_{1,a}+\frac{1}{4}\varepsilon^{\frac{1}{3}}\|U\|_{1,a},
\end{align}
which implies that
\begin{align*}
	\|\rho\|_{1,a}+\| U\|_{1,a}\leq 8M\left(\|\mathcal{I}_{f_0}(\rho,U)\|_{1,a}+\|\mathcal{R}(\rho,U)\|_{1,a}\right).
\end{align*}
Thanks to \eqref{I norm} and \eqref{z13}, we get\begin{align*}
	\|\mathcal{I}_{f_0}(\rho,U)\|_{1-\delta_1}+\|\mathcal{R}(\rho,U)\|_{1-\delta_1}\lesssim_{\text{\r{c}},a,M^*} 1,
\end{align*}
where $\delta_1=\min\{\frac{4-\sqrt{13}}{9},\frac{1-a}{3}\}$.
Then we use Proposition \ref{higherRnorm} and \eqref{higher I norm} to obtain 
\begin{align*}
	\|\mathcal{I}_{f_0}(\rho,U)\|_{1,a}+\|\mathcal{R}(\rho,U)\|_{1,a}\lesssim_{\text{\r{c}},a,M^*} 1.
\end{align*}
Consequently, it gives that 
$$
\|(\rho,U)\|_{1,a}\lesssim_{\text{\r{c}},a,M^*} 1.
$$
The proof is complete. 
\end{proof}\\
\begin{proof}[Proof of Theorem \ref{Thm2}]From Theorem \ref{Thm1}, we know that $\|(\rho ,U)\|_{S^{{\varepsilon}}_a}\leq\widetilde{\varepsilon}_0$ and $\|(\rho,U)\|_{1,a}\lesssim_{\text{\r{c}},a,M^*} 1$. Applying the Gagliardo-Nirenberg interpolation inequality, we get $\|(\rho,U)\|_{1-\delta_0}\lesssim_{a,M^*} 1$, where $\delta_0\in(0,\min\{\frac{4-\sqrt{13}}{3^{m+1}},\frac{1-b}{3^m}\})$.
Combining  Theorem \ref{t modulus estimation} and Proposition \ref{I 12norm} with Proposition \ref{R estimate}, one has
\begin{align*}
	&\|\rho\|_{1,1-3\delta_0}=	\|\mathcal{F}_1(\rho,U)\|_{1,1-3\delta_0}\leq M \left(\| \mathcal{I}_{f_0}(\rho,U)\|_{1,1-3\delta_0} +\| \mathcal{R}(\rho,U)\|_{1,1-3\delta_0}+\| A(U)\|_{1,1-3\delta_0}\right),\\
	&\| U\|_{1,1-3\delta_0}=\|\mathcal{F}_2(\rho,U)\|_{1,1-3\delta_0}\leq \|\rho\|_{1,a}+\|A(U)\|_{1,1-3\delta_0}\leq\|\rho\|_{1,a}+ \varepsilon^{\frac{1}{3}}\| U\|_{1,a},
	\\
	&\| \mathcal{I}_{f_0}(\rho)\|_{1,1-3\delta_0} +\| \mathcal{R}(\rho)\|_{1,1-3\delta_0}\lesssim_{a,M^*,\mathbf{c},\mathbf{c}'}1.
\end{align*}
Using the same method in \eqref{tr11} we have
$
\|(\rho,U)\|_{1,1-3\delta_0}=\|(\mathcal{F}_1(\rho,U),\mathcal{F}_2(\rho,U))\|_{1,1-3\delta_0}\lesssim_{a,M^*,\mathbf{c},\mathbf{c}',C_{A,m}}1.
$
Note that\begin{align*}
	\|\langle s\rangle^2\nabla^2(A(U))\|_{\dot{B}^{1-9\delta_0}_{p,\infty}}\leq C_{A,m} \|U\|_{1,1-9\delta_0}+C_A\varepsilon^{\frac{2}{3}}	\|\langle s\rangle^2\nabla^2U\|_{\dot{B}^{1-9\delta_0}_{p,\infty}}
\end{align*}
We do one more step to obtain
\begin{align*}
	\|(\rho,U)\|_{2,1-9\delta_0}&=	\|(\mathcal{F}_1(\rho,U),\mathcal{F}(\rho,U))\|_{2,1-9\delta_0}\\
	&\lesssim_{\text{\r{c}},a,M^*,\mathbf{c},\mathbf{c}',C_{A,m}} \| \mathcal{I}_{f_0}(\rho,U)\|_{2,1-9\delta_0} +\|\mathcal{R}(\rho,U)\|_{2,1-9\delta_0}+\|(\rho,U)\|_{2,1-9\delta_0}\\
	&
	\lesssim_{\text{\r{c}},a,M^*,\mathbf{c},\mathbf{c}',C_{A,m}}1.
\end{align*}
Then, with $m$ times iteration, we have
$$ \|(\rho,U)\|_{m,1-3^{m}\delta_0}\lesssim_{\text{\r{c}},a,M^*,\mathbf{c},\mathbf{c}',C_{A,m}}1. $$
Note that $1-3^m\delta_0>b$, thus by the Gagliardo-Nirenberg interpolation inequality again to get $$\|(\rho,U)\|_{m,b}\lesssim_{\text{\r{c}},a,M^*,\mathbf{c},\mathbf{c}',C_{A,m}}1, $$
which completes the proof.			
\end{proof}\\
\section{Appendix}
\begin{lemma}\label{mathcal H1H2}
Assume $\mathcal{H}\in L^1(\mathbb{R}^d)\cap L^\infty(\mathbb{R}^d)$ and  $\varphi(x,v)$ satisfy $\|\nabla_{x,v} \varphi(x,v)\|_{L^\infty_{x,v}}\leq\frac{1}{2}$, then for $p=1,\infty$, $0\leq s\leq t$ and $t\geq0$ we have
\begin{align}\label{mathcal H p}
\left\|\int_{\mathbb{R}^d}\mathcal{H}\left(\varphi(x,v)+x-(t-s)v\right)\frac{dv}{\langle v\rangle^N}\right\|_{L^p_x}\lesssim\|\mathcal{H}\|_{L^p}.
\end{align}
Moreover, for $0\leq s\leq \frac{t}{2}$ and $t\geq1$, we also get
\begin{align}\label{mathcal H 1}
\left\|\int_{\mathbb{R}^d}\mathcal{H}\left(\varphi(x,v)+x-(t-s)v\right)\frac{dv}{\langle v\rangle^N}\right\|_{L^p_x}\lesssim\frac{1}{t^{\frac{d(p-1)}{p}}}\|\mathcal{H}\|_{L^1}.
\end{align}
\end{lemma}
\begin{proof}
We can directly prove the first inequality  as follows
\begin{align*}
\left\|\int_{\mathbb{R}^d}\mathcal{H}\left(\varphi(x,v)+x-(t-s)v\right)\frac{dv}{\langle v\rangle^N}\right\|_{L^p_x}\leq\|\mathcal{H}\|_{L^p}\left(1+\|\nabla_{x,v} \varphi(x,v)\|_{L^\infty_{x,v}}\right)
\lesssim\|\mathcal{H}\|_{L^p}.
\end{align*}
Furthermore, we change variable $\zeta=x-(t-s)v$ and obtain for $0\leq s\leq \frac{t}{2}$,
\begin{align*}
&\left\|\int_{\mathbb{R}^d}\mathcal{H}\left(\varphi(x,v)+x-(t-s)v\right)\frac{dv}{\langle v\rangle^N}\right\|_{L^\infty_x}=\frac{1}{(t-s)^d}\left\|\int_{\mathbb{R}^d}\mathcal{H}\left(\varphi(x,\frac{x-\zeta}{t-s})+\zeta\right)\frac{dw}{\langle \frac{x-\zeta}{t-s}\rangle^N}\right\|_{L^\infty_x}
\\&\qquad\lesssim\frac{1}{t^d}\bigg[\sup_{x,\zeta}\left(\left|(\nabla_v\varphi)(x,\frac{x-\zeta}{t-s})\right|(t-s)^{-d}+1\right)^{-1}\bigg]\|\mathcal{H}\|_{L^1}\lesssim\frac{1}{t^d}\|\mathcal{H}\|_{L^1}.
\end{align*}
\end{proof}
\begin{remark}\label{mathcal Halpha}
Similarly, under the same assumptions of $\mathcal{H}$ and $\varphi$, one has
\begin{align}\label{mathcal H alpha p}
\left\|\int_{\mathbb{R}^d}\left|\mathcal{H}\left(\varphi(x,v)+x-(t-s)v\right)\right|\left(\frac{1}{\langle v\rangle^N}+\frac{1}{\langle v-\frac{\alpha}{t}\rangle^N}\right)dv\right\|_{L^p_x}\lesssim\|\mathcal{H}\|_{L^p},
\end{align}
and for $0\leq s\leq \frac{t}{2},~t\geq1$, 
\begin{align}\label{mathcal H alpha 1}
\left\|\int_{\mathbb{R}^d}\left|\mathcal{H}\left(\varphi(x,v)+x-(t-s)v\right)\right|\left(\frac{1}{\langle v\rangle^N}+\frac{1}{\langle v-\frac{\alpha}{t}\rangle^N}\right)dv\right\|_{L^p_x}\lesssim\frac{1}{t^{\frac{d(p-1)}{p}}}\|\mathcal{H}\|_{L^1}.
\end{align}
\end{remark}

\end{document}